%% file: sn-article.tex
\documentclass[sn-mathphys,Numbered]{sn-jnl}

\usepackage{graphicx}%
\usepackage{multirow}%
\usepackage{amsmath,amssymb,amsfonts}%
\usepackage{amsthm}%
\usepackage{mathrsfs}%
\usepackage[title]{appendix}%
\usepackage{xcolor}%
\usepackage{textcomp}%
\usepackage{manyfoot}%
\usepackage{booktabs}%
\usepackage{algorithm}%
\usepackage{algorithmicx}%
\usepackage{algpseudocode}%
\usepackage{listings}%

\theoremstyle{thmstyleone}%
\newtheorem{theorem}{Theorem}
\theoremstyle{thmstyletwo}%
\newtheorem{remark}{Remark}%

\theoremstyle{thmstylethree}%
\newtheorem{definition}{Definition}%

\input{michael-preamble}

\input{def}

\begin{document}

\newcommand{\ArxivToggle}[1]{}

\title[Accelerated First-Order Optimization under Nonlinear Constraints]{Accelerated First-Order Optimization under Nonlinear Constraints}


\author*[1]{\fnm{Michael} \sur{Muehlebach}}\email{michaelm@tuebingen.mpg.de}

\author[2]{\fnm{Michael I.} \sur{Jordan}}\email{jordan@berkeley.edu}

\affil*[1]{\orgdiv{Learning and Dynamical Systems}, \orgname{Max Planck Institute for Intelligent Systems}, \orgaddress{\street{Max-Planck-Ring 4}, \city{Tuebingen}, \postcode{72076}, \state{Baden-Wuerttemberg}, \country{Germany}}}

\affil[2]{\orgdiv{Department of Electrical Engineering and Computer Science}, \orgname{University of California, Berkeley}, \orgaddress{\street{387 Soda Hall}, \city{Berkeley}, \postcode{94720}, \state{California}, \country{USA}}}


\abstract{We exploit analogies between first-order algorithms for constrained optimization and non-smooth dynamical systems to design a new class of accelerated first-order algorithms for constrained optimization. Unlike Frank-Wolfe or projected gradients, these algorithms avoid optimization over the entire feasible set at each iteration. We prove convergence to stationary points even in a nonconvex setting and we derive accelerated rates for the convex setting both in continuous time, as well as in discrete time. An important property of these algorithms is that constraints are expressed in terms of velocities instead of positions, which naturally leads to sparse, local and convex approximations of the feasible set (even if the feasible set is nonconvex). Thus, the complexity tends to grow mildly in the number of decision variables and in the number of constraints, which makes the algorithms suitable for machine learning applications. We apply our algorithms to a compressed sensing and a sparse regression problem, showing that we can treat nonconvex $\ell^p$ constraints ($p<1$) efficiently, while recovering state-of-the-art performance for $p=1$.}




\maketitle

\input{introduction}

\input{review}
\input{secondOrder}

\input{convergenceAnalysis}
\input{example}
\input{conclusion}

\backmatter

%
%
%

\bmhead{Acknowledgments}
Michael Muehlebach thanks the German Research Foundation and the Branco Weiss Fellowship, administered by ETH Zurich, for the support.
This work was also funded by the European Union (ERC-2022-SYG-OCEAN-101071601). Views and opinions expressed are however those of the author(s) only and do not necessarily reflect those of the European Union or the European Research Council
Executive Agency. Neither the European Union nor the granting authority can be held responsible for them.

\ArxivToggle{
\section*{Declarations}


\begin{itemize}
\item Funding: Michael Muehlebach is supported by the German Research Foundation (Emmy Noether grant) and the Branco Weiss Fellowship, administered by ETH Zurich. Michael Jordan is supported by the European Research Council in the form of an ERC Synergy Grant.
\item Conflict of interest: There are no conflicts of interests.
\item Ethics approval: not applicable
\item Consent to participate: not applicable
\item Consent for publication: not applicable
\item Availability of data and materials: not applicable
\item Code availability: not applicable
\item Authors' contributions: The research was mainly developed by Michael Muehlebach with inputs from Michael Jordan. Michael Jordan edited and improved the final manuscript.
\end{itemize}
}

\newpage
\begin{appendices}

\input{appendixS2}

\input{AppendixProof}

\input{appendixProofADT}




\end{appendices}


\bibliography{sn-bibliography}

\end{document}

%% file: michael-preamble.tex
\usepackage[english,american]{babel}

\usepackage{url}
\usepackage{epsfig} 
\usepackage{amsthm}

\usepackage{empheq}
\newcommand*\widefbox[1]{\fbox{\hspace{.5em}#1\hspace{.5em}}}

\graphicspath{{media/}}
\usepackage{tikz}
\usetikzlibrary{shapes,arrows}
\usetikzlibrary{positioning}

\usepackage{epstopdf}
\usepackage{units}
\usepackage{pgfplots}
\usepackage{float}

\usepackage{pstool}

\graphicspath{{media/}}
\newcommand{\executeiffilenewer}[3]{%
 \ifnum\pdfstrcmp{\pdffilemoddate{#1}}%
 {\pdffilemoddate{#2}}>0%
 {\immediate\write18{#3}}\fi%
}

\newtheorem{lemma}[theorem]{Lemma}
\newtheorem{corollary}[theorem]{Corollary}

\pgfplotsset{every tick label/.append style={font=\footnotesize}}
\pgfdeclarelayer{background}
\pgfsetlayers{background,main}



%% file: def.tex
\newcommand{\argmin}{\mathop{\mathrm{argmin}}}

\newcommand{\TT}{\ensuremath{\mathsf{\tiny{T}}}}
\newcommand{\T}{^{\TT}}

\newcommand{\diff}[1][]{\mathrm{d}#1}
\newcommand{\dt}{\diff t }

\newcommand{\hchange}[1]{#1}

%% file: introduction.tex
\section{Introduction}
Our work is concerned with developing first-order algorithms for nonlinear constrained optimization problems of the following form:
\begin{equation}
\min_{x\in \mathbb{R}^n} f(x), \quad \text{s.t.} \quad g(x)\geq 0, \label{eq:fundProb}
\end{equation}
where the function $f:\mathbb{R}^n \rightarrow \mathbb{R}$ defines the objective, the function $g:\mathbb{R}^n \rightarrow \mathbb{R}^{n_\text{g}}$ the constraints, and where $n$ and $n_\text{g}$ are positive integers. We assume that $f$ and $g$ are continuous, $f$ is coercive, and that the feasible set \hchange{$\{x\in \mathbb{R}^n~|~g(x)\geq 0\}$} is compact. The applications we have in mind include problems in statistics, machine learning, and control theory, where $n$ and $n_\text{g}$ are typically on the order of $10^6$.

We develop a new class of accelerated first-order algorithms for computing stationary points of \eqref{eq:fundProb}. These algorithms have four distinctive features: (i) they rely on local linear approximations of the feasible set, thereby avoiding projections or optimizations over the entire feasible set at every iteration, (ii) they are conceptually easy to understand and easy to implement, (iii) their iteration complexity for convex problems is dimension-independent,\footnote{This contrasts with interior-point methods, for example, which require $\mathcal{O}(\sqrt{n_\text{g}})$ Newton iterations to decrease the value of the objective by a constant factor.} and
 (iv) in many important cases (even nonconvex ones) their per-iteration complexity scales roughly linearly in the number of decision variables and the number of constraints.

An important aspect of our work is to lift position constraints to a velocity level, which naturally results in a local linear approximation of the feasible set. These approximations come in two variants, each providing a different trade-off between the per-iteration complexity and the resulting convergence guarantees. More precisely, in our algorithms the forward increments, $(x_{k+1}-x_k)/T$, where $T>0$ is the step size, will be constrained to the set
\begin{equation}
V_{\alpha}(x):=\{v\in \mathbb{R}^n~|~\nabla g_i(x)\T v + \alpha g_i(x) \geq 0, ~\forall i\in I \}, \label{eq:defVa}
\end{equation}
where $I$ either has the form $I=[n_\text{g} ]:=\{1,2,\dots,n_\text{g}\}$ or the form \hchange{$I_x:=\{~i\in [n_\text{g}]~|~g_i(x)\leq 0\}$}. The former version includes every constraint at every iteration, while the latter version includes only constraints that are violated at the current iterate $x$. \hchange{The analysis of algorithms is significantly more challenging when $I=I_x$ compared to $I=[n_\text{g}]$. In the following we will consider both variants, however, non-asymptotic linear rates in discrete time will only be derived for $I=[n_\text{g}]$. We conjecture that the same rates can be achieved asymptotically for $I=I_x$, which has been shown in the non-accelerated situation in earlier work, see \cite{ownWorkC}.} \hchange{Our convergence results span both continuous-time and discrete-time models and are summarized in Table~\ref{Tab:summary} (see the corresponding theorems for the precise statements).}

\begin{table}\label{Tab:summary}
\begin{tabular}{lllllll}
        \toprule
setting & rate & version & objective & constraints & $L_l$ & result \\\hline\\[-5pt]
        \multirow{4}{*}{\shortstack{discrete\\ time}} & $\mathcal{O}(\sqrt{\kappa_l} \log(1/\varepsilon))$ & $I=[n_\text{g}]$ & smt/str cvx & smt cve & kwn & Thm.~\ref{Thm:ADT}\\[5pt]
& $\tilde{\mathcal{O}}(1/\sqrt{\varepsilon})$ & $I=[n_\text{g}]$ & smt/str cvx & smt cve & - & Cor.~\ref{Cor:SAR}\\[5pt]
& $\tilde{\mathcal{O}}(1/\sqrt{\varepsilon})$ & $I=[n_\text{g}]$ & smt cvx & smt cve & kwn & Cor.~\ref{Cor:SCA}\\[5pt]
& conv. to stat. & $I=I_x$ & smt & smt & - & Thm.~\ref{Thm:ConvDT}\\[5pt]\hline\\[-5pt]
 \multirow{3}{*}{\shortstack{continuous\\ time}} & $\mathcal{O}(\sqrt{\kappa}\log(1/\varepsilon))$ & $I=I_x$ & smt/str cvx & smt cve & - & Thm.~\ref{Thm:ConvRate}\\[5pt]
& $\mathcal{O}(1/\sqrt{\varepsilon})$ & $I=I_x$ & smt cvx & smt cve & - & Thm.~\ref{Thm:ConvRate}\\[5pt]
& conv. to stat. & $I=I_x$ & smt & smt & - & Thm.~\ref{Thm:stability}
\end{tabular}
\caption{\hchange{The table summarizes the convergence results presented in the main text, where $\mu$ and $\kappa$ denote the strong convexity constant and condition number of $f$, $L_l$ is the smoothness constant of the Lagrangian $f(x)-{\lambda^*}\T g(x)$ with $\lambda^*$ an optimal multiplier of \eqref{eq:fundProb}, and $\kappa_l=L_l/\mu$. Moreover, smt stands for smooth, str for strongly, cvx for convex, cve for concave, and conv. to stat. for convergence to stationary points. The column under $L_l$ describes whether a bound on the smoothness constant $L_l$ needs to be known (knw) and the tolerance is denoted by $\varepsilon$. In the special case where $g$ is linear, $\kappa_l$ reduces to $\kappa$, and our algorithm recovers the rate of accelerated projected gradient descent, without requiring projections or optimizations over the entire feasible set.}}
\end{table}


\hchange{Our treatment builds on recent progress in using tools from continuous-time dynamical systems to analyze discrete-time algorithms in gradient-based optimization~\citep{SuAcc, WibisonoVariational, Diakonikolas2, KricheneAcc, Gui, Betancourt, ourWork, ourWork3, ourWork2, Attouch, Attouch1, Attouch2,Dilsad,Guanchun}. Much of this work aims at understanding accelerated first-order optimization methods, such as Nesterov's algorithm, by exposing links between differential and symplectic geometry, dynamical systems, and mechanics.} While in the absence of constraints these analogies result in \emph{smooth} dynamical systems, the current article presents analogies between constrained optimization and \emph{non-smooth} dynamical systems. Indeed, one of the closest point of contacts with existing literature is the notion of Moreau time-stepping in non-smooth mechanics \citep{moreau}. The important feature of Moreau time-stepping, which also lies at the heart of our work, is that smooth and non-smooth motion are treated on equal footing, which is achieved by discretizing a certain kind of differential inclusion~\citep[see, e.g.,][]{Glocker,Studer}.

Our approach can also be interpreted through the lens of the projected gradient methodology and indeed it has certain similarities to inexact projected gradient methods, as proposed by \citet{InexactPG} and \citet{InexactPG2}. While projected gradient approaches have been successfully applied in various machine learning problems~\citep[see, e.g.,][]{SignalProcessing,SVMwPG}, the Frank-Wolfe algorithm has also received considerable attention in recent years~\citep{Jaeggi}. The appeal of Frank-Wolfe is further increased by the fact that it provides a unified framework for many first-order machine learning algorithms in constrained settings, including support vector machines, online estimation of mixtures of probability densities, and boosting \citep{Clarkson}. Recent results extend the Frank-Wolfe algorithm to the stochastic setting \citep{PfOnlineLearning,OneSampleStochasticFW}, or improve on its relatively slow convergence rate \citep{FasterRates,BoostingFrankWolfe}. 

In some cases constraints can be handled very efficiently with mirror descent, \citep[Ch.~3]{ProblemComplexity}, where a non-Euclidean metric is introduced that adapts gradient descent to the specific type of objective function or the specific type of constraints at hand \citep{BeckMirror}. Although mirror descent is based on projections onto the feasible set, the non-Euclidean metric can improve on problem-specific constants. An important example is the optimization of linear functions over the unit simplex, which has applications in online machine learning \citep{BubeckBandits}. \hchange{Another important class of methods arises from the alternating direction of multipliers, which can also be formulated with inertial dynamics, as recently shown by \citet{Attouch3} and \citet{Attouch4}. In fact, \cite{Attouch3, Attouch4} proposes analogies to smooth dynamical systems for understanding the convergence properties and asymptotic behavior of algorithms. Our work follows a similar guiding principle, focusing on non-smooth dynamics that arise from set-valued operators.}

Compared to projected gradients, mirror descent, and the Frank-Wolfe algorithm, our approach avoids optimizing over the entire feasible set at each iteration and instead relies on sparse, local and convex approximations. This article focuses on accelerated gradient descent, building on the recent results of \citet{ownWorkC}, which treats gradient descent. Including constraints in momentum-based algorithms is challenging: The presence of constraints requires a need for sudden and large changes in momentum (impacts) in order to avoid infeasible iterates. This requires us to not only characterize the smooth motion (if constraints are absent or the solution slides along the boundary of the feasible set), but also the non-smooth motion (if the solution suddenly hits the boundary of the feasible set).

Specific problems which have the potential to benefit from our approach include planning problems in reinforcement learning and/or optimal control \cite[see, e.g.,][]{Pavel}, optimizations over nonconvex matrix manifolds (such as the set of orthogonal matrices \cite[see, e.g.,][]{COLT}), distance geometry problems that arise in computational chemistry/NMR spectroscopy~\cite{Liberti}, optimal transport problems \citep{Abdul}, or supervised learning tasks that involve nonlinear constraints (for example in an imitation learning framework \cite{Wenshuai}, where nonlinear constraints arise from stability requirements on the closed-loop system). We will also demonstrate our approach on $\ell^p$-regularized inverse problems that arise in compressed sensing and signal processing, where we are not only able to obtain state-of-the-art results for $p=1$, but can also seamlessly handle the regime $0<p<1$.

The article is structured in the following way: Sec.~\ref{Sec:GF} summarizes earlier work of \citet{ownWorkC}, which covers gradient descent and sets the stage for discussing momentum-based algorithms in Sec.~\ref{Sec:AGF}. A variety of convergence results that capture both discrete-time and continuous-time models are presented in Sec.~\ref{Sec:ConvAnal}; in particular, in the nonconvex regime we establish convergence to stationary points and we derive accelerated rates in the convex regime. Sec.~\ref{Sec:NumEx} presents numerical experiments, which include nonconvex sparse regression and compressed sensing problems. The paper concludes with a short discussion in Sec.~\ref{Sec:Conc}.

%% file: review.tex
\section{\hchange{Velocity Constraints}}\label{Sec:GF}
The fundamental idea of this work is to express constraints in terms of the forward increment or velocity of our algorithms instead of constraining the iterates or positions directly. As we will see shortly, this naturally leads to local, sparse and convex approximations of the feasible set. Our treatment builds upon \citet{ownWorkC}, which focused on gradient descent and gradient flow, whereas this work focuses on accelerated gradient algorithms. \hchange{The section briefly summarizes the results from \cite{ownWorkC} in order to set the stage for deriving accelerated algorithms in Sec.~\ref{Sec:AGF} and quantify their convergence rates in Sec.~\ref{Sec:ConvAnal}.} 

We model an optimization algorithm as a continuous-time or discrete-time dynamical system, whose equilibria correspond to the stationary points of \eqref{eq:fundProb}. In continuous time, the configuration of the system will be denoted by a function $x: [0,\infty) \rightarrow \mathbb{R}^n$, which is assumed to be absolutely continuous. 
A fundamental observation, lying at the heart of the current research, is that the constraint $x(t) \in C$, for all $t\geq 0$, is equivalent to the constraint $\dot{x}(t)^+ \in T_C(x(t))$, for all $t\geq 0$, $x(0) \in C$, where $T_C(x(t))$ denotes the tangent cone (in the sense of Clarke) of the set $C$ at $x(t) \in \mathbb{R}^n$, and $\dot{x}(t)^+$ denotes the forward velocity: $\dot{x}(t)^+:= \lim_{\dt\downarrow 0} (x(t+\dt)-x(t))/\dt$. The tangent cone $T_C(x)$ is defined as the set of all vectors $v$ such that $(x_k-x)/t_k \rightarrow v$ for two sequences $x_k\in C$ and $t_k\geq 0$ with $x_k\rightarrow x$, $t_k \rightarrow 0$. Provided that a constraint qualification holds (for example Mangasarian-Fromovitz or Abadie constraint qualification), the tangent cone can be expressed as
\begin{equation*}
T_C(x)=\{ v\in \mathbb{R}^n ~|~ \nabla g_i(x)\T v \geq 0, ~~\forall i \in I_x \},
\end{equation*}
where \hchange{$I_x=\{i\in [n_\text{g}]~|~g_i(x)\leq 0\}$} denotes the set of active inequality constraints at $x$.

We therefore conclude that the constraint $x(t)\in C$, which constrains the position $x$, is equivalent to a constraint on the forward velocity $\dot{x}^+$. We note that the velocity $\dot{x}$ is allowed to be discontinuous and need not exist for every $t\geq 0$.\footnote{We assume that $\dot{x}$ is of locally bounded variation, which means that on any compact interval $\dot{x}$ has countably many discontinuity points, where left and right limits exist.} For example, if the trajectory $x$ reaches the boundary of the feasible set, an instantaneous jump of the velocity might be required to ensure that $x$ remains in $C$.

In discrete time, however, this equivalence between position and velocity constraints no longer holds, since $T_C(x)$ is only a first-order approximation of the feasible set. 
Thus, implementing $(x_{k+1}-x_k)/T \in T_C(x_k)$ may lead to infeasible iterates. \citet{ownWorkC} therefore suggest to introduce the velocity constraint $V_\alpha(x)$, see \eqref{eq:defVa} with $I=I_x$, which includes the restitution coefficient $\alpha > 0$.
The following remarks motivate \eqref{eq:defVa} \hchange{and are important for understanding the accelerated algorithms presented subsequently}:
\begin{itemize}
\item[i)] For $x\in C$, the set $V_\alpha(x)$ reduces to the tangent cone $T_C(x)$ (assuming constraint qualification).
\item[ii)] For a fixed $x\in \mathbb{R}^n$, $V_\alpha(x)$ is a convex polyhedral set involving only the active constraints $I_x$. The set $V_\alpha(x)$ therefore amounts to a sparse and linear approximation of the feasible set $C$, even if $C$ is nonconvex.
\item[iii)] In continuous time, the constraint $\dot{x}(t)^+\in V_\alpha(x(t))$ for all $t\geq 0$ implies 
\begin{align}
g_i(x(t)) &\geq \min\{g_i(x(0)) e^{-\alpha t}, 0\}, \label{eq:groenwall}
\end{align}
for all $t\geq 0$ and all $i\in\{1,\dots,n_\text{g}\}$, which can be verified with Gr\"{o}nwall's inequality. This means that potential constraint violations decrease at rate $\alpha$.
\end{itemize}

\begin{remark}\label{rem:nonempty} Nonemptiness of $V_\alpha(x)$:
If $C$ is convex, $V_\alpha(x)$ is guaranteed to be nonempty for all $x\in \mathbb{R}^n$. If $C$ is nonconvex, nonemptiness of $V_\alpha(x)$ for all $x$ in a neighborhood of $C$ is guaranteed if the Mangasarian-Fromovitz constraint qualification holds for all $x\in C$. We note that the Mangasarian-Fromovitz constraint qualification is  generic in the following sense: Provided that $g$ is semi-algebraic (these cases include all the usual functions used in optimization) there exists $\epsilon_0\in \mathbb{R}^m, \epsilon_0> 0$, such that the set $C_\epsilon:=\{x\in \mathbb{R}^n~|~g(x)\geq -\epsilon\}$ satisfies the Mangasarian-Fromovitz constraint qualification for all $x\in C_\epsilon$ and for all $\epsilon\in (0,\epsilon_0)$; see \citet{Bolte}.
\end{remark}

\hchange{The continuous-time gradient flow dynamics that were studied in \citet{ownWorkC} arise from the following equation:
\begin{equation}
\dot{x}(t)^+ +\nabla f(x(t))=R(t), \quad -R(t)\in N_{V_\alpha (x(t))}(\dot{x}(t)^+),  \label{eq:form1}
\end{equation}
for all $t\geq 0$, where $N_{V_\alpha(x(t))}(\dot{x}(t)^+)$ denotes the normal cone of the set $V_\alpha(x(t))$ at $\dot{x}(t)^+$. Thus, the variable $R(t)$ can be regarded as a constraint force that imposes the constraint $\dot{x}(t)^+\in V_\alpha(x(t))$.}
\hchange{In discrete time, it suffices to replace $\dot{x}(t)^+$ by $(x_{k+1}-x_k)/T$ and $x(t)$ by $x_k$ in order to obtain the corresponding constrained gradient-descent dynamics. These can be expressed as (see \cite{ownWorkC} for details)
\begin{equation}
    x_{k+1} = x_k + T \argmin_{v\in V_\alpha(x_k)} |v+\nabla f(x_k)|^2, \label{eq:GDdt}
\end{equation}
where $T>0$ is the step size and can be interpreted as a modified projected gradient scheme, where projections over the entire feasible set $C$ are replaced with optimizations over the sparse and convex approximation $V_\alpha(x_k)$.} 

\hchange{The results from \citet{ownWorkC} establish convergence of \eqref{eq:form1} and \eqref{eq:GDdt}. In continuous time, it was shown that even when $f$ and $C$ are nonconvex, the trajectories of \eqref{eq:form1} converge to the set of stationary points. Moreover, if $f$ is strongly convex with strong convexity constant $\mu$ and $\alpha$ is set to $2\mu$, the trajectories converge from any initial condition to the minimizer of \eqref{eq:fundProb} at a linear rate, which scales with $1/\kappa$, where $\kappa$ is the condition number of $f$. A similar, albeit asymptotic, convergence rate was found in discrete time for an adequate choice of step size. However, a non-asymptotic linear convergence result was missing in discrete time; the asymptotic result required a relatively complex analysis of the algorithm's dynamics. For completeness, we therefore provide a non-asymptotic convergence result in App.~\ref{App:Sec1} (see Thm.~\ref{Thm:GDV}).}

\hchange{The results from Thm.~\ref{Thm:GDV} can be compared to a type of composite optimization \cite[see, e.g., ][and others]{BertsekasComp, PaquetteComp, FletcherComp, BolteComp, NesterovComp, NicolasComp}, which can also be applied to problems of the type \eqref{eq:fundProb} resulting in updates similar to \eqref{eq:GDdt}. The resulting convergence rates known in the literature for the convex setting are similar to Thm.~\ref{Thm:GDV}, see, e.g., \cite[Thm.~2]{BolteComp}, \cite[Thm.~5]{NesterovComp}, \cite{NicolasComp},\footnote{\hchange{In \cite{NicolasComp} there is a distinction between number of calls to a linear minimization oracle and number of gradient evaluations. Compared to Thm.~\ref{Thm:GDV}, \cite{NicolasComp} requires a similar number of linear minimization calls, but many fewer gradient evaluations (\cite{NicolasComp} is accelerated in terms of gradient evaluations but not in terms of linear minimization steps).}} with the important difference that \cite{BolteComp, NesterovComp,NicolasComp} do not require knowledge of the constant $L_l$. However, the current article highlights that much faster and accelerated rates can be achieved in both the smooth and strongly convex case, as well as in the smooth convex case by adding momentum.}

\hchange{To summarize, \eqref{eq:form1} and \eqref{eq:GDdt} implement gradient-flow and gradient-descent dynamics that can handle constraints and converge linearly with the typical $1/\kappa$ (continuous time) and $1/\kappa_l$-rate (discrete time) if the objective function $f$ is smooth and strongly convex, where $\kappa_l=L_l/\mu$ and $L_l$ relates to the smoothness constant of the Lagrangian.\footnote{\hchange{We recall that the Lagrangian is given by $f(x)-{\lambda^*}\T g(x)$ with $\lambda^*$ an optimal multiplier of \eqref{eq:fundProb}. Hence, the smoothness constant $L_l$ is bounded by $L_f+\sum_i^{n_g} \lambda_i^* L_g^i$, where $L_g^i$ refers to the smoothness constant of $g_i$ and $L_f$ to the smoothness constant of $f$.}} The set $V_\alpha(x)$ can be seen as a velocity constraint and provides a natural generalization of the tangent cone. It also reduces the computational cost for each iteration, since projections on the entire feasible set are avoided. In the next section, we generalize these ideas to algorithms that have momentum. This will naturally lead to accelerated algorithms that converge linearly at a rate of $1/\sqrt{\kappa_l}$ (if $f$ is smooth and strongly convex) or at the sublinear rate $1/t^2$ (if $f$ is smooth and convex), which is a significant speedup. We will also derive convergence if $f$ and $C$ are nonconvex.}

%% file: secondOrder.tex
\section{Accelerated Gradient Flow}\label{Sec:AGF}
We begin our presentation with a derivation in continuous time. The corresponding discrete-time algorithms will be stated subsequently. A natural starting point is the work of authors such as \citet{PolyakHeavyBall}, \citet{SuAcc},
and \citet{ourWork}, who argued that in the \emph{unconstrained} case, accelerated optimization algorithms can be viewed as dynamical systems described by second-order differential equations. A canonical example is the following:
\begin{equation}
\dot{u}(t)+2 \delta u(t) +\nabla f(x(t)+\beta u(t))=0,  \label{eq:diff1}
\end{equation}
where we use the variable $u(t)=\dot{x}(t)$ to denote the velocity (or momentum), and where $\delta \geq 0$ and $\beta\geq 0$ are damping parameters.\footnote{The variables $\delta$, $\beta$ may also depend on time. For ease of presentation we focus on the case where $\delta$ and $\beta$ are fixed, but also state corresponding results for time-varying parameters.}

In the presence of constraints, $u(t)$ is allowed to be discontinuous, which is in sharp contrast to \eqref{eq:diff1}. For example, if the trajectory $x(t)$ approaches the boundary of the feasible set, an instantaneous jump in $u(t)$ might be required to ensure that $x(t)$ remains feasible. Thus, compared to \eqref{eq:form1}, where the state $x(t)$ is absolutely continuous, we are now in a position where we allow for the state $(x(t),u(t))$ (which includes the velocity $u$) to be discontinuous. This means that in addition to a differential equation of the type \eqref{eq:form1}, which characterizes the smooth motion, we also prescribe how the discontinuities in $u$ can arise. If we regard $(x(t),u(t))$ as the position and velocity of a mechanical system, discontinuities in $u$ have a mechanical meaning as impacts, which are described by a corresponding impact law. The mathematical formalism, which enables discontinuities in $u$, is summarized next.

We regard the state $z:=(x,u)$ to be the result of an integration process:
\begin{equation*}
z(t)=z(t_0) + \int_{t_0}^{t} \diff z, \quad \forall t\geq t_0.
\end{equation*}
However, instead of the usual Lebesgue density $\diff z=\dot{z}(t)\dt$, $\diff z$ now represents a differential measure \citep{Leine}, and admits both a density with respect to the Lebesgue measure (denoted by $\dt$), as well as a density with respect to an atomic measure (denoted by $\diff \eta$). As is common in non-smooth mechanics, we assume that $z(t)$ is of locally bounded variation and does not contain any singular terms. This means that $z(t)$ can be decomposed as an absolutely continuous function and a piecewise constant step function \citep{Leine}. At every time $t$, $z(t)$ has well-defined left and right limits, $z(t)^-$ and $z(t)^+$, even though $z(t)$ might not exist or might not be of interest. We can express the differential measure $\diff z$ as
$\diff z=\dot{z}(t)\dt + (z(t)^+-z(t)^-) \diff \eta,$
and the integration over an interval $[t_0,t]$, which contains the time instants $t_{\text{d}i}$, $i=1,2,\dots$, where $z(t)$ is discontinuous, yields
\begin{equation*}
z(t)^+=z(t_0)^- + \int_{t_0}^{t} \dot{z}(t) \dt + \sum_{i\geq 1} z(t_{\text{d}i})^+-z(t_{\text{d}i})^-.
\end{equation*}
As a consequence of allowing the state to be discontinuous, we need to delineate both the density $\dot{z}(t)$ with respect to the Lebesgue measure $\dt$ (which describes the smooth part of the motion) as well as the density $z(t)^+-z(t)^-$ (which describes the non-smooth motion) for fully determining the state trajectory $z(t)$. By analogy to non-smooth mechanics \citep[see, e.g.,][]{Studer}, this can be achieved with the following measure-differential inclusion:
\begin{equation}
\diff u + 2 \delta u \dt + \nabla f(x+\beta u)\dt = \sum_{i\in I_x} \nabla g_i(x) \diff \lambda_i,\quad  
\gamma_i^+ + \epsilon \gamma_i^- \in N_{\mathbb{R}_{\leq 0}}(-\diff \lambda_i), \quad i\in I_x,\label{eq:mde1}
\end{equation}
where $\epsilon\in [0,1)$ is a constant, $\gamma_i$ is the velocity associated with the $i$th constraint and is defined as
\begin{equation*}
\gamma_i(x,u):=\nabla g_i(x)\T u + \alpha g_i(x),
\end{equation*}
and where we have omitted the dependence on $t$ (as we will do frequently in the subsequent presentation). We note that the set $I_x$ (or $I_{x(t)}$ in full notation) is time-dependent.
The normal cone inclusion in \eqref{eq:mde1} is illustrated with Fig.~\ref{Fig:NCi} and will be further discussed below. The constant $\epsilon$ has the interpretation of a restitution coefficient, whereby $\epsilon=0$ leads to inelastic collisions, and $\epsilon=1$ yields elastic collisions. \hchange{We use the following solution concept for \eqref{eq:mde1}:
\begin{definition}
    (Solution of a measure differential inclusion, see \cite[Def.~4.7]{Leine}) A solution $z(t)=\varphi(t,t_0,z_0)$ of the measure differential inclusion \eqref{eq:mde1} with initial condition $z(t_0)^-=z_0$ is a function $z:\mathbb{R} \rightarrow \mathbb{R}^{2n}$, being of locally bounded variation, which fulfills \eqref{eq:mde1} and $\diff x=u(t) \dt$ for all $t\geq t_0$ and which is not defined at its discontinuity points.
\end{definition}} Measure-differential inclusions are common in non-smooth mechanics. \hchange{However, establishing existence of solutions is challenging, and solutions are rarely unique. Exceptions arise in situations where the right-hand side of the differential inclusion is monotone, which is, however, too restrictive for our purposes. We refer the reader for existence results to the literature, see, for example \cite{DiPiazza, Leine2} and references therein.} 

We note that if $x$ is in the interior of the feasible set, $I_x$ is empty, and therefore \eqref{eq:mde1} reduces to \eqref{eq:diff1}. This means that \eqref{eq:mde1} generalizes \eqref{eq:diff1} from the unconstrained case to the constrained case by including the constraint $\gamma_i^+ + \epsilon \gamma_i^- \in N_{\mathbb{R}_{\leq 0}} (-\diff \lambda_i)$, which, as we will discuss below, describes the discontinuities of $u$ via Newton's impact law and imposes the velocity constraint $u(t)\in V_\alpha(x(t))$, whenever $u(t)$ exists.


\begin{figure}
\newlength{\figurewidth}
\newlength{\figureheight}
\begin{center}
\setlength{\figurewidth}{.35\columnwidth}
\setlength{\figureheight}{.28\columnwidth}
\input{media/discNC.tikz}\hspace{1cm}
\scalebox{0.85}{\input{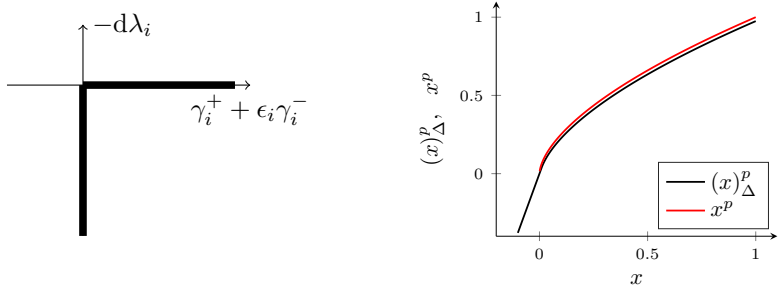}}
\end{center}
\caption{The left panel shows the normal cone inclusion $\gamma_i^+ +\epsilon \gamma_i^-\in N_{\mathbb{R}_{\leq 0}}(-\diff \lambda_i)$, which is equivalent to the complementarity condition $\diff \lambda_i\geq 0$, $\gamma_i^+ + \epsilon \gamma_i^- \geq 0$, $\diff \lambda_i (\gamma_i^+ + \epsilon \gamma_i^-)=0$. The right panel shows the approximation $(x)^p_\Delta$ of $x^p$ for $\Delta=0.01$ and $p=0.6$. There is an excellent agreement between the approximation and $x^p$ even though $\Delta$ is comparably large. In the numerical experiments, see Sec.~\ref{Sec:NumEx}, $\Delta$ is set to $10^{-6}$.}
\label{Fig:NCi}
\end{figure}

It is important to note that \eqref{eq:mde1} is understood in the sense of integration: For any compact time interval $[t_0,t_1]$, \eqref{eq:mde1} defines the difference $u(t_1)^+-u(t_0)^-$, which is obtained by integrating $\diff u$ from $t_0$ to $t_1$; similarly, the difference $x(t_1)^+-x(t_0)^-$ is obtained by integrating $u(t) \dt$. This means that \eqref{eq:mde1} has a very natural discretization, which will be discussed in the next paragraph. \hchange{The second and the third paragraphs} describe the interpretation of \eqref{eq:mde1} in terms of its smooth and non-smooth components. Formal convergence results in continuous and discrete time will be derived in Sec.~\ref{Sec:ConvAnal}.

\paragraph{Discretization of \eqref{eq:mde1}:} The measure-differential inclusion \eqref{eq:mde1} lends itself to the following discretization: $\diff u=u_{k+1}-u_k$, $\dt = T_k$, $\diff \lambda_i=\Lambda_{ki}$, $\gamma_i^+=\gamma_i(x_{k},u_{k+1})$, $\gamma_i^-=\min\{0,\gamma_i(x_k,u_k)\}$, where $T_k>0$ is the step size.\footnote{The $\min$ in $\gamma_i^-$ ensures $u_{k+1} \in V_\alpha(x_k)$, which is not automatically satisfied in discrete time.} This yields
\begin{align*}
u_{k+1}-u_k + \hchange{2}\delta u_k T_k + \nabla f(x_k  + \beta u_k) T_k &= \sum_{i\in I_{x_k}} \nabla g_i(x_k) \Lambda_{ki},\\
\gamma_i(x_k,u_{k+1}) + \epsilon \min\{0,\gamma_i(x_k, u_k)\} &\in N_{\mathbb{R}_{\leq 0}}(-\Lambda_{ki}), \quad i\in I_{x_k}.
\end{align*}
We use the newly computed momentum for updating the position $x_k$:
$x_{k+1}=x_k + T_k u_{k+1}$, 
which is motivated by analogy to unconstrained optimization. (This discretization scheme is found to be superior compared to the standard Euler method \cite{ourWork2}.) The resulting update for $u_{k+1}$ can be interpreted as a stationarity condition for $u_{k+1}$, and as a result, the proposed algorithm can be summarized as follows:
\begin{empheq}[box=\widefbox]{align}
u_{k+1}&\!=\!\argmin_{v\in \mathbb{R}^n} \!\frac{1}{2} |v\!-\!u_k\!+\!2 \delta u_k T_k\!+\!\nabla f(x_k+\beta u_k) T_k|^2, \nonumber\\
&~~\text{s.t.} ~~\gamma_i(x_k,v)\geq -\epsilon \min\{\gamma_i(x_k,u_k),0\}, ~~ i\in I_{x_k}\nonumber\\
x_{k+1}&\!=\!x_{k}+T_k u_{k+1}.\label{eq:dis1}
\end{empheq}
%
Remark~\ref{rem:nonempty} applies here in the same way: If $C$ is convex and $\epsilon=0$, the feasible set in \eqref{eq:dis1} is guaranteed to be nonempty, which means that $u_{k+1}$ is well defined (existence and uniqueness). If $C$ is nonconvex or $\epsilon> 0$, nonemptiness of the feasible set is guaranteed if constraint qualifications are satisfied (for example Mangasarian-Fromovitz). These constraint qualifications are generic, as discussed in Remark~\ref{rem:nonempty}, which ensures that $u_{k+1}$ is well defined as long as $x_k$ stays in a neighborhood of the feasible set. As will be shown with our convergence analysis (see Sec.~\ref{Sec:ConvAnal}), we can indeed ensure that $x_k$ remains in a neighborhood of $C$; the size of the neighborhood can be controlled by the value of $T_k$. The pseudo-code of the full algorithm is listed in App.~\ref{App:Alg}.

The following remarks are important:
\begin{itemize}
\item[i)] The update \eqref{eq:dis1} has the interpretation of choosing $u_{k+1}$ to be as close as possible to the update in the unconstrained case subject to the velocity constraint $\gamma_i(x_k,u_{k+1}) \geq -\epsilon \min\{\gamma_i(x_k,u_k),0\}$. As a result, in case $I_{x_k}$ is empty, \eqref{eq:dis1} reduces to a standard momentum-based method; if $\beta=0$ we obtain the heavy-ball algorithm, if $\beta\neq 0$ we obtain Nesterov's method.
\item[ii)] The update \eqref{eq:dis1} includes only the constraints $I_{x_k}$ which are active at iteration $k$. The constraint on $v$ in \eqref{eq:dis1} is guaranteed to be convex, even if the underlying feasible set is nonconvex. The constraints in \eqref{eq:dis1} yield therefore a sparse, local and convex approximation of the feasible set. Instead of performing optimizations on the position level as is common with projected gradients or the Frank-Wolfe method, \eqref{eq:dis1} suggests to constrain the velocities $u_k$, $k=1,2,\dots$.
\end{itemize}

We now proceed to give an interpretation and explanation of the continuous-time dynamics \eqref{eq:mde1}.

\paragraph{Smooth motion:}
If $u(t)$ happens to be absolutely continuous in the interval $(t_0,t_1)$, its differential measure reduces to $\dot{u}(t)\dt$. Similarly, the multipliers $\diff \lambda_i$ have only a density with respect to the Lebesgue measure $\dt$, which we denote by $\lambda_i(t)$, i.e., $\diff \lambda_i=\lambda_i(t) \dt$. As a result, \eqref{eq:mde1} reduces to 
\begin{equation}
\dot{u} + 2\delta u + \nabla f(x+\beta u) = \sum_{i\in I_{x}} \nabla g_i(x)\lambda_i, \label{eq:smoothpart}
\end{equation} 
for all $t\in (t_0,t_1)$ (a.e.). Furthermore, absolute continuity of $u(t)$ implies absolute continuity of $\gamma_i$, i.e., $\gamma_i^+=\gamma_i^-=\gamma_i$. In the limit $\dt \downarrow 0$, the inclusion in \eqref{eq:mde1} therefore reduces to
\begin{equation*}
(1+\epsilon) \gamma_i \in N_{\mathbb{R}_{\leq 0}} (-\lambda_i) \quad \Leftrightarrow \quad \gamma_i \in N_{\mathbb{R}_{\leq 0}} (-\lambda_i),
\end{equation*}
for all $i\in I_{x(t)}$ and for all $t\in (t_0,t_1)$ (a.e.). ($N_{\mathbb{R}_{\leq 0}}$ is a cone; we can therefore divide by $1+\epsilon>0$.) The normal cone inclusion prescribing the relationship between $\gamma_i$ and $\lambda_i$ is similar to Fig.~\ref{Fig:NCi}. 

From a physics perspective the normal cone inclusion $\gamma_i \in N_{\mathbb{R}\leq 0} (-\lambda_i)$ represents a force law, which by conic duality can also be expressed as (see again Fig.~\ref{Fig:NCi})
\begin{equation*}
-\lambda_i \in N_{\mathbb{R} \geq 0} (\gamma_i).
\end{equation*}
The sum $\nabla g_i(x(t)) \lambda_i$ over $i\in I_x$ on the right-hand side of \eqref{eq:smoothpart} therefore has a physical interpretation as a constraint force:
\begin{equation*}
-R=-\sum_{i\in I_x} \nabla g_i(x) \lambda_i \in N_{V_\alpha (x)} (u),
\end{equation*}
which imposes the velocity constraint $u(t)\in V_\alpha (x(t))$ for all $t\in (t_0,t_1)$ (a.e.). By virtue of Gr\"onwall's inequality this ensures that constraint violations decrease linearly at rate $\alpha$, as highlighted in \eqref{eq:groenwall}.

We therefore conclude that in case of smooth motion, the measure-differential inclusion \eqref{eq:mde1} generalizes the differential equation \eqref{eq:diff1} from the unconstrained case to the constrained case, where the additional constraint force $R(t)$ imposes the velocity constraint $u(t)\in V_\alpha(x(t))$ (for almost all $t$). The introduction of the force $R(t)$ is analogous to \eqref{eq:form1}.

Since the motion is smooth for almost every $t$, the normal cone inclusion in \eqref{eq:mde1} guarantees the satisfaction of the velocity constraint $u(t)\in V_{\alpha}(x(t))$ (or equivalently, $\gamma_i(x(t),u(t))\geq 0$ for all $i\in I_{x(t)}$) for all $t\geq 0$ (a.e.). However, when a new constraint arises at time $t_0$, there might be a situation where $\gamma_i^-(x(t_0),u(t_0))<0$. In such a case an impact will be required to ensure that $\gamma_i^+(x(t_0),u(t_0))\geq 0$. This is the subject of the next paragraph.

\paragraph{Non-smooth motion:} In order to derive the non-smooth motion we integrate \eqref{eq:mde1} over a time instant $\{t\}$, where $u(t)$ is discontinuous; that is, $u(t)^- \neq u(t)^+$. Due to the fact that the singleton $\{t\}$ has zero Lebesgue measure, we are left with the atomic parts, leading to $\diff u=(u(t)^+-u(t)^-)\diff \eta$, $\diff \lambda_i:=\Lambda_i \diff \eta$,
\begin{equation}
u(t)^+-u(t)^-=\sum_{i\in I_{x(t)}} \nabla g_i(x(t)) \Lambda_i,\quad 
\gamma_i^+ + \epsilon \gamma_i^- \in N_{\mathbb{R}\leq 0}(-\Lambda_i), \quad i\in I_{x(t)}. \label{eq:ns1}
\end{equation}
The normal cone inclusion should be interpreted as a generalization of Newton's impact law. For $\Lambda_i>0$, it implies $\gamma_i^+ + \epsilon \gamma_i^-=0$, meaning that the velocity associated to constraint $i$ after impact, $\gamma_i^+$, is $-\epsilon \gamma_i^-$, where $\gamma_i^-$ is the velocity associated to constraint $i$ before impact. From the discussion of the smooth motion it follows $\gamma_i(x(t_0),u(t_0))^-$ at time $t_0$ can only be negative if the constraint $i$ becomes active at time $t_0$; that is, $i\not\in I_{x(t)}$ for $t<t_0$ and $i\in I_{x(t)}$ for $t=t_0$. This necessitates a discontinuity in $u$ at time $t_0$, which according to the above normal cone inclusion comes in two variants: i) $\Lambda_i>0$, which implies $\gamma_i^+=-\epsilon \gamma_i^-$ and ii) $\Lambda_i=0$, which implies $\gamma_i^+\geq -\epsilon \gamma_i^-$. In variant i), the impulsive force $\Lambda_i$ contributes the component $\Lambda_i \nabla g_i(x(t))$ (normal to constraint $i$) to the velocity jump $u(t_0)^+-u(t_0)^-$, whereas in variant ii), there is no such contribution. Both variants ensure $\gamma_i(x(t_0),u(t_0))^+\geq -\epsilon\gamma_i(x(t_0),u(t_0))^-\geq 0$.

The characterization of the non-smooth motion according to \eqref{eq:ns1} can be interpreted as a stationarity condition for $u(t)^+$, which yields
\begin{multline}
u(t)^+=\argmin_{v\in \mathbb{R}^n} \frac{1}{2} |v-u(t)^-|^2\quad \text{s.t.}~~ \gamma_i(x(t),v) \geq -\epsilon\gamma_i(x(t),u(t)^-),~~\forall i\in I_{x(t)}. \label{eq:statcond333}
\end{multline}
The minimization in \eqref{eq:statcond333} has the following meaning: for each $u(t)^-$ there is a unique $u(t)^+$, which is chosen to be as close as possible to $u(t)^-$ subject to Newton's impact law $\gamma_i^+\geq -\epsilon \gamma_i^-$ for all $i\in I_{x(t)}$.

\paragraph{Equilibria of \eqref{eq:mde1}:}
The equilibria of \eqref{eq:mde1} are obtained from $x(t)\equiv x_0$, $\diff \lambda_i \equiv \lambda_{0i} \dt$, $u(t)\equiv 0$, $\diff u\equiv 0$, where $x_0\in \mathbb{R}^n$ and the multipliers $\lambda_{0i}\geq 0$, $i\in I_{x_0}$ are constant. As a result, \eqref{eq:mde1} reduces to 
\begin{multline*}
-\nabla f(x_0)+\sum_{i\in I_{x_0}}\nabla g_i(x_0) \lambda_{i0}=0, ~~ (1+\epsilon) \alpha g_i(x_0) \in N_{\mathbb{R}_{\leq 0}}(-\lambda_{i0}), ~~i\in I_{x_0}.
\end{multline*}
The normal cone inclusion can be simplified by dividing by $\alpha (1+\epsilon)>0$ (the normal cone is a cone), which implies that $g_i(x_0)$ and $\lambda_{i0}$ satisfy the complementarity conditions
\begin{equation*}
g_i(x_0)\geq 0, \quad \lambda_{i0} \geq 0, \quad \lambda_{i0} g_i(x_0)=0, \quad \forall i\in I_{x_0}.
\end{equation*}
Hence, the equilibria of \eqref{eq:mde1} satisfy the Karush-Kuhn-Tucker conditions of \eqref{eq:fundProb}, which means that the stationary points of  \eqref{eq:fundProb} are indeed equilibria.

%% file: media/discNC.tikz
\begin{tikzpicture}
\draw[->](-1,0) -- (2.2,0) node [below] {$\gamma_i^+ + \epsilon_i \gamma_i^-$};
\draw[->](0,-2) -- (0,0.8) node [right] {$-\diff \lambda_i$};
\draw[-,line width=1mm](0,0) -- (2,0);
\draw[-,line width=1mm](0,-2) -- (0,0);
\node[] () [below = 2.5cm] {}; 
\end{tikzpicture}

%% file: convergenceAnalysis.tex
\section{Convergence Analysis}\label{Sec:ConvAnal}
The following section discusses the convergence of trajectories of \eqref{eq:mde1} and \eqref{eq:dis1}, and characterizes the rate of convergence. We start by summarizing the continuous-time results. Without loss of generality we assume that $f$ is normalized such that the Lipschitz constant of the gradient is unity. 
\begin{theorem}\label{Thm:stability}
Let $(x(t)$, $u(t))$ be a trajectory satisfying \eqref{eq:mde1} with $x(0)\in C$. Let $f$ be $1$-smooth, let $g$ satisfy the Mangasarian-Fromovitz constraint qualification, and let either $f$ be convex or $2\delta - \beta > 0$. Then, $x(t)$ converges to the set of stationary points, while $u(t)$ converges to zero. Moreover, each isolated local minimum corresponds to an asymptotically stable equilibrium in the sense of Lyapunov.
\end{theorem}
\begin{proof}
From the analysis of the smooth motion in Sec.~\ref{Sec:AGF}, we infer that $\gamma_i(x(t),u(t))\geq 0$ for all $t$ such that $i\in I_{x(t)}$ (a.e.). By Gr\"{o}nwall's inequality, this implies $g(x(t))\geq 0$ for all $t\geq 0$ and therefore $x(t) \in C$ for all $t\in [0,\infty)$.

We consider the following Lyapunov function
\begin{equation}
\hchange{V(t)}=\frac{1}{2} |u(t)|^2 + f(x(t)),
\end{equation}
which is bounded below by assumption. We investigate how \hchange{$V$} evolves along the trajectories of \eqref{eq:mde1}. The differential measure corresponding to \hchange{$V$} can be expressed as \citep{Leine}:
\begin{align*}
\diff V &= \frac{1}{2} (u^+ + u^-)\T \diff u + \nabla f(x)\T u ~\dt \\
&= -2 \delta |u|^2 ~\dt - (\nabla f(x+\beta u)-\nabla f(x))\T u ~\dt +\sum_{i\in I_{x}} \left(\frac{1}{2} (\gamma_i^+ + \gamma_i^-) - \alpha \hchange{g_i(x)}\right) \diff \lambda_i.
\end{align*}
The second line follows from replacing $\diff u$ with \eqref{eq:mde1}, using the fact that the Lebesgue measure captures only the smooth motion, which means, for example, $(u^+)\T u~\dt=(u^-)\T u~\dt=|u|^2 \dt$, and adding and subtracting $\alpha g_i\diff \lambda_i$. From the assumption $2\delta - \beta >0$ (or alternatively by convexity of $f$) we can upper bound $\diff V$ by
\begin{equation*}
\diff V \leq -c_2 |u|^2 \dt + \sum_{i\in I_x} \left( \frac{1}{2}(\gamma_i^+ +\gamma_i^-) - \alpha \hchange{g_i(x)} \right) \diff \lambda_i,
\end{equation*}
where $c_2>0$ is constant. The summand in the second part of the expression can be rewritten as
\begin{equation*}
-\alpha \hchange{g_i(x)} \diff \lambda_i + \frac{1}{2} (\gamma_i^+ + \epsilon \gamma_i^-) \diff \lambda_i + \frac{1-\epsilon}{2} (\gamma_i^-) \diff \lambda_i.
\end{equation*}
The fact that $g_i(x(t))=0$ for all $i\in I_{x(t)}$ ($x(t)$ remains feasible) implies that the first term vanishes. The second term vanishes due to the complementarity condition in \eqref{eq:mde1}. The third term is guaranteed to be non-positive, since, on the one hand, $\diff \lambda_i\geq 0$ (see \eqref{eq:mde1}), and on the other hand, $\gamma_i^-\leq 0$ in case of impact, and $\gamma_i^-~\diff \lambda_i=\gamma_i \diff \lambda_i=0$ in case of smooth motion (see again \eqref{eq:mde1}). We therefore conclude that
\begin{equation}
\diff V \leq -c_2 |u|^2 \dt, \label{eq:decrease}
\end{equation}
which means that $V(x(t),u(t))$ is monotonically decreasing in $t$. 

For proving convergence to stationary points, we note that since $u(t)$ is of locally bounded variation with no singular part (by assumption), it can be decomposed in an absolutely continuous function and a piecewise constant function \citep[Lebesgue decomposition, see for example Ch.~3 of][]{Leine}. As a result, $\dot{u}(t)$ (whenever it exists) is uniformly locally integrable \citep[see][]{Teel}. Combined with the fact that $u$ is square integrable, which follows from \eqref{eq:decrease} and the observation that $V$ is bounded below, we conclude that $u(t)\rightarrow 0$ as $t\rightarrow \infty$ by a variant of Barbalat's lemma \citep[see][]{Teel}. 
We recall that $x(t)$ is bounded ($C$ is compact) and consider any cluster point $\bar{x}$ of $x(t)$, which means that there exists a sequence $t_k \rightarrow \infty$ such that $x(t_k) \rightarrow \bar{x}$. We pick any $\tau > 0$ and consider
\begin{equation*}
u(t_k+\tau)^+-u(t_k)^-=\int_{t_k}^{t_k+\tau} \diff u =  \int_{t_k}^{t_k+\tau} - 2\delta u - \nabla f(x+\beta u) \dt +\sum_{i=1}^{n_\text{g}} \int_{t_k}^{t_k+\tau} \nabla g_i(x) \diff \lambda_i,
\end{equation*}
where we set $\diff \lambda_i = 0$ whenever $i\not \in I_{x(t)}$ to simplify notation. The previous expression is guaranteed to vanish for $k\rightarrow \infty$, which by continuity of $\nabla g_i$ and $\nabla f$ means that
\begin{equation}
-\nabla f(\bar{x}) \tau + \sum_{i=1}^{n_\text{g}} \nabla g_i(\bar{x}) \int_{t_k}^{t_k+\tau} \diff \lambda_i \rightarrow 0 \label{eq:proofConv}
\end{equation}
as $k\rightarrow \infty$. We note that by continuity of $g$, there exists a constant $k_0>0$ and $c_\text{lg}>0$ such that $g_i(x(t_k))>c_\text{lg}$ for all $k>k_0$ and all $i\not\in I_{\bar{x}}$. From the fact that $u(t)\rightarrow 0$ we infer that for large enough $k$, $g_i(x(t))>c_\text{lg}/2$ for all $t\in [t_k,t_k+\tau]$ and $i\not\in I_{\bar{x}}$. We introduce the notation 
\begin{equation*}
\int_{t_k}^{t_k+\tau} \diff \lambda_i =: \lambda_k^i \tau,
\end{equation*}
for all $k>0$ and all $i\in [n_\text{g}]$ and conclude that $\lambda_k^i \geq 0$ and, for all $k$ large enough, $\lambda_k^i=0$ if $i\not\in I_{\bar{x}}$. 

We argue next that $\lambda_k^i$ is bounded for all $i\in [n_\text{g}]$. We argue by contradiction and consider the sequence $(\lambda_k^1,\dots,\lambda_k^{n_\text{g}})/|(\lambda_k^1,\dots,\lambda_k^{n_\text{g}})|$, which is guaranteed to be bounded, even though $|(\lambda_k^1,\dots,\lambda_k^{n_\text{g}})|\rightarrow \infty$. Upon dividing \eqref{eq:proofConv} by $|(\lambda_k^1,\dots,\lambda_k^{n_\text{g}})|$ we conclude that 
\begin{equation*}
\sum_{i=1}^{n_\text{g}} \nabla g_i(\bar{x}) \xi_i = 0,
\end{equation*}
where $\xi$ denotes an accumulation point of $(\lambda_k^1,\dots,\lambda_k^{n_\text{g}})/|(\lambda_k^1,\dots,\lambda_k^{n_\text{g}})|$, which satisfies $\xi_i \geq 0$, $\xi_i=0$ for all $i\not\in I_{\bar{x}}$, and $|\xi|=1$. However, the fact that $\xi\neq 0$ contradicts the constraint qualification, since it precludes the existence of a vector $w$ such that $w\T \nabla g_i(\bar{x})>0$ for all $i\in I_{\bar{x}}$ (which would mean $\sum_{i=1}^{n_\text{g}} w\T \nabla g_i(\bar{x}) \xi_i>0$). This implies that $\lambda_k^i$ is bounded for all $i\in [n_\text{g}]$.

We take any accumulation point of $(\lambda_k^1,\dots,\lambda_k^{n_\text{g}})$, which we denote by $\bar{\lambda}$. The accumulation point $\bar{\lambda}$ satisfies $\bar{\lambda}\geq 0$, $\bar{\lambda}_i=0$ for all $i\not \in I_{\bar{x}}$ (complementary slackness), and by \eqref{eq:proofConv}, $-\nabla f(\bar{x}) + \sum_{i=1}^{n_\text{g}} \nabla g_i(\bar{x}) \bar{\lambda}_i = 0$. Hence, $\bar{x}$ and $\bar{\lambda}$ satisfy the Karush-Kuhn-Tucker conditions of \eqref{eq:fundProb}, and $\bar{x}$ is stationary.

We conclude the proof by showing asymptotic stability of isolated local minima (in the sense of Lyapunov). We accordingly pick any isolated local minimum $x^*$ and note that the function $|u|^2/2 +f(x)-f(x^*)$ is positive definite in a neighborhood of $(x^*,0)$. We conclude from \eqref{eq:decrease} that $x^*$ is therefore stable in the sense of Lyapunov. We have already shown attractiveness (see previous paragraphs) and therefore conclude that $x^*$ is asymptotically stable in the sense of Lyapunov.
\end{proof}
\begin{table}
\caption{The table summarizes convergence rates that arise from different choices of $\alpha$, $\beta$, and $\delta$. Without loss of generality $f$ is normalized, such that the Lipschitz constant of the gradient is unity.}
\begin{tabular}{l|c|c|c|c}
variant & $\alpha$ & $\delta$ & $\beta$ & rate $\rho$ \\ \hline
heavy ball& $\sqrt{\mu}$ & $\sqrt{\mu}$ & $0$ & $e^{-\sqrt{\mu} t}$\\
Nesterov constant parameters & $\sqrt{\mu}-\mu/2$ & $\frac{\sqrt{\mu}}{1+\sqrt{\mu}}$ & $\frac{1-\sqrt{\mu}}{1+\sqrt{\mu}}$ & $e^{-(\sqrt{\mu}-\mu/2)t}$ \\
Nesterov varying parameters & $\frac{2}{t+3}$ & $\frac{3}{2(t+3)}$ & $\frac{t}{t+3}$ & $\frac{9}{(t+3)^2}$
\end{tabular}
\label{Tab:params}
\end{table}
%
%
%
%

The following theorem demonstrates that the use of momentum combined with well-chosen damping parameters indeed results in accelerated convergence rates ($\mathcal{O}(1/t^2)$ in the smooth and convex case, and $e^{-\sqrt{\mu} t}$ in the smooth and strongly convex case).\footnote{Continuous-time rates are indeed meaningful in this context, since both the Lipschitz constant of $\nabla f$ and the constant in front of $\nabla f$ in \eqref{eq:mde1} are fixed to unity. This fixes the time scale, as any reparametrization of time, i.e., $t=\tau(s)$, where $\tau:\mathbb{R}_{\geq 0} \rightarrow \mathbb{R}_{\geq 0}$ is a diffeomorphism, would alter the way $\nabla f$ enters \eqref{eq:mde1}. We refer the reader to \citet{ourWork3} for a more general discussion.}
\begin{theorem}\label{Thm:ConvRate}
Let $C$ be convex and $f$ be $1$-smooth and either convex or strongly convex with strong convexity constant $\mu>0$. Let the parameters $\alpha$, $\beta$, $\delta$, and $\rho$ be chosen according to Table~\ref{Tab:params} and assume that Slater's condition holds. Then, for any $x(0)\in \mathbb{R}^{n}$, $u(0)=0$, the following holds:
\begin{multline*}
\min\{ 0, g(x(0))\} \T \lambda^* \rho(t) \leq \\ f(x(t))-f(x^*)
\leq \left( \frac{\alpha^2}{2} |x(0)-x^*|^2 + f(x(0))-f(x^*)\right) \rho(t),
\end{multline*}
where $x^*$ is a minimizer of \eqref{eq:fundProb} and $\lambda^*$ is a multiplier that satisfies the Karush-Kuhn-Tucker conditions.
\end{theorem}
\begin{proof}
\paragraph{i) Heavy ball ($\alpha=\sqrt{\mu}$):}
We start by analyzing the heavy ball variant ($\alpha=\sqrt{\mu}$) and consider the evolution of the function
\begin{equation*}
\hchange{W}(t)=\frac{1}{2} |\alpha (x(t)-x^*)+u(t)|^2 + f(x(t))-f(x^*),
\end{equation*}
along the trajectories $(x(t),u(t))$. The corresponding differential measure $\diff W$ is given by \citep{Leine}:
\begin{equation*}
(\alpha (x-x^*)+ (u^+ + u^-)/2 )\T (\alpha u ~\dt + \diff u) + \nabla f(x)\T u ~\dt.
\end{equation*}
By following the same steps as in the proof of Thm.~\ref{Thm:stability} we obtain
\begin{equation*}
(u^+ + u^-)/2\T  \diff u + \nabla f(x)\T u ~\dt \leq 
 - 2\delta |u|^2~ \dt
+ \sum_{i\in I_x} -\alpha g_i(x) \diff \lambda_i.
\end{equation*}
However, in contrast to Thm.~\ref{Thm:stability}, we allow also for infeasible initial conditions, and therefore the term $-\alpha g_i(x) \diff \lambda_i$ remains. This yields the following upper bound on $\diff W$:
\begin{multline}
\diff W \leq -(2\delta -\alpha) |u|^2 \dt - \alpha (x-x^*)\T \nabla f(x) \dt 
+(\alpha^2-2\alpha \delta) (x-x^*)\T u \dt \\
+ \sum_{i\in I_x} (\alpha \nabla g_i(x)\T (x-x^*) - \alpha g_i(x)) \diff \lambda_i. \label{eq:sumtmp3}
\end{multline}
The fact that $f$ is strongly convex means that the following holds: 
\begin{align*}
-(x-x^*)\T \nabla f(x) \leq &-(f(x)-f(x^*))-\frac{\mu}{2} |x-x^*|^2.
\end{align*}
In addition, each $g_i$ is concave, and therefore
\begin{equation*}
\nabla g_i(x)\T (x^*-x)\geq g_i(x^*) - g_i(x)\geq -g_i(x),
\end{equation*}
where we used the fact that $x^*$ is feasible for the last inequality. This concludes that the summand in \eqref{eq:sumtmp3} is non-positive
since $\diff \lambda_i\geq 0$. Thus, after some elementary manipulations, we obtain the following upper bound on $\diff W$:
\begin{align*}
\diff W \leq &-\alpha W \dt -(2\delta -3 \alpha/2) |u|^2 \dt \leq -\alpha W \dt.
\end{align*}
Applying Gr\"{o}nwall's inequality then implies 
the desired upper bound on $f(x(t))-f(x^*)$. 

The lower bound is obtained from a perturbation analysis, using an argument similar to \citet{ownWorkC}. We define
\begin{equation*}
f^*(t):=\min_{\xi\in \mathbb{R}^n} f(\xi), \quad \text{s.t.} \quad g(\xi)\geq \min\{0,g(x(0))\} e^{-\alpha t},
\end{equation*} 
which is of the form \eqref{eq:fundProb}, but with a modified right-hand side of the constraints. The trajectory $x(t)$ satisfies $g(x(t))\geq \min\{0,g(x(0))\} e^{-\alpha t}$ and is therefore a feasible candidate for the above minimization, which implies $f(x(t))\geq f^*(t)$. The minimum is clearly attained, since $f$ is bounded below, $f(x)\rightarrow \infty$ for $|x|\rightarrow \infty$, and the modified set of feasible points is closed. The multiplier $\lambda^*$ satisfying the Karush-Kuhn-Tucker conditions of \eqref{eq:fundProb} captures the sensitivity of the cost function with respect to perturbations of the constraints. This means that $\lambda^*$ is guaranteed to satisfy the following inequality \citep[see, e.g.,][p.~277]{RockafellarConvex}:
\begin{equation*}
f^*(t)-f^* \geq \min\{ 0, g(x(0))\}\T \lambda^* e^{-\alpha t},
\end{equation*}
which combined with $f(x(t))\geq f^*(t)$ implies the desired lower bound. 

\paragraph{ii) Nesterov - constant parameters ($\alpha=\sqrt{\mu}-\mu/2$):}
We consider again the evolution of the function $W(t)$ (see i)) along the trajectories $(x(t),u(t))$. The corresponding differential measure is given by
\begin{equation*}
\diff W= (\hchange{\alpha} (x-x^*)+ (u^+ + u^-)/2 )\T \\ (\hchange{\alpha} u \dt + \diff u) + \nabla f(x)\T u \dt.
\end{equation*}
From the proof of Thm.~\ref{Thm:stability} we obtain
\begin{multline}
(u^+ + u^-)/2\T  \diff u + \nabla f(x) u \dt \leq -(2\delta +\mu \beta) |u|^2 \dt
 \\
 + u\T (\nabla f(x+\beta u)-\nabla f(x))\dt + \sum_{i\in I_x} -\alpha g_i \diff \lambda_i.
\end{multline}
Including the remaining terms yields the following upper bound on $\diff W$:
\begin{multline}
-(2\delta -\hchange{\alpha}) |u|^2 \dt + (\hchange{\alpha}^2 - 2\delta a) (x-x^*)\T u \dt 
- (\nabla f(y) - \nabla f(x))\T u \dt - \hchange{\alpha} (x-x^*)\T \nabla f(y) \dt \\
+ \sum_{i\in I_x} (a \nabla g_i(x)\T (x-x^*) - \alpha g_i) \diff \lambda_i, \label{eq:sumtmp}
\end{multline}
where we introduced the variable $y:=x+\beta u$ to simplify notation. The fact that $f$ is strongly convex means that the following holds:
\begin{multline}
-(x-x^*)\T \nabla f(y) \leq -(f(x)-f(x^*)) -\beta (\nabla f(x)-\nabla f(y))\T u 
- \mu \beta^2 |u|^2 \\
- \mu \beta u\T (x-x^*) -\frac{\mu}{2} |x-x^*|^2. \label{eq:tmptmp}
\end{multline}
In addition, $C$ is convex, which implies that each $g_i$ is concave. As a result,
\begin{equation*}
\nabla g_i(x)\T (x^*-x)\geq g_i(x^*) - g_i(x)\geq -g_i(x),
\end{equation*}
where we used the fact that $x^*$ is feasible for the last inequality. The summand in \eqref{eq:sumtmp} can therefore be upper bounded by
\begin{equation*}
\left( \hchange{\alpha} \nabla g_i(x)\T (x-x^*) - \alpha g_i(x)\right) \diff \lambda_i \leq 0,
\end{equation*}
since, by definition $\diff \lambda_i\geq 0$. Combined with \eqref{eq:tmptmp}, and after some elementary manipulations, this yields the following upper bound on $\diff W$:
\begin{multline*}
\diff W \leq -\hchange{\alpha} W \dt -(2\delta -3\hchange{\alpha}/2 +\mu \beta ) |u|^2 \dt 
- \frac{1}{2} (\hchange{\alpha} \mu - \hchange{\alpha}^3) |x-x^*|^2 \dt \\
+ (2 \hchange{\alpha}^2 - 2 \delta \hchange{\alpha} - \hchange{\alpha} \mu \beta) (x-x^*)\T u \dt,
\end{multline*}
where we have used the fact that $1-\beta \hchange{\alpha}\geq 0$. We note that the term $2\hchange{\alpha}^2-2\delta \hchange{\alpha}-\hchange{\alpha}\mu\beta$ vanishes, that $\hchange{\alpha}(\mu-\hchange{\alpha}^2)\geq 0$, and $2\delta -3\hchange{\alpha}/2+\mu \beta \geq 0$ for all $\mu \in [0,1]$. We therefore obtain $\diff W\leq -\hchange{\alpha} W \dt$, which, by Gr\"{o}nwall's inequality, implies the desired result.\\

\paragraph{iii) Nesterov - varying parameters ($\hchange{\alpha(t)}=2/(t+3)$):}
The proof follows the same steps. We consider again the evolution of the function
\begin{equation*}
\hchange{W(t)}=\frac{1}{2} |\hchange{\alpha(t)}(x(t)-x^*) + u(t)|^2+f(x(t))-f(x^*),
\end{equation*}
along the trajectories $(x(t),u(t))$, \hchange{where $\alpha(t)$ is now time-varying, and where we will again omit the dependence with respect to $t$ to simplify notation}. The differential measure $\diff \hchange{W}$ is given by
\begin{equation*}
\diff \hchange{W}= (\hchange{\alpha} (x-x^*)+ (u^+ + u^-)/2 )\T (\hchange{\alpha} u \dt + \hchange{\dot{\alpha}} (x-x^*) \dt + \diff u) + \nabla f(x)\T u \dt,
\end{equation*}
where according to the chain rule the derivative of $\hchange{\alpha}$ with respect to $t$ enters. This leads to two additional quadratic terms of the type 
\begin{equation*}
\hchange{\alpha} \hchange{\dot{\alpha}} |x-x^*|^2 \dt \quad\text{and}\quad \hchange{\dot{\alpha}} u\T (x-x^*)\dt.
\end{equation*}
As a result, by following the same steps as in the variant ii) (see previous section), we obtain
\begin{multline*}
\diff \hchange{W}\leq -\hchange{\alpha} \hchange{W} \dt -(2\delta -3\hchange{\alpha}/2 +\mu \beta ) |u|^2 \dt 
- \frac{1}{2} (\hchange{\alpha} \mu - \hchange{\alpha}^3 - 2\hchange{\alpha} \hchange{\dot{\alpha}}) |x-x^*|^2 \dt 
\\
+ (2 \hchange{\alpha}^2 - 2 \delta \hchange{\alpha} - \hchange{\alpha} \mu \beta+\hchange{\dot{\alpha}}) (x-x^*)\T u \dt.
\end{multline*}
We note that $2\hchange{\alpha}^2-2\delta \hchange{\alpha} + \hchange{\dot{\alpha}}$ vanishes. The same applies to $2\delta-3\hchange{\alpha}/2$ and $\hchange{\alpha}^3-2\hchange{\alpha}\hchange{\dot{\alpha}}$, which simplifies the above inequality to
\begin{align*}
\diff \hchange{W}\leq -\hchange{\alpha} \hchange{W} \dt -\mu \beta |u|^2 \dt - \frac{1}{2} \hchange{\alpha} \mu  |x-x^*|^2 \dt - \hchange{\alpha} \mu \beta (x-x^*)\T u \dt.
\end{align*}
Applying Young's inequality to the cross-term $(x-x^*)\T u$ concludes that $\diff \hchange{W} \leq -\hchange{\alpha} \hchange{W} \dt$ for any $\mu\in [0,1]$. We finally apply Gr\"{o}nwall's inequality, which yields
\begin{align*}
\hchange{W(t)} \leq \hchange{W(0)} \exp\left(-\int_{0}^{t} \hchange{\alpha}(s) \diff s\right)=\hchange{W(0)} \frac{9}{(t+3)^2},
\end{align*}
and concludes the proof.
\end{proof}
We also demonstrate convergence of the discrete algorithm \eqref{eq:dis1} in a nonconvex and possibly stochastic setting. For simplicity we state and prove the deterministic result when $\epsilon=0$ (as becomes apparent from the proof, the stochastic case with bounded zero-mean gradient perturbations follows from the same arguments). We further note that the restriction $1/2<s<1$ can be loosened to $1/2<s\leq 1$ if additional assumptions on the damping parameters $\delta$ and $\beta$ are satisfied (this requires a slightly more detailed proof).
\begin{theorem}\label{Thm:ConvDT}
Let $T_k=T_0/k^s$, $k=1,2,\dots$, for some $T_0>0$ and $s\in (1/2,1)$, and let the function $\min\{0,g_1(x)\}$ have compact level sets. Let $f$ be 1-smooth and either convex or such that $2\delta -\beta >0$, let $x_k,u_k$ be the iterates defined in \eqref{eq:dis1} with $\epsilon=0$ and the initial values $(x_0,u_0)\in \mathbb{R}^{2n}$ arbitrary, and let $g$ satisfy the Mangasarian-Fromovitz constraint qualification. If $u_k$ is bounded and $f$ has isolated stationary points, then $x_k$ converges to a  stationary point of \eqref{eq:fundProb}, while $u_k$ converges to zero.
\end{theorem}
\begin{proof}
The update \eqref{eq:dis1} can be divided into the following two steps
\begin{align}
\bar{x}_k&= x_k, \qquad  \qquad \qquad \qquad \qquad ~~~~ x_{k+1}=\bar{x}_k + u_{k+1} T_k, \nonumber \\
\bar{u}_k&= u_k - T_k f_\text{d}(x_k,u_k) + R_k, \qquad u_{k+1}=\bar{u}_k - T_k \nabla f(\bar{x}_k), \label{eq:decomp}
\end{align}
where $f_\text{d}(x_k,u_k):=2 \delta u_k + \nabla f(x_k + \beta u_k) - \nabla f(x_k)$ contains the dissipative terms and $R_k$ denotes the constraint forces. The first step, which maps $(x_k,u_k)$ to $(\bar{x}_k, \bar{u}_k)$ is an update of the velocity with the dissipative terms and the constraint forces $R_k$, whereas the second step, which maps $(\bar{x}_k, \bar{u}_k)$ to $(x_{k+1},u_{k+1})$ is a symplectic Euler discretization that captures the conservative parts of the underlying dynamics. In the following we will exploit the fact that the second step is a symplectic map. More precisely, the fact that the second step is symplectic implies that an energy function \hchange{$V_k: \mathbb{R}^n\times \mathbb{R}^n \rightarrow \mathbb{R}$} (bounded below, uniformly in $k$) exists, such that
\begin{equation*}
V_{k+1}(x_{k+1},u_{k+1}) - V_k(x_k,u_k) \leq - c_{\text{V}1} T_k |u_k|^2 - c_{\text{V}2} |R_k|^2 + c_{\text{V}3} T_k^2 - \alpha \sum_{i\in I_{x_k}} \lambda_k^i g_i (x_k)
\end{equation*}
for $k\geq k_0$ and for some constants $k_0>0$, $c_{\text{V}1}, c_{\text{V}2}, c_{\text{V}3} > 0$. A detailed derivation of this step is included in App.~\ref{App:ConvDT} (see Lemma~\ref{Lem:Lyap}). Furthermore, as a result of the diminishing step-size, we infer that $g_i(x_k)\geq -c_\text{g} T_k$ holds for all sufficiently large $k$ and a constant $c_\text{g}>0$ (see Lemma~\ref{Lemma:constr} in App.~\ref{App:ConvDT}). We can therefore bound the terms $\lambda_k^i g_i(x_k)$ as follows
\begin{equation}
|-\alpha \sum_{i\in I_{x_k}} \lambda_k^i g_i (x_k)| \leq \frac{\alpha}{c_\lambda} c_\text{g} T_k |R_k| n_\text{g} \leq \frac{c_{\text{V}2}}{2} |R_k|^2 + \frac{1}{ 2 c_{\text{V}2}} (n_\text{g} c_\text{g} \frac{\alpha}{c_\lambda})^2 T_k^2,
\end{equation}
where we have exploited constraint qualification to conclude that $|\lambda_k|\leq |R_k|/c_\lambda$ (see Lemma~\ref{lemma:d2} in App.~\ref{App:ConvDT}) in the first step, and we have used Young's inequality in the second step. We therefore obtain 
\begin{equation}
V_{k+1}(x_{k+1},u_{k+1}) - V_k(x_k,u_k) \leq - c_{\text{V}1} T_k |u_k|^2 - \frac{c_{\text{V}2}}{2} |R_k|^2 + \bar{c}_{\text{V}3} T_k^2, \label{eq:proofdecV}
\end{equation}
for large enough $k$ and a modified constant $\bar{c}_{\text{V}3}>0$. We further introduce the constants
\begin{equation*}
-c_{\text{V}0} = \inf_k V_k(x_k,u_k), \qquad c_{\text{T}1}= \sum_{k=k_0}^{\infty} T_k^2,
\end{equation*}
and make the following claim: There exists a subsequence $k(j), j=1,2,\dots$, such that 
\begin{equation*}
-c_{\text{V}1} T_{k(j)} |u_{k(j)}|^2 - \frac{c_{\text{V}2}}{2} |R_{k(j)}|^2 \geq \underbrace{\left(-\bar{c}_{\text{V}3} - \frac{V_{k_0}(x_{k_0},u_{k_0})+c_{\text{V}0}}{c_{\text{T}1}} - \frac{1}{c_{\text{T}1}} \right)}_{:=-c_{\text{V}4}} T_{k(j)}^2,
\end{equation*}
where $c_{\text{V}4}>0$.
For the sake of contradiction we assume that the claim is not true, which means that 
\begin{equation*}
-c_{\text{V}1} T_{k} |u_{k}|^2 - \frac{c_{\text{V}2}}{2} |R_{k}|^2 \leq - {c}_{\text{V}4} T_{k}^2
\end{equation*}
for all $k>0$. However, we have chosen the constant $c_{\text{V}4}$ deliberately in such a way that we can generate a contradiction when summing over $V_{k+1}(x_{k+1},u_{k+1})-V_{k}(x_k,u_k)$. More precisely,
\begin{align*}
-c_{\text{V}0} - V_{k_0}(x_{k_0},u_{k_0})& \leq V_N(x_N,u_N) - V_{k_0}(x_{k_0},u_{k_0}) \leq \sum_{k=k_0}^{N} (-c_{\text{V}4} + \bar{c}_{\text{V}3}) T_k^2 \\
&\leq \frac{-c_{\text{V}0} - V_{k_0}(x_{k_0},u_{k_0}) -1 }{c_{\text{T}1}} \sum_{k=k_0}^{N} T_k^2,
\end{align*}
which leads to the desired contradiction since the right-hand side approaches $-c_{\text{V}_0} - V_{k_0}(x_{k_0},u_{k_0}) - 1$ for $N\rightarrow \infty$.

Upon passing to another subsequence we infer that $u_{k(j)} \rightarrow 0$,  $R_{k(j)}/T_{k(j)} \rightarrow \bar{R}$, and $x_{k(j)} \rightarrow \bar{x}\in C$. By the properties of the set $V_\alpha(x)$ we conclude that $\bar{x}$ is stationary and $\bar{x}$ and $\bar{R}$ satisfy the Karush-Kuhn-Tucker conditions, see Lemma~\ref{Lemma:d3} in App.~\ref{App:ConvDT}.

We now prove that the entire sequence converges. \hchange{We infer from \eqref{eq:proofdecV} and \cite[Lemma 2, Sec.~2.2.1]{Polyak} that the sequence $V_k(x_k,u_k)$ converges and satisfies
\begin{equation*}
\lim_{k\rightarrow \infty} V_k(x_k,u_k)= \lim_{j\rightarrow \infty} \left(\frac{1}{2}|u_{k(j)}|^2 + f(x_{k(j)}) \right) = f(\bar{x}),
\end{equation*}}
where $\bar{x}$ is the limit of $x_{k(j)}$ (see above). We further introduce the collection of all accumulation points of $x_k,u_k$, that is,
\begin{equation*}
\omega := \bigcap_{m\geq 0} \text{cl}\{(x_k,u_k):k>m\},
\end{equation*}
where $\text{cl}$ denotes closure. We note that $\omega\subset C \times \mathbb{R}^n$ is connected (since $|x_{k+1}-x_k| \rightarrow 0$ and $|u_{k+1}-u_k|\rightarrow 0$) and for any $(x,u)\in \omega$, $|u|^2/2+f(x)=f(\bar{x})$. 

We claim that $\omega=\{(\bar{x},0)\}$. For the sake of contradiction, we assume the existence of a sequence $(\tilde{x}_k, \tilde{u}_k) \rightarrow \hchange{(\tilde{x},\tilde{u})}$, $(\tilde{x}_k,\tilde{u}_k)\in \omega$ with \hchange{$(\tilde{x},\tilde{u}) \neq (\bar{x},0)$} ($\omega$ is connected). We consider first the case where $\bar{x}$ lies in the interior of $C$ (constraints are not active). In that case $\tilde{x}_k$ lies likewise in the interior of $C$ for large $k$. However, in the absence of constraints the dynamics in \eqref{eq:dis1} are smooth, which means that $\tilde{x}_k$ and $\tilde{u}_k$ are guaranteed to be equilibria and therefore satisfy $\tilde{u}_k=0$ and $\nabla f(\tilde{x}_k)=0$. This contradicts the fact that $f$ has isolated stationary points. Next, we consider the case where $\bar{x}$ lies on the boundary of $C$, where we infer from $(\tilde{x},\tilde{u})\in \omega$ that \hchange{$f(\tilde{x})=f(\bar{x})-|\tilde{u}|^2/2$. Therefore $f(\tilde{x}) \leq  f(\bar{x})$, $\tilde{x}\in C$,} which also contradicts the fact that $\bar{x}$ is an isolated stationary point.

Thus, we conclude $\omega=\{(\bar{x},0)\}$ and the result follows.
\end{proof}

We note that the behavior of algorithm \eqref{eq:dis1} is complex, as it relies on a \emph{local} approximation of the feasible set, whereby multiple constraints can become active or inactive over the course of the optimization. Establishing Thm.~\ref{Thm:ConvDT} is therefore nontrivial and requires blending ideas from numerical analysis, optimization, and dynamical systems.

We conclude the section by proving an accelerated rate for a slightly modified version of \eqref{eq:dis1} in a convex setting. The algorithm we will consider is the following:
\begin{empheq}[box=\widefbox]{align}
&u_{k+1}=\argmin_{v\in \mathbb{R}^n} \frac{1}{2} |v-u_k+2 \delta T_k u_k +\nabla f(y_k)T_k|^2, \nonumber\\
& ~~\text{s.t.}~~\nabla g_i(y_k)\T v \geq -\alpha g_i(x_k)-\frac{g_i(y_k)-g_i(x_k)-\nabla g_i(y_k)\T \beta u_k}{T_k}, ~~ i\in [n_\text{g}],\nonumber\\
& ~~~~~~~~ y_k=x_k+\beta u_k,\nonumber \\
&x_{k+1}=x_{k}+T_k u_{k+1}.\label{eq:disMod}
\end{empheq}
There are a few differences between \eqref{eq:dis1} and \eqref{eq:disMod}. Instead of evaluating $\nabla g$ at $x_k$, the gradient is evaluated at $x_k+\beta u_k$, which will simplify our proof (both $\nabla f$ and $\nabla g$ are evaluated at the same point). Furthermore the right-hand side of the inequality constraint includes an additional term that can be related to the curvature of $g_i$ as follows
\begin{equation*}
-\frac{g_i(y_k)-g_i(x_k)-\nabla g_i(y_k)\T \beta u_k}{T_k} = \frac{\nabla^2 g_i(\hchange{\xi_k^i}) \beta^2 |u_k|^2}{2T_k},
\end{equation*}
for some $\hchange{\xi_k^i}$ between $x_k$ and $y_k$, where Taylor's theorem has been used in the second step. \hchange{These modifications simplify the analysis and will enable us to derive accelerated non-asymptotic convergence rates for the discrete algorithm \eqref{eq:disMod}. We conjecture (based on numerical experiments, see Sec.~\ref{Sec:NumEx}) that the modification of the right-hand side of the inequality constraints in \eqref{eq:disMod} is an artifact of our analysis.}

\begin{theorem}
Let $f$ be $\mu$-strongly convex and let $g$ be smooth and concave. Let $l(x)=f(x)-{\lambda^*}\T g(x)$ denote the Lagrangian, where $\lambda^*$ denotes an optimal multiplier of \eqref{eq:fundProb}. Then, the iterates of \eqref{eq:disMod} with $T_k=1/\sqrt{L_l}$, $\delta=\sqrt{L_l}/(\sqrt{\kappa_l}+1)$, $\beta=T_k(1-2\delta T_k)$, $\alpha=\delta$, and $u_0=0$ satisfy:
\begin{equation*}
\frac{\mu}{2} |x_k-x^*|^2 \leq l(x_k)-l(x^*)\leq \left(1-\frac{1}{1+\sqrt{\kappa_l}}\right)^k \left(\frac{L_l}{8} |x_0-x^*|^2 +  l(x_0)-l(x^*)\right),
\end{equation*}
for all $k\geq 0$, where $L_l$ denotes the smoothness constant of $l$ and $\kappa_l=L_l/\mu$. \label{Thm:ADT}
\end{theorem}
\begin{proof}
We simplify the presentation by assuming that $L_l=\hchange{T_k}=1$, which is without loss of generality. The proof hinges on the following Lyapunov function:
\begin{equation*}
V(x,u):=\frac{1}{2} |\delta (x-x^*) + (1-\delta) u|^2+ l(x)-l(x^*),
\end{equation*}
which can be readily verified to be positive definite for all $\kappa_l\geq 1$. Moreover, the evolution of $V$ along the iterates \eqref{eq:disMod} is given by \hchange{
\begin{multline*}
 V_{k+1}-V_k=\frac{1}{2} |\delta (x_{k+1}-x_k)+ (1-\delta) (u_{k+1}-u_k)|^2 + l(x_{k+1})-l(x_k)\\
 + (\delta (x_{k+1}-x_k)+ (1-\delta) (u_{k+1}-u_k))\T (\delta (x_k-x^*)+(1-\delta) u_k),
\end{multline*}
where we abbreviated $V(x_k,u_k)$ by $V_k$. We further rewrite $u_{k+1}=(1-2\delta) u_k-\nabla \bar{l}(y_k)$, with $\nabla \bar{l}(y_k):=\nabla f(y_k)- \nabla g(y_k)\lambda_k$ and $\lambda_k$ an optimal multiplier of \eqref{eq:disMod} (see also \eqref{eq:appref} below). This implies that $\delta (x_{k+1}-x_k)+ (1-\delta) (u_{k+1}-u_k)=u_{k+1}-(1-\delta)u_k=-\delta u_k - \nabla \bar{l}(y_k)$ and therefore
\begin{multline*}
    V_{k+1}-V_k=\frac{1}{2} |\delta u_k + \nabla \bar{l}(y_k)|^2 + l(x_{k+1})-l(x_k)\\
    +(-\delta u_k - \nabla \bar{l}(y_k))\T (\delta (x_k-x^*) + (1-\delta) u_k). 
\end{multline*}
By expanding the terms and exploiting that $\beta=1-2\delta$ we arrive at}
\begin{align}
V_{k+1}-V_k = -\delta (1-\frac{3\delta}{2}) |u_k|^2 -\beta u_k\T \nabla \bar{l}(y_k)+\delta  (x^*-x_k)\T \nabla \bar{l}(y_k)\nonumber\\
+\delta^2  u_k\T (x^*-x_k) +\frac{1}{2} |\nabla \bar{l}(y_k)|^2+l(x_{k+1})-l(x_k). \label{eq:decV}
\end{align}

Next, we will relate $l(x_{k+1})$ to $l(y_k)$. To do so, we start by slightly reformulating the minimization in \eqref{eq:disMod} by performing a change of variables, $x=x_k+T_k v$, which yields
\begin{align}
x_{k+1}&=\argmin_{x\in \mathbb{R}^n} \frac{1}{2} |x-y_k+\nabla f(y_k)|^2\label{eq:appref}\\
&\text{s.t.}~\nabla g_i(y_k)\T (x-y_k) \geq (1-\alpha T_k) (g_i(x_k)-g_i(y_k))-\alpha T_k g_i(y_k), ~~i\in [n_\text{g}], \nonumber
\end{align}
where we have used the fact that $T_k(1-2\delta T_k)=\beta$ to simplify the objective function. We note that $x=\alpha (x^*-x_k)+x_k$ is a feasible solution candidate in the minimization \eqref{eq:appref}. This follows from the following reasoning:
\begin{align*}
\nabla g_i(y_k) (\alpha (x^*-x_k)+x_k-y_k)&=(1-\alpha)\nabla g_i(y_k)\T (x_k-y_k)+\alpha \nabla g_i(y_k)\T (x^*-y_k)\\
&\geq (1-\alpha) (g_i(x_k)-g_i(y_k)) + \alpha (g_i(x^*)-\hchange{g_i}(y_k)),
\end{align*}
where concavity of $g$ is used in the second step together with the fact that $T_k=1$ and $g_i(x^*)\geq 0$. As a result, the stationarity condition of \eqref{eq:appref} implies
\begin{equation*}
(x_{k+1}-y_k+\nabla f(y_k))\T (\alpha (x^*-x_k)+x_k-x_{k+1})\geq 0.
\end{equation*}
This inequality can be restated as follows:
\begin{multline*}
(x_{k+1}-y_k+\nabla f(y_k))\T (\alpha (x^*-x_k)-\beta u_k -(x_{k+1}-y_k))\\
= 
- |\nabla \bar{l}(y_k)|^2 -\nabla \bar{l}(y_k)\T (\alpha (x^*-x_k)-\beta u_k) \\
+ \nabla f(y_k)\T (\alpha (x^*-x_k) -\beta u_k) - \nabla f(y_k)\T (x_{k+1}-y_k) \geq 0.
\end{multline*}
Rearranging this inequality yields the following bound:
\begin{multline}
\nabla l(y_k)\T (x_{k+1}-y_k)\leq - |\nabla \bar{l}(y_k)|^2 -\nabla \bar{l}(y_k)\T (\alpha (x^*-x_k)-\beta u_k) \nonumber\\
+ \nabla f(y_k)\T (\alpha (x^*-x_k) -\beta u_k)  -{\lambda^*}\T \nabla g(y_k)\T (x_{k+1}-y_k).
\end{multline}
As a result of this inequality and the smoothness of $l$, we can bound the evolution of $l(x_{k+1})$ as
\begin{align}
l(x_{k+1})&\leq l(y_k)+ \nabla l(y_k)\T (x_{k+1}-y_k) +\frac{1}{2} |\nabla \bar{l}(y_k)|^2\nonumber\\
&\leq l(y_k) - \frac{1}{2} |\nabla \bar{l}(y_k)|^2-\nabla \bar{l}(y_k)\T (\alpha (x^*-x_k)-\beta u_k) \nonumber\\
&~~+ \nabla f(y_k)\T (\alpha (x^*-x_k) -\beta u_k)  -{\lambda^*}\T \nabla g(y_k)\T (x_{k+1}-y_k). \label{eq:apptmp223}
\end{align}
By combining \eqref{eq:apptmp223} with \eqref{eq:decV} we obtain
\begin{align}
V_{k+1}-V_k\leq  -\delta (1-\frac{3\delta}{2}) |u_k|^2 +\nabla f(y_k)\T (\alpha (x^*-x_k) -\beta u_k)  +l(y_k)-l(x_k) \nonumber\\
-{\lambda^*}\T \nabla g(y_k)\T (x_{k+1}-y_k)
+\delta^2  u_k\T (x^*-x_k), \label{eq:gsf}
\end{align}
where we have exploited that $\alpha=\delta$ and the fact that all the terms containing $\nabla \bar{l}(y_k)$ cancel out (in fact the Lyapunov function and the algorithm \eqref{eq:disMod} are specifically engineered in this way). The strong convexity of $f$ implies that 
\begin{multline}
\nabla f(y_k)\T (\delta (x^*-x_k) - \beta u_k)\leq -\delta (f(x_k)-f(x^*)) - \frac{\delta\mu}{2} |x^*-x_k|^2 - \delta\beta\mu (x_k-x^*)\T u_k \nonumber\\
- \frac{\mu}{2} \beta^2 |u_k|^2 +f(x_k)-f(y_k);\label{eq:apptmp3}
\end{multline}
see \citet[App.~A6]{ourWork}, which yields, combined with \eqref{eq:gsf},
\begin{multline*}
V_{k+1}-V_k\leq  -\delta (1-\frac{3\delta}{2}+\frac{\beta^2\mu}{2\delta}) |u_k|^2- \frac{\delta\mu}{2} |x^*-x_k|^2 + \delta(\beta\mu+\delta) (x^*-x_k)\T u_k \\
-\delta (f(x_k)-f(x^*))
-\hchange{{\lambda^*}\T} g(y_k) + {\lambda^*}\T g(x_k) 
-{\lambda^*}\T \nabla g(y_k)\T (x_{k+1}-y_k).
\end{multline*}
This can be rearranged as follows:
\begin{multline*}
\!\!\!V_{k+1}-V_k\leq -\delta V_k -\frac{1}{2} (\delta-\delta^2-\delta^3+\beta^2\mu) |u_k|^2- \frac{\delta\mu-\delta^3}{2} |x^*-x_k|^2 + \delta(\beta\mu+\delta^2) (x^*-x_k)\T u_k\\
+(1-\delta) {\lambda^*}\T g(x_k)-{\lambda^*}\T g(y_k)
-{\lambda^*}\T \nabla g(y_k)\T (x_{k+1}-y_k).
\end{multline*}
We further note that the inequality constraint in \eqref{eq:appref} implies
\begin{equation*}
-\nabla g_i(y_k)\T (x_{k+1}-y_k)+(1-\delta) g(x_k)-g(y_k)\leq 0,
\end{equation*}
and therefore we obtain
\begin{equation*}
V_{k+1}-V_k\leq   -\delta V_k -\frac{1}{2} (\delta-\delta^2-\delta^3+\beta^2\mu) |u_k|^2- \frac{\delta\mu-\delta^3}{2} |x^*-x_k|^2 + \delta(\beta\mu+\delta^2) (x^*-x_k)\T u_k.
\end{equation*}
We note that $\beta \mu=(1-\mu) \delta^2$ and therefore $\beta\mu+\delta^2=\delta^2 (2-\mu)$. We now apply Young's inequality to the term $(x^*-x_k)\T u_k$ and conclude
\begin{equation*}
-\frac{1}{2}(\delta-\delta^2-\delta^3+\beta^2\mu) |u_k|^2- \frac{\delta\mu-\delta^3}{2} |x^*-x_k|^2 + \delta^3 (2-\mu) (x^*-x_k)\T u_k\leq 0,
\end{equation*} 
for all $\kappa_l\geq 1$, $x^*-x_k$, and $u_k$. This yields $V_{k+1}-V_k\leq -\delta V_k$ and concludes the proof.
\end{proof}

Thm.~\ref{Thm:ADT} requires knowledge of the smoothness constant of $l$, which can be restrictive for certain applications. However, by following the same argument as in App.~\ref{App:Sec1} we obtain the following corollary that proves a $\mathcal{O}(\log(1/\varepsilon)/\sqrt{\varepsilon})$ convergence rate if $L_l$ is unknown.
\begin{corollary}\label{Cor:SAR}
Let the function $f$ be $\mu$-strongly convex and $L$-smooth, let $g$ be concave and $L_\text{g}$-smooth, and let 
\begin{equation*}
B:=\max_{x\in C} |\nabla f(x)|^2/(2\mu).
\end{equation*}
Then, the iterate $x_N$ of \eqref{eq:disMod} with $T_k=1/\sqrt{L_l}$, $\delta=\sqrt{L_l}/(\sqrt{\kappa_l}+1)$, $\beta=T_k(1-2\delta T_k)$, $\alpha=\delta$, and $L_l=L+B L_\text{g}/\varepsilon$ satisfies $|x_N-x^*|\leq \epsilon$, where
\begin{equation*}
N\!\geq\!2\sqrt{\frac{L\!+\!B L_\text{g}/\varepsilon}{\mu}} \!\left( \!\!2\log(1/\varepsilon) \!+\! \log\left(\frac{5L\!+\!5B L_\text{g}/\varepsilon}{4\mu}\right) \!+\! 2 \log(|x^*-x_0|)\!\right)\!=\!\mathcal{O}\!\left(\!\frac{\log(1/\varepsilon)}{\sqrt{\varepsilon}}\!\right),
\end{equation*}
and $x^*$ denotes the minimizer of
$\min_{x\in \mathbb{R}^n} f(x)~\text{s.t.}~g(x)\geq - \varepsilon$.
\end{corollary}

\hchange{Thm.~\ref{Thm:ADT} further requires strong convexity of the objective function. We also obtain an accelerated rate when $f$ is convex and smooth, but not strongly convex. This is summarized in the following corollary, which adds a small, but well-chosen, regularization to $f$ to render the problem strongly convex.}
\begin{corollary}\label{Cor:SCA}
    \hchange{Let the objective function $f$ be convex and $L$-smooth, let $g$ be concave and $L_g$-smooth, and let $\nabla^2 f(x^*)-\sum_{i=1}^{n_g}\lambda^*_i \nabla^2 g_i(x^*)$ be positive definite, where $\lambda^*$ denotes an optimal multiplier of \eqref{eq:fundProb}. Let $L_l$ denote the smoothness constant of the Lagrangian $l(x)=f(x)-g(x)\T \lambda^*$ and consider the modified objective function \begin{equation*}
    f_\mu(x)=f(x)+\mu |x-x_0|^2/2\quad \text{with} \quad \mu=\frac{16 L_l \mathrm{log}(N)^2}{N^2}.    
    \end{equation*} 
    Then, the iterate $x_N$ of \eqref{eq:disMod}, when applied to $f_\mu$ with $\kappa_l=L_l/\mu$, $T_k=1/\sqrt{2L_l}$, $\delta=\sqrt{2L_l}/(\sqrt{2\kappa_l}+1)$, $\beta=T_k(1-2\delta T_k)$, $\alpha=\delta$, and $u_0=0$ satisfies
    \begin{equation*}
        l(x_N)-l(x^*) \leq \mathcal{O}\left( \frac{\mathrm{log}(N)^2}{N^2}\right).
    \end{equation*}}
\end{corollary}
\begin{proof}
\hchange{We consider the regularized Lagrangian $l_\mu(x):=f_\mu(x)-g(x)\T \lambda_\mu^*$, where $\lambda_\mu^*$ denotes the optimal multiplier of \eqref{eq:fundProb} with $f$ replaced by $f_\mu$, and define $\lambda^*:=\lim_{\mu \downarrow 0} \lambda_\mu^*$. We further introduce the function $d(\lambda,\mu):=\min_{x\in\mathbb{R}^n} f_\mu(x)-\lambda\T g(x)$ and note that $\nabla_\lambda d(\lambda,\mu)$ is given by
    \begin{equation*}
        \nabla_\lambda d(\lambda,\mu)=-g(x_\mu^*), \quad \text{where~$x_\mu^*$~is defined via}~ \nabla f_\mu(x_\mu^*)-\nabla g(x_\mu^*) \lambda =0.
    \end{equation*}
    Due to the fact that $H:=\nabla^2 f(x^*)-\sum_{i=1}^{n_g} \lambda_i^* \nabla^2 g_i(x^*)$ is positive definite, we conclude from the implicit function theorem that $x^*_\mu$ is a continuously differentiable function of $\mu$ and $\lambda$ in a neighborhood of $\mu=0$, $\lambda=\lambda^*$, and has derivatives
    \begin{equation*}
        \frac{\partial x^*_{\mu}}{\partial \lambda}=H^{-1} \nabla g(x^*), \qquad \frac{\partial x^*_{\mu}}{\partial \mu} = -H^{-1} (x^*-x_0)
    \end{equation*}
    at $\mu=0, \lambda=\lambda^*$. This means that $\nabla_\lambda d(\lambda,\mu)$ is continuously differentiable in both arguments and that $\nabla^2_\lambda d(\lambda,\mu)$ is strictly negative definite in a neighborhood around $\mu=0$, $\lambda=\lambda^*$. We now apply Robinson's implicit function theorem, see, e.g., \cite[Thm.~2B.1]{Donchev},
    which yields
    \begin{equation*}
        |\lambda_\mu^*-\lambda^*| \leq |x^*-x_0|c_1 \mu,
    \end{equation*}
    for all $\mu$ in a neighborhood of the origin and where $c_1\geq 0$ is constant (the scaling with $|x^*-x_0|$ arises from the derivative $\partial x_\mu^*/\partial \mu$). This means that for large enough $N$, the modified Lagrangian $l_\mu$ is $2L_l$-smooth. Hence, as a result of Thm.~\ref{Thm:ADT}, we conclude
\begin{equation*}
l_\mu(x_N)-l_\mu(\tilde{x}^*) \leq \exp\left(-\frac{1}{1+\sqrt{2\kappa_l}} N\right) \left(\frac{5L_l}{4} |x_0-\tilde{x}^*|^2\right),
\end{equation*}
where we have used the fact that $e^{-\xi} \geq 1-\xi$ for all $\xi\in \mathbb{R}$, as well as the $2L_l$-smoothness of $l_\mu$. In addition, $\tilde{x}^*$ denotes the minimizer of \eqref{eq:fundProb} with $f$ replaced by $f_\mu$. As a result of the choice of $\mu$ we note that $\sqrt{2\kappa_l}=N/(2\mathrm{log}(N)\sqrt{2})$ and $1+\sqrt{2\kappa_l}\leq N/(2\mathrm{log}(N))$ for large $N$. Hence, we can simplify the right-hand side of the previous inequality for large $N$ as follows
\begin{equation*}
    l_\mu(x_N)-l_\mu(\tilde{x}^*) \leq \exp\left(-2\mathrm{log}(N) \right) \left(\frac{5L_l}{4} |x_0-\tilde{x}^*|^2\right) \leq \frac{1}{N^2} \left(\frac{5L_l}{4} |x_0-\tilde{x}^*|^2\right).
\end{equation*}
We now relate $l_\mu$ to $l$ in the following way:
\begin{align*}
    (f(x_N)-g(x_N)\T \lambda_\mu^*)-(f(x^*)-g(x^*)\T \lambda_\mu^*) &\leq l_\mu(x_N) - l_\mu(x^*) + \frac{\mu}{2} |x^*-x_0|^2\\
    &\leq l_\mu(x_N) - l_\mu(\tilde{x}^*) + \frac{\mu}{2} |x^*-x_0|^2,
\end{align*}
where we have used the fact that $l_\mu(\tilde{x}^*)\leq l_\mu(x ^*)$ in the second step. From the construction of $f_\mu$ we note that $|\tilde{x}^*-x_0|\leq |x^*-x_0|$, which yields, when combining the previous two inequalities,
\begin{equation*}
    l(x_N)-l(x^*) \leq \frac{1}{N^2} \left(\frac{5L_l}{4} |x_0-x^*|^2\right)+ \frac{\mu}{2} |x^*-x_0|^2+(g(x_N)-g(x^*))\T (\lambda_\mu^*-\lambda^*).
\end{equation*}
The result follows from the choice of $\mu$ (i.e., $\mu=\mathcal{O}(\mathrm{log}(N)^2/N^2)$), and the fact that $|\lambda_\mu^*-\lambda^*|\leq |x^*-x_0| c_1 \mu$.}
\end{proof}

%% file: example.tex
\begin{figure}
\includegraphics[scale=.85]{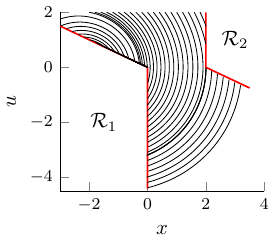}
\includegraphics[scale=.85]{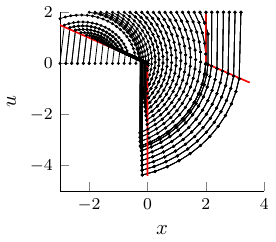}
\includegraphics[scale=.85]{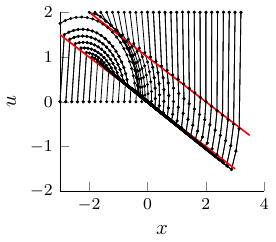}
\caption{The first panel shows trajectories resulting from \eqref{eq:mde1} (with parameters $\alpha=0.5, \delta=0.1, \beta=0, \epsilon=0$). The boundaries of $\mathcal{R}_1$ and $\mathcal{R}_2$ are highlighted in red. The second panel shows the results from the discretization \eqref{eq:dis1} with $T_k=T=0.1$, while the third panel shows the results from the discretization \eqref{eq:disMod} with $T_k=T=0.1$. An important difference between \eqref{eq:dis1} and \eqref{eq:disMod} lies in the fact that only violated constraints are considered in \eqref{eq:dis1}, whereas \eqref{eq:disMod} includes all constraints. This is indicated by the red lines, which denote $\mathcal{R}_1$, $\mathcal{R}_2$ in the second panel and $\gamma_1(x,u)=0$, $\gamma_2(x,u)=0$ in the third panel.}
\label{Fig:SimEx}
\end{figure}

\begin{figure}
\includegraphics[scale=.72]{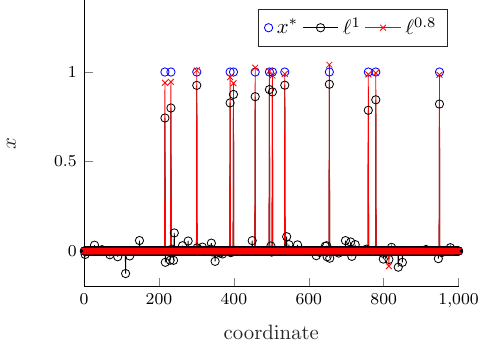}
\setlength{\figurewidth}{.45\columnwidth}
\setlength{\figureheight}{.32\columnwidth}
\scalebox{0.8}{\input{media/resultCompressedSensingObj.tikz}}
\caption{ The left panel shows the solution vector of the compressed sensing problem with $\ell^1$ and $\ell^{0.8}$ regularization. The right panel shows the evolution of the objective function for the different methods. We note that Alg.~\ref{Alg:ImageDenoising}, Alg.~\ref{Alg:ImageDenoising2}, and accelerated projected gradients converge at a similar rate, which is much faster than gradient descent. We applied the following settings for Alg.~\ref{Alg:ImageDenoising} and Alg.~\ref{Alg:ImageDenoising2}: $\alpha_k=2/(k+3)$, $\delta_k=3/(2(k+3))$, $\beta_k=T(1-2\delta_k T)$ (see Tab.~\ref{Tab:params}) with $T=1.8$ and $T=2$, respectively. Accelerated gradient descent corresponds to the algorithm from \citet[p.~78, Constant Step Scheme I]{NesterovIntro}. The corresponding trajectories for Alg.~\ref{Alg:ImageDenoising} and Alg.~\ref{Alg:ImageDenoising2} for $p<1$ are similar to $p=1$ and are shown in Fig.~\ref{Fig:SimExCS2}.}
\label{Fig:SimExCS}
\end{figure}

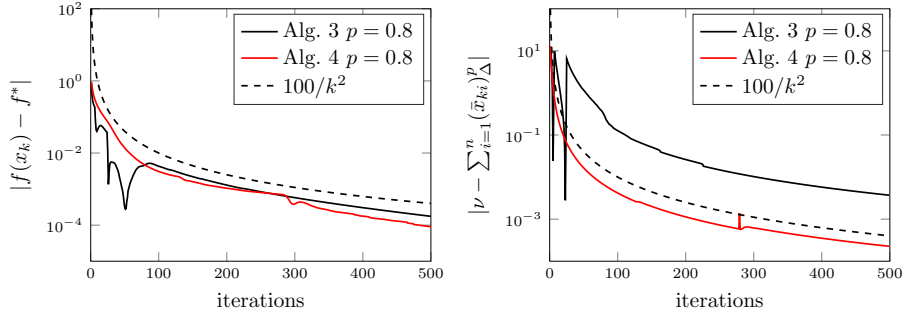
\begin{figure}
\setlength{\figurewidth}{.45\columnwidth}
\setlength{\figureheight}{.32\columnwidth}
\scalebox{0.8}{\input{media/resultCompressedSensingObjp0p8.tikz}}
\scalebox{0.8}{\input{media/resultCompressedSensingConstp0p8.tikz}}
\caption{ The figure shows the trajectories of Alg.~\ref{Alg:ImageDenoising} and Alg.~\ref{Alg:ImageDenoising2} applied to the compressed sensing problem with $\ell^{0.8}$ regularization. The left panel shows the evolution of the objective function for the different methods, whereas the right panel shows the value of the constraint violation. We applied the following settings for Alg.~\ref{Alg:ImageDenoising} and Alg.~\ref{Alg:ImageDenoising2}: $\alpha_k=2/(k+3)$, $\delta_k=3/(2(k+3))$, $\beta_k=T(1-2\delta_k T)$ (see Tab.~\ref{Tab:params}) with $T=1$ and $\Delta=1e-3$.}
\label{Fig:SimExCS2}
\end{figure}

\section{Numerical Examples}\label{Sec:NumEx}
The following section is divided into two parts. The first part illustrates the dynamics of \eqref{eq:mde1} and the discretization via \eqref{eq:dis1} and \eqref{eq:disMod} on a one-dimensional example and is intended to provide insights concerning the non-smooth dynamics, as well as the discretization. The second part applies \eqref{eq:dis1} and \eqref{eq:disMod} to (nonconvex) compressed sensing and large-scale sparse regression problems. As we will see, our algorithm recovers state-of-the-art performance for convex relaxations, while also handling nonconvex sparsity constraints in a seamless manner (traditional projection-based methods cannot be easily extended to this setting). 

\subsection{Illustrative example}
In order to plot trajectories in the phase space we choose $f(x)=(x+2)^2/2$ and $g(x)=(x,-x+2)$, where $x$ is scalar. Each constraint $g_i(x)\geq 0$ and its corresponding velocity constraint $\gamma_i(x)\geq 0$ induces a region,
\begin{equation*}
\mathcal{R}_i:=\{(x,u)\in \mathbb{R}^2~|~g_i(x)\leq 0, \gamma_i(x,u)\leq 0\},
\end{equation*}
in the phase space, $i=1,2$, where trajectories are either non-smooth or slide along the boundary of $\mathcal{R}_i$. Outside of $\mathcal{R}_i$, the trajectories follow the smooth motion \eqref{eq:diff1}. The first panel in Fig.~\ref{Fig:SimEx} shows the continuous-time trajectories \eqref{eq:mde1} along with $\mathcal{R}_1$ and $\mathcal{R}_2$. For a given $(x(t_0),u(t_0)^-)$ an impact happens if $g_i(x(t_0))\leq 0$ and $\gamma_i(x(t_0),u(t_0))^-<0$, which ensures that $\gamma_i(x(t_0),u(t_0))^+\geq -\epsilon \gamma_i(x(t_0),u(t_0))^-$. In our example only the case $\gamma_i^+=-\epsilon \gamma_i^-$ occurs, as there are no impacts where more than one constraint participates ($\mathcal{R}_1$ and $\mathcal{R}_2$ are disjoint). The coefficient of restitution $\epsilon$ therefore determines the velocity after impact. For $\epsilon=0$ trajectories end up at the boundary of the set $\mathcal{R}_i$, whereas for $\epsilon >0$ they will leave $\mathcal{R}_i$ (in case of impact). If $g_i(x(t_0))\leq 0$, $\gamma_i(x(t_0),u(t_0))^-=0$, no impact occurs, ($u(t_0)=u(t_0)^-=u(t_0)^+$), and trajectories either leave $\mathcal{R}_i$ or slide along its boundary. This depends on the contribution of the unconstrained dynamics, that is, on the vector
$v_\text{uc}(t_0):=(u(t_0),-2\delta u(t_0) - \nabla f(x(t_0)+\beta u(t_0))).$
If $v_\text{uc}(t_0)$ points outwards, trajectories will leave $\mathcal{R}_i$ and follow the unconstrained motion ($\diff \lambda_i=0$). If $v_\text{uc}(t_0)$ points inwards, there will be a contribution from $\diff \lambda_i=\lambda_i(t_0) \dt$, which ensures that trajectories slide along the boundary of $\mathcal{R}_i$.


The second panel in Fig.~\ref{Fig:SimEx} shows the trajectories resulting from a discretization according to \eqref{eq:dis1} with $T_k=T=0.1$. We can clearly see the consequences of including constraints on the velocity level: Trajectories may become infeasible, since constraints enter \eqref{eq:dis1} only once they are violated. Nevertheless, even for large time steps $T_k=T$ (up to $T\approx 1.8$), trajectories converge to the unique minimizer of our problem. The third panel in Fig.~\ref{Fig:SimEx} shows trajectories from the discretization according to \eqref{eq:disMod} with $T=0.1$. We note that in \eqref{eq:disMod} each constraint is permanently active (that is, $I=[n_\text{g}]$ in \eqref{eq:defVa}), which explains the contrast between the trajectories in the second and third panel. In this example the dynamics \eqref{eq:disMod} are more robust with respect to large time steps, and in fact, convergence can still be observed for $T=3$.

\subsection{Nonconvex compressed sensing and image reconstruction}
We consider the following $\ell^p$-regularized inverse problem:
\begin{equation}
\min_{x\in \mathbb{R}^n} \frac{1}{2} |Ax-b|^2 \quad \text{s.t.} \quad |x|_p^p\leq \nu, \label{eq:inverse} 
\end{equation}
where $|x|_p$ refers to the $\ell^p$ ``norm" (we explicitly allow for $0<p\leq 1$). This has numerous applications in machine learning, statistics, and signal processing \citep[see, e.g.,][]{ElementsOfSL}. The traditional convex approach for solving such an inverse problem is to set $p=1$ and to leverage the fact that projections onto the $\ell^1$ ball have closed-form solutions. This yields both accelerated and non-accelerated gradient descent schemes, which in the setting of \eqref{eq:inverse} are also known under the name of iterative shrinkage-thresholding (ISTA) and fast iterative shrinkage-thresholding (FISTA) \citep[see, e.g.,][]{FISTA}. However, when $p<1$, projections onto the $\ell^p$ ``norm" ball no longer have closed-form solutions and it is unclear how to generalize projected gradient algorithms to this setting. Nonetheless the setting $p<1$ can be handled easily with our algorithms, as we highlight next.

In order to handle the absolute value and the $p$th power, we add slack variables, $\bar{x}\in \mathbb{R}^n$ and reformulate \eqref{eq:inverse} as
\begin{equation}
\min_{(x,\bar{x})\in \mathbb{R}^{2n}} \frac{1}{2} |A x - b|^2 \quad \text{s.t.} \quad -\bar{x} \leq x \leq \bar{x}, \quad \sum_{i=1}^n (\bar{x}_i)_\Delta^p \leq \nu, \label{eq:imp}
\end{equation}
where $(\cdot)_\Delta^p: \mathbb{R} \rightarrow \mathbb{R}$ is continuously differentiable and approximates $x^p$ for $x>0$ ($x^p$, $p<1$ is nondifferentiable at the origin). The approximation $(\cdot)_\Delta^p$ depends on the approximation parameter $\Delta>0$ and is defined as
\begin{equation*}
(x)_\Delta^p:=\begin{cases} x^p -\Delta^p (1-p) & x\geq \Delta, \\
p \Delta^{p-1} x, & x< \Delta.
\end{cases}
\end{equation*}
For $p=1$ our approximation recovers the constraint $|x|_1\leq \nu$ exactly for any $\Delta>0$. The approximation is depicted in Fig.~\ref{Fig:NCi} for $p=0.6$ and $\Delta=0.01$ and highlights that $x^p$ is approximated well even for modest values of $\Delta$. In the numerical examples $\Delta$ is typically set to $10^{-6}$, which yields an excellent agreement between $x^p$ and $(x)^p_\Delta$.

Despite the fact that the constraint in \eqref{eq:imp} is nonlinear and nonconvex for $p<1$, the optimization in \eqref{eq:dis1} and \eqref{eq:disMod} can be carried out in closed form, which yields the two algorithms Alg.~\ref{Alg:ImageDenoising} and Alg.~\ref{Alg:ImageDenoising2} stated in App.~\ref{App:Example2}.

\paragraph{Nonconvex compressed sensing example:}
In the first example, each element of $A\in \mathbb{R}^{100 \times 1000}$ is sampled from a standard normal distribution. The vector $b$ is set to $A x^* + n/2$, where the components of $n\in \mathbb{R}^{100}$ are sampled from a standard normal and $x^*$ is a vector that contains zeros everywhere except for 13 randomly chosen entries that are set to one. This gives rise to a challenging and ill-conditioned optimization problem that includes 1000 decision variables. The left panel in Fig.~\ref{Fig:SimExCS} compares the results computed by our algorithm for $p=1$ and $p=0.8$, whereas the right panel (solid lines) compares our approach to projected gradient descent and accelerated projected gradient descent for $p=1$. We note: i) the quality of the reconstruction for $p=1$ is significantly worse compared to $p=0.8$ (the parameter $\nu$ was tuned with five-fold cross validation, which yielded $\nu\approx 13$ in both cases) and ii) our algorithms decrease the objective function at a similar rate as accelerated projected gradient for $p=1$. Aside from the computation of the gradient, all algorithms have the same $\mathcal{O}(n \text{log}(n))$ complexity  per iteration, which is determined by a sorting operation that is used for projecting the iterates onto the $\ell^1$ ball or solving the minimization in \eqref{eq:dis1} or \eqref{eq:disMod}, respectively. It is important to note that for both our algorithms the per-iteration complexity is independent of $p$. Fig.~\ref{Fig:SimExCS2} shows the trajectories of Alg.~\ref{Alg:ImageDenoising} and Alg.~\ref{Alg:ImageDenoising2} for $p=0.8$ and highlights a convergence rate, both in terms of function value and in terms of constraint violation, on the order of $1/k^2$ also in the nonconvex case.

\paragraph{Nonconvex image reconstruction example:}
The second example consists of an image reconstruction problem taken from \citet{FISTA}, where $A=RW\in \mathbb{R}^{n\times n}, n=65536$, with $R$ representing a Gaussian blur operator, $W$ the inverse of a three-stage Haar wavelet transform, and $\nu=6\cdot 10^{3}$. The problem given by \eqref{eq:imp} is of considerable size and includes 131,072 decision variables and 131,073 constraints. Similar to the previous example, our approach is on par with the performance of accelerated gradient descent for $p=1$, as shown in Fig.~\ref{Fig:denoising}. However, we are also able to solve problems with $p<1$, as highlighted in Fig.~\ref{Fig:Outimage}. Fig.~\ref{Fig:Outimage} compares the resulting reconstruction of accelerated gradient descent with $p=1$ compared to our reconstruction $p=0.8$ after 100 iterations, whereby the latter has visibly fewer artifacts.
%
We found that the choice of the damping parameters $\delta_k$ and $\beta_k$ has a significant effect on the convergence rate. Our choice $\delta_k=3/(2(k+3)), \beta_k=1-2\delta_k$, $\alpha_k=2/(k+3)$, and $T=1$ as motivated in Tab.~\ref{Tab:params}, is not optimized and a different schedule might lead to additional improvements.

Further details, including the derivation of Alg.~\ref{Alg:ImageDenoising} and Alg.~\ref{Alg:ImageDenoising2} from \eqref{eq:dis1} and \eqref{eq:disMod}, are included in App.~\ref{App:Example2}.

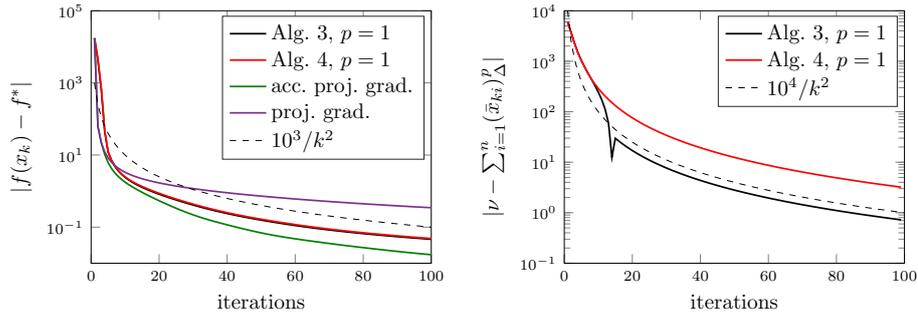
\begin{figure}
\setlength{\figurewidth}{.45\columnwidth}
\setlength{\figureheight}{.32\columnwidth}
\scalebox{0.8}{
\input{media/imageDenoisingTraj.tikz}
}
\scalebox{0.8}{
\input{media/imageDenoisingConst.tikz}
}
\caption{ The figure on the left shows the decrease in the objective function as a function of the iterations for the different algorithms. We note that Alg.~\ref{Alg:ImageDenoising}, Alg.~\ref{Alg:ImageDenoising2}, and accelerated projected gradients converge at a similar rate, which is substantially faster than gradient descent. We applied the following settings for Alg.~\ref{Alg:ImageDenoising} and Alg.~\ref{Alg:ImageDenoising2}: $\alpha_k=2/(k+3)$, $\delta_k=3/(2(k+3))$, $\beta_k=T(1-2\delta_k T)$ (see Tab.~\ref{Tab:params}) with $T=1$. Accelerated gradient descent corresponds to the algorithm by \citet[p.~78, Constant Step Scheme I]{NesterovIntro}. The figure on the right shows how constraint violations decrease as a function of the number of iterations. The black dashed line indicates a rate of $\mathcal{O}(1/k^2)$ as a reference. The corresponding trajectories of Alg.~\ref{Alg:ImageDenoising} and Alg.~\ref{Alg:ImageDenoising2} for $p<1$ are similar to $p=1$.}
\label{Fig:denoising}
\end{figure}

\begin{figure}
    \begin{minipage}[l]{.25\columnwidth}
    \centering
    \includegraphics[width=\textwidth]{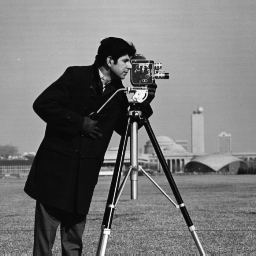}
    \end{minipage}%
    \begin{minipage}{.25\columnwidth}
    \centering
    \includegraphics[width=\textwidth]{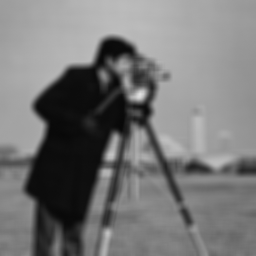}
    \end{minipage}%
    \begin{minipage}{.25\columnwidth}
    \includegraphics[width=\textwidth]{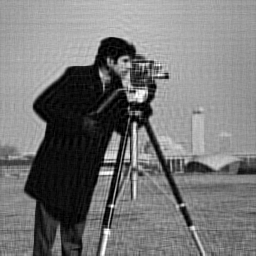}
    \end{minipage}%
    \begin{minipage}{.25\columnwidth}
    \includegraphics[width=\textwidth]{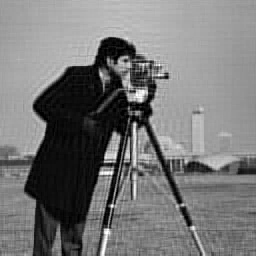}
    \end{minipage}
    \caption{From left to right: Original image, blurred and noisy image, output after 100 iterations with accelerated gradient descent ($p=1$), output after 100 iterations with Alg.~\ref{Alg:ImageDenoising2} ($p=0.8$). Both the results with $p=1$ and $p=0.8$ lead to reasonable reconstructions, however, $p=0.8$ has fewer artifacts.}
    \label{Fig:Outimage}
\end{figure}


%% file: media/resultCompressedSensingObj.tikz
%
%
\definecolor{mycolor1}{rgb}{0.00000,0.44700,0.74100}%
\definecolor{mycolor2}{rgb}{0.85000,0.32500,0.09800}%
\definecolor{mycolor3}{rgb}{0,0.5,0}%
\definecolor{mycolor4}{rgb}{0.49400,0.18400,0.55600}%
\begin{tikzpicture}

\begin{axis}[%
width=0.951\figurewidth,
height=\figureheight,
at={(0\figurewidth,0\figureheight)},
scale only axis,
xmin=0,
xmax=300,
xlabel style={font=\color{white!15!black}},
xlabel={iterations},
ymode=log,
ymin=1e-07,
ymax=1,
yminorticks=true,
ylabel style={font=\color{white!15!black}},
ylabel={$|f(x_k)-f^*|$},
axis background/.style={fill=white},
legend style={legend cell align=left, align=left, draw=white!15!black,legend columns=2,at={(1.15,1.15)}}
]
\addplot [color=black,thick]
  table[row sep=crcr]{%
1	0.913566325963014\\
2	0.0822718755961752\\
3	0.0840067234429931\\
4	0.0924387640467611\\
5	0.0853478048977777\\
6	0.0761111840383203\\
7	0.0664331078401359\\
8	0.0583222544807385\\
9	0.0518540849994055\\
10	0.0464913734371344\\
11	0.0417465570401864\\
12	0.0376938921604789\\
13	0.0341674082848015\\
14	0.030961456678977\\
15	0.0281074630707992\\
16	0.0255259226383316\\
17	0.0230805355907026\\
18	0.0210070859977881\\
19	0.0191855784367933\\
20	0.0175602012582213\\
21	0.0161563493816411\\
22	0.0148865711129868\\
23	0.0137504789764779\\
24	0.0128177312220581\\
25	0.0120028288874981\\
26	0.0113039871386189\\
27	0.010659164058595\\
28	0.0100686823332543\\
29	0.00951524794684783\\
30	0.00898696437786825\\
31	0.00847681353165276\\
32	0.00798063732268347\\
33	0.0074987106632315\\
34	0.00702956358297195\\
35	0.00657246646073621\\
36	0.00613971279269329\\
37	0.00572325597396749\\
38	0.00531754229493773\\
39	0.00493734018883575\\
40	0.00457307955134436\\
41	0.00422301913093401\\
42	0.00388735446859133\\
43	0.00356928807000632\\
44	0.00326476166092348\\
45	0.0029749714658118\\
46	0.00269621067420573\\
47	0.00242968958701914\\
48	0.00218138926720297\\
49	0.00194683391465651\\
50	0.00173714245035364\\
51	0.00154664752129613\\
52	0.00137768239732646\\
53	0.00122296576375371\\
54	0.00107904670843144\\
55	0.000944921129805227\\
56	0.0008215698534197\\
57	0.000708016904034377\\
58	0.00060321794620381\\
59	0.000505624496271562\\
60	0.000414690551386676\\
61	0.000329310183330793\\
62	0.000254205990020019\\
63	0.000187764463365859\\
64	0.000127250668829937\\
65	8.78801005600883e-05\\
66	5.87918801140883e-05\\
67	3.95371632420623e-05\\
68	2.47533549480716e-05\\
69	1.37688826625133e-05\\
70	4.82437472078875e-06\\
71	1.88990850023916e-06\\
72	6.94210281211929e-06\\
73	9.28079683146753e-06\\
74	8.63435475697769e-06\\
75	6.85310278008156e-06\\
76	4.56488815346601e-06\\
77	2.36000015710683e-06\\
78	4.95487812742119e-07\\
79	9.14234002029855e-07\\
80	1.74455897336209e-06\\
81	1.93944688650464e-06\\
82	1.54827458895192e-06\\
83	6.89459270982254e-07\\
84	4.72756577405385e-07\\
85	1.7508941776356e-06\\
86	2.97253231724449e-06\\
87	4.01209547210447e-06\\
88	4.8049935982902e-06\\
89	5.26437994476884e-06\\
90	5.48834255955852e-06\\
91	5.59329244514354e-06\\
92	5.66924914273403e-06\\
93	5.7710902349687e-06\\
94	5.92464769256975e-06\\
95	5.89651744123572e-06\\
96	5.85703965916631e-06\\
97	5.932014271905e-06\\
98	6.16976378564355e-06\\
99	6.46832468876365e-06\\
100	6.86049532997845e-06\\
101	7.30545071401037e-06\\
102	7.76531012235387e-06\\
103	8.20061497231166e-06\\
104	8.51345855446776e-06\\
105	8.70064237046882e-06\\
106	8.78874490311491e-06\\
107	8.80134124075713e-06\\
108	8.75691708696405e-06\\
109	8.67071454852247e-06\\
110	8.55577150097545e-06\\
111	8.42246921398636e-06\\
112	8.27781611893296e-06\\
113	8.12539712383847e-06\\
114	7.96602595983052e-06\\
115	7.79872041741566e-06\\
116	7.62170325471333e-06\\
117	7.43331899354192e-06\\
118	7.23281434088664e-06\\
119	7.02288439518358e-06\\
120	6.80273961357544e-06\\
121	6.57727806921676e-06\\
122	6.35197748384986e-06\\
123	6.13256012501999e-06\\
124	5.91294002205787e-06\\
125	5.71033605437608e-06\\
126	5.53206072761731e-06\\
127	5.37879825410377e-06\\
128	5.24859165858e-06\\
129	5.13870984277139e-06\\
130	5.04570080529527e-06\\
131	4.965162667913e-06\\
132	4.89202179822061e-06\\
133	4.82130053065018e-06\\
134	4.74891572142724e-06\\
135	4.67214866061404e-06\\
136	4.58973532858377e-06\\
137	4.50175217235008e-06\\
138	4.40979809329233e-06\\
139	4.31646432806371e-06\\
140	4.22431798625597e-06\\
141	4.13524507373014e-06\\
142	4.05023957854994e-06\\
143	3.96950034338341e-06\\
144	3.89265563153877e-06\\
145	3.8189946579897e-06\\
146	3.74765348873634e-06\\
147	3.67774436525487e-06\\
148	3.60843785296178e-06\\
149	3.53901604907385e-06\\
150	3.4689143192778e-06\\
151	3.39775951492538e-06\\
152	3.32540027100144e-06\\
153	3.2519179377433e-06\\
154	3.17760895379108e-06\\
155	3.10293851032116e-06\\
156	3.02847505990762e-06\\
157	2.95078358230998e-06\\
158	2.87258432863594e-06\\
159	2.79558250210199e-06\\
160	2.72129581996641e-06\\
161	2.65088641538021e-06\\
162	2.58513981029468e-06\\
163	2.52446405093511e-06\\
164	2.468887221602e-06\\
165	2.41808433542115e-06\\
166	2.37145294861452e-06\\
167	2.32822564368625e-06\\
168	2.28759059812435e-06\\
169	2.24879367914828e-06\\
170	2.21120706970016e-06\\
171	2.17436122574959e-06\\
172	2.13794527321315e-06\\
173	2.1017854192739e-06\\
174	2.06581210395688e-06\\
175	2.03002516237834e-06\\
176	1.99446326787974e-06\\
177	1.95918066312178e-06\\
178	1.92423166960227e-06\\
179	1.88966205035522e-06\\
180	1.85550575943719e-06\\
181	1.82178551772199e-06\\
182	1.78851570710545e-06\\
183	1.75570619135364e-06\\
184	1.72336588908154e-06\\
185	1.69150527068191e-06\\
186	1.6601373837655e-06\\
187	1.62927744858099e-06\\
188	1.59894141204056e-06\\
189	1.56914405180508e-06\\
190	1.53989725566565e-06\\
191	1.51120898659891e-06\\
192	1.4830832231038e-06\\
193	1.45552090127007e-06\\
194	1.42852163736214e-06\\
195	1.40208582975233e-06\\
196	1.37621665236183e-06\\
197	1.35092146688782e-06\\
198	1.32621228718464e-06\\
199	1.3021051038075e-06\\
200	1.27861808756589e-06\\
201	1.25576890150011e-06\\
202	1.23357152424959e-06\\
203	1.21203309119255e-06\\
204	1.19115127517576e-06\\
205	1.17091264831282e-06\\
206	1.15129230815209e-06\\
207	1.1322548435222e-06\\
208	1.11375649660284e-06\\
209	1.09574819193506e-06\\
210	1.07817898090719e-06\\
211	1.06099941476687e-06\\
212	1.04416440963821e-06\\
213	1.02763528970519e-06\\
214	1.01138085995423e-06\\
215	9.95377532848896e-07\\
216	9.79608679852959e-07\\
217	9.64063475225943e-07\\
218	9.48735532401226e-07\\
219	9.33621606358237e-07\\
220	9.18720563155943e-07\\
221	9.04032720373642e-07\\
222	8.89559563958982e-07\\
223	8.7530376702543e-07\\
224	8.61269385842196e-07\\
225	8.47462091895998e-07\\
226	8.33889313480926e-07\\
227	8.20560194776318e-07\\
228	8.0748532761514e-07\\
229	7.94676255306236e-07\\
230	7.82144787842692e-07\\
231	7.69902191850164e-07\\
232	7.57958332768711e-07\\
233	7.46320845648046e-07\\
234	7.34994403922551e-07\\
235	7.23980140540127e-07\\
236	7.13275260286195e-07\\
237	7.02872864594833e-07\\
238	6.92761996189291e-07\\
239	6.82927896383106e-07\\
240	6.7335245730691e-07\\
241	6.64014841840477e-07\\
242	6.54892236291033e-07\\
243	6.45960695035315e-07\\
244	6.37196032524795e-07\\
245	6.28574715594472e-07\\
246	6.20074709725173e-07\\
247	6.11676235137519e-07\\
248	6.03362394368916e-07\\
249	5.95119640277655e-07\\
250	5.86938064099825e-07\\
251	5.78811493561794e-07\\
252	5.70737404522945e-07\\
253	5.62716660362089e-07\\
254	5.54753105110224e-07\\
255	5.46853044082127e-07\\
256	5.39024652251516e-07\\
257	5.31277352220846e-07\\
258	5.2362120320313e-07\\
259	5.16066337779772e-07\\
260	5.08622476066978e-07\\
261	5.01298538516798e-07\\
262	4.9410236765534e-07\\
263	4.87040560264752e-07\\
264	4.80118401515774e-07\\
265	4.73339885473899e-07\\
266	4.66707801381923e-07\\
267	4.60223861954018e-07\\
268	4.53888849937314e-07\\
269	4.4770276130205e-07\\
270	4.41664927270692e-07\\
271	4.35774102620565e-07\\
272	4.30028513986359e-07\\
273	4.24425867715829e-07\\
274	4.18963322508861e-07\\
275	4.13637436521425e-07\\
276	4.08444101849398e-07\\
277	4.03378480789626e-07\\
278	3.98434958239977e-07\\
279	3.93607123069274e-07\\
280	3.88887788375404e-07\\
281	3.84269057059593e-07\\
282	3.79742434824561e-07\\
283	3.75298987945888e-07\\
284	3.70929540355683e-07\\
285	3.66624899699602e-07\\
286	3.6237610067779e-07\\
287	3.5817465217098e-07\\
288	3.54012774187243e-07\\
289	3.49883611625891e-07\\
290	3.45781413534259e-07\\
291	3.4170166875961e-07\\
292	3.37641192178773e-07\\
293	3.33598158304069e-07\\
294	3.29572083139842e-07\\
295	3.25563756871719e-07\\
296	3.21575133542529e-07\\
297	3.17609184936549e-07\\
298	3.13669727707369e-07\\
299	3.09761232776019e-07\\
300	3.05888626158692e-07\\
};
\addlegendentry{Alg.~3 p=1}

\addplot [color=mycolor2,thick]
  table[row sep=crcr]{%
1	0.913566325963014\\
2	0.0750745992815885\\
3	0.122873019792163\\
4	0.100072012321929\\
5	0.0838644412922525\\
6	0.0716734575987896\\
7	0.0622123388962697\\
8	0.0546770836637779\\
9	0.0484497019781569\\
10	0.0431818991772682\\
11	0.0386694164963214\\
12	0.0347214980900658\\
13	0.0312588050818286\\
14	0.0282022564730986\\
15	0.0255175348643806\\
16	0.0231658824508675\\
17	0.0211025492298491\\
18	0.0193053234958801\\
19	0.0177402348654937\\
20	0.0163675502413948\\
21	0.0151551177772797\\
22	0.0140754643339075\\
23	0.0131048004051367\\
24	0.0122245702523333\\
25	0.0114198227704953\\
26	0.0106777316170798\\
27	0.00998865806969374\\
28	0.00934443735887626\\
29	0.00873838947492334\\
30	0.0081657245680786\\
31	0.00762341257212925\\
32	0.00710894387300757\\
33	0.00662056289464825\\
34	0.00615650361191208\\
35	0.00571512356125508\\
36	0.00529487777395459\\
37	0.0048956046985879\\
38	0.00451828405422708\\
39	0.00416321860146818\\
40	0.00383016486593516\\
41	0.00351811021297076\\
42	0.00322596174270682\\
43	0.00295309837662189\\
44	0.00269847811426151\\
45	0.00246098854689282\\
46	0.0022397082794209\\
47	0.00203399806768009\\
48	0.00184307643108387\\
49	0.0016674350480795\\
50	0.0015069321436584\\
51	0.00136107967123285\\
52	0.00122924158934979\\
53	0.00111068653451098\\
54	0.00100469173263469\\
55	0.000910294090567026\\
56	0.000826393771395974\\
57	0.000751970775139613\\
58	0.000686027275878467\\
59	0.000627609966451353\\
60	0.000575815799630229\\
61	0.000529839307813405\\
62	0.000488983279078404\\
63	0.000452653737771857\\
64	0.000420295343312631\\
65	0.000391377518169474\\
66	0.00036541024069733\\
67	0.000341962219839747\\
68	0.000320668776428577\\
69	0.000301230673039487\\
70	0.00028340745184221\\
71	0.000267007571747257\\
72	0.000251874564624022\\
73	0.000237878299718539\\
74	0.000224904798452349\\
75	0.000212851016656425\\
76	0.000201622074441665\\
77	0.000191134113854812\\
78	0.000181317784963231\\
79	0.000172121045797004\\
80	0.00016349818627976\\
81	0.00015540682067171\\
82	0.000147804117992565\\
83	0.000140649265067364\\
84	0.000133905894951626\\
85	0.000127542179064099\\
86	0.000121530451344248\\
87	0.00011584749557363\\
88	0.000110473706537654\\
89	0.000105391742270929\\
90	0.000100585386864969\\
91	9.60389773259542e-05\\
92	9.17373574487968e-05\\
93	8.76661122926831e-05\\
94	8.38118267029574e-05\\
95	8.01622161424476e-05\\
96	7.67065079449849e-05\\
97	7.34349499297762e-05\\
98	7.03380055089306e-05\\
99	6.74065523423731e-05\\
100	6.46318290312689e-05\\
101	6.20056761928035e-05\\
102	5.95195930243285e-05\\
103	5.71651895823642e-05\\
104	5.49341714590817e-05\\
105	5.28184393274195e-05\\
106	5.08102047931946e-05\\
107	4.89020846779721e-05\\
108	4.70871752781517e-05\\
109	4.53591104462329e-05\\
110	4.37120976337583e-05\\
111	4.21409236494317e-05\\
112	4.06409254323954e-05\\
113	3.92079323784343e-05\\
114	3.7838191435228e-05\\
115	3.65282832281537e-05\\
116	3.52750414811549e-05\\
117	3.40754866137664e-05\\
118	3.29267804334522e-05\\
119	3.18262035709097e-05\\
120	3.07711525239559e-05\\
121	2.97591501768031e-05\\
122	2.87878627447157e-05\\
123	2.78551167803799e-05\\
124	2.69589114535477e-05\\
125	2.60974231957338e-05\\
126	2.52690016297671e-05\\
127	2.44721572798283e-05\\
128	2.37055427513436e-05\\
129	2.29679298022426e-05\\
130	2.22581849703384e-05\\
131	2.15752457609526e-05\\
132	2.09180879080133e-05\\
133	2.02857090060059e-05\\
134	1.96771283313794e-05\\
135	1.90913840621204e-05\\
136	1.85275314523777e-05\\
137	1.79846440223931e-05\\
138	1.74618172263394e-05\\
139	1.6958173302252e-05\\
140	1.64728663189576e-05\\
141	1.60050866835036e-05\\
142	1.55540644591525e-05\\
143	1.51190710428644e-05\\
144	1.46994190839576e-05\\
145	1.4294460810943e-05\\
146	1.39035850890935e-05\\
147	1.3526213597877e-05\\
148	1.31617965572258e-05\\
149	1.28098084499545e-05\\
150	1.24697441624862e-05\\
151	1.21411158871343e-05\\
152	1.18234510087653e-05\\
153	1.15162910599926e-05\\
154	1.12191916925132e-05\\
155	1.0931723492957e-05\\
156	1.06534733848793e-05\\
157	1.03840463100691e-05\\
158	1.01230668812131e-05\\
159	9.87018073509826e-06\\
160	9.62505538246953e-06\\
161	9.38738043530438e-06\\
162	9.15686717726898e-06\\
163	8.93324752181037e-06\\
164	8.71627246058163e-06\\
165	8.50571014600517e-06\\
166	8.30134376675947e-06\\
167	8.10296937421313e-06\\
168	7.91039379755866e-06\\
169	7.72343275794174e-06\\
170	7.54190925684149e-06\\
171	7.36565227987904e-06\\
172	7.19449582504104e-06\\
173	7.0282782381869e-06\\
174	6.86684181866459e-06\\
175	6.71003264494062e-06\\
176	6.55770056267376e-06\\
177	6.4096992774855e-06\\
178	6.26588649714835e-06\\
179	6.12612407628697e-06\\
180	5.99027812684273e-06\\
181	5.85821906972425e-06\\
182	5.72982161632825e-06\\
183	5.60496468067322e-06\\
184	5.48353123394196e-06\\
185	5.36540812001458e-06\\
186	5.25048585689643e-06\\
187	5.13865844793335e-06\\
188	5.02982322444969e-06\\
189	4.92388073666285e-06\\
190	4.82073470007855e-06\\
191	4.72029199772607e-06\\
192	4.62246272900976e-06\\
193	4.52716029054889e-06\\
194	4.43430146846672e-06\\
195	4.3438065216268e-06\\
196	4.25559923559939e-06\\
197	4.169606931715e-06\\
198	4.08576042038467e-06\\
199	4.00399389610326e-06\\
200	3.92424477642806e-06\\
201	3.84645349424482e-06\\
202	3.7705632567938e-06\\
203	3.69651978686467e-06\\
204	3.62427106229656e-06\\
205	3.55376706876603e-06\\
206	3.48495957709898e-06\\
207	3.41780195303037e-06\\
208	3.35224900356416e-06\\
209	3.28825685882387e-06\\
210	3.22578288625835e-06\\
211	3.16478563143052e-06\\
212	3.1052247775722e-06\\
213	3.04706111718331e-06\\
214	2.99025652914127e-06\\
215	2.93477395602558e-06\\
216	2.88057737879432e-06\\
217	2.82763178727479e-06\\
218	2.77590314705541e-06\\
219	2.72535836428299e-06\\
220	2.67596525053027e-06\\
221	2.62769249032534e-06\\
222	2.58050961356778e-06\\
223	2.53438697422545e-06\\
224	2.48929573565954e-06\\
225	2.44520786233608e-06\\
226	2.40209611676936e-06\\
227	2.35993405947686e-06\\
228	2.31869604972942e-06\\
229	2.27835724502662e-06\\
230	2.23889359709378e-06\\
231	2.20028184252756e-06\\
232	2.16249948798135e-06\\
233	2.12552478910839e-06\\
234	2.08933672431189e-06\\
235	2.05391496428849e-06\\
236	2.01923983905294e-06\\
237	1.98529230423673e-06\\
238	1.95205390851258e-06\\
239	1.91950676341462e-06\\
240	1.88763351720127e-06\\
241	1.85641733286245e-06\\
242	1.82584187062053e-06\\
243	1.79589127484014e-06\\
244	1.76655016397281e-06\\
245	1.73780362285395e-06\\
246	1.70963719618115e-06\\
247	1.6820368816178e-06\\
248	1.65498912172921e-06\\
249	1.62848079398627e-06\\
250	1.60249919789208e-06\\
251	1.57703203953856e-06\\
252	1.55206741345778e-06\\
253	1.52759378216696e-06\\
254	1.50359995421817e-06\\
255	1.48007506132163e-06\\
256	1.45700853531813e-06\\
257	1.43439008602753e-06\\
258	1.41220968026372e-06\\
259	1.39045752289759e-06\\
260	1.36912403981854e-06\\
261	1.34819986389624e-06\\
262	1.32767582308995e-06\\
263	1.30754293077139e-06\\
264	1.28779237888385e-06\\
265	1.26841553252724e-06\\
266	1.2494039266708e-06\\
267	1.23074926414776e-06\\
268	1.21244341505164e-06\\
269	1.19447841687333e-06\\
270	1.17684647567458e-06\\
271	1.15953996776379e-06\\
272	1.14255144194417e-06\\
273	1.12587362194154e-06\\
274	1.10949940921624e-06\\
275	1.09342188570321e-06\\
276	1.07763431648773e-06\\
277	1.06213015235376e-06\\
278	1.04690303168349e-06\\
279	1.03194678190838e-06\\
280	1.0172554200669e-06\\
281	1.00282315248622e-06\\
282	9.88644373268663e-07\\
283	9.74713661539527e-07\\
284	9.61025777524056e-07\\
285	9.47575657272125e-07\\
286	9.34358406023787e-07\\
287	9.21369290602519e-07\\
288	9.08603730558587e-07\\
289	8.9605728870179e-07\\
290	8.83725660818895e-07\\
291	8.71604665197631e-07\\
292	8.59690231976365e-07\\
293	8.47978392684579e-07\\
294	8.36465270039229e-07\\
295	8.2514706868566e-07\\
296	8.14020066344442e-07\\
297	8.0308060626099e-07\\
298	7.92325090481771e-07\\
299	7.81749974441102e-07\\
300	7.71351762664241e-07\\
};
\addlegendentry{Alg.~4 p=1}

\addplot [color=mycolor3,thick]
  table[row sep=crcr]{%
1	0.913566325963014\\
2	0.218849328668934\\
3	0.164526163478926\\
4	0.136293259339846\\
5	0.118272896514932\\
6	0.105144856984558\\
7	0.0950674450115387\\
8	0.0869973220630887\\
9	0.0804426128883778\\
10	0.0749376670050576\\
11	0.0703164675955274\\
12	0.0663785652034667\\
13	0.0629521509303361\\
14	0.0599100434822513\\
15	0.0571908266663501\\
16	0.0547592003339716\\
17	0.0525385038801222\\
18	0.0505527037791488\\
19	0.0487383725703733\\
20	0.0470643887293395\\
21	0.0455079940671638\\
22	0.0440611750667397\\
23	0.0427029767284519\\
24	0.0414293553274315\\
25	0.0402364022124906\\
26	0.039105544534946\\
27	0.0380404584795253\\
28	0.0370193897324855\\
29	0.0360606075890395\\
30	0.0351398073391532\\
31	0.0342645470613541\\
32	0.0334291913179092\\
33	0.0326296608777963\\
34	0.0318624689732393\\
35	0.0311274221259974\\
36	0.0304320028598176\\
37	0.0297589031004341\\
38	0.0291057461733799\\
39	0.0284725601869799\\
40	0.0278586789907999\\
41	0.0272649621720755\\
42	0.0266911179566644\\
43	0.026142833895599\\
44	0.0256264431443324\\
45	0.0251293052406446\\
46	0.0246527782893155\\
47	0.024199762846913\\
48	0.0237602975819366\\
49	0.0233344954102251\\
50	0.022920430286561\\
51	0.0225203922669306\\
52	0.0221349316096562\\
53	0.0217629043589791\\
54	0.0214002552826343\\
55	0.0210482320201725\\
56	0.0207064265655474\\
57	0.020373505179071\\
58	0.0200482962701744\\
59	0.0197378258882505\\
60	0.0194379850989262\\
61	0.0191472225925997\\
62	0.0188631091461358\\
63	0.0185852330000455\\
64	0.0183201976151646\\
65	0.0180631192277129\\
66	0.0178117976377762\\
67	0.017569126529139\\
68	0.0173341324055165\\
69	0.0171041235734464\\
70	0.0168788066091593\\
71	0.0166579262352537\\
72	0.0164412512065736\\
73	0.0162293194985498\\
74	0.0160268894371749\\
75	0.0158283379641783\\
76	0.0156334380687221\\
77	0.0154457936700052\\
78	0.0152630358736537\\
79	0.0150847607581311\\
80	0.0149100719365854\\
81	0.0147423148674411\\
82	0.0145817298563686\\
83	0.0144247253177499\\
84	0.0142710803349814\\
85	0.0141207590252336\\
86	0.0139733978033058\\
87	0.0138297529064918\\
88	0.0136901714790784\\
89	0.0135545044585725\\
90	0.0134222108630765\\
91	0.0132937362424712\\
92	0.0131690279456912\\
93	0.0130479798770777\\
94	0.0129300606019156\\
95	0.0128158540124646\\
96	0.0127039108173796\\
97	0.0125943867772219\\
98	0.0124869687015834\\
99	0.012381598552422\\
100	0.012277968674936\\
101	0.0121759911533445\\
102	0.0120755878007667\\
103	0.0119766883451338\\
104	0.0118792290706958\\
105	0.0117831517819478\\
106	0.0116884029283188\\
107	0.0115949329612297\\
108	0.0115026959004129\\
109	0.0114116489160023\\
110	0.0113217519812145\\
111	0.0112329675813254\\
112	0.0111452604668861\\
113	0.0110600841984389\\
114	0.0109767540720533\\
115	0.0108945594220177\\
116	0.0108134455077223\\
117	0.0107333638437993\\
118	0.0106542710075624\\
119	0.0105761277415092\\
120	0.0104990820839493\\
121	0.0104231634856457\\
122	0.0103481288986254\\
123	0.0102739456561617\\
124	0.0102005840223351\\
125	0.010128016714779\\
126	0.0100562185347742\\
127	0.0099851660763215\\
128	0.0099148374940169\\
129	0.00984521231525924\\
130	0.00977627128629096\\
131	0.00970799624438833\\
132	0.00964036976641301\\
133	0.00957337465304855\\
134	0.00950699448104827\\
135	0.00944121372496396\\
136	0.00937601767452148\\
137	0.00931139236376854\\
138	0.00924732450966276\\
139	0.00918380145833862\\
140	0.00912081113772177\\
141	0.00905834148688929\\
142	0.0089963803667029\\
143	0.00893491628413295\\
144	0.00887393837397639\\
145	0.00881343633871353\\
146	0.00875340039706607\\
147	0.00869382123956708\\
148	0.00863474242056612\\
149	0.00857611260191041\\
150	0.00851791992186845\\
151	0.00846015630876078\\
152	0.00840281407361892\\
153	0.00834588587160338\\
154	0.00828936467025498\\
155	0.00823324372292492\\
156	0.00817751654619268\\
157	0.0081221769003993\\
158	0.00806721877265978\\
159	0.00801263636186824\\
160	0.00795842406534269\\
161	0.00790457646682686\\
162	0.00785113487935141\\
163	0.00779822869706915\\
164	0.00774568231831387\\
165	0.00769348890763319\\
166	0.00764164222446696\\
167	0.00759013646343018\\
168	0.00753896600730809\\
169	0.00748812549638532\\
170	0.00743760987674663\\
171	0.0073874143491599\\
172	0.00733753432711926\\
173	0.00728796540540948\\
174	0.00723870334585444\\
175	0.00718974400944098\\
176	0.0071410833558014\\
177	0.00709271748595321\\
178	0.00704464257921422\\
179	0.00699698030022109\\
180	0.00694970227348056\\
181	0.00690273248425262\\
182	0.00685606617069373\\
183	0.00680969887824182\\
184	0.00676362639698912\\
185	0.00671845949561869\\
186	0.00667384463696557\\
187	0.0066295828585692\\
188	0.00658566041351238\\
189	0.00654206586072375\\
190	0.00649878943723015\\
191	0.00645582262895556\\
192	0.00641315789387588\\
193	0.00637078843087389\\
194	0.00632870796413904\\
195	0.00628691074169271\\
196	0.00624539140846602\\
197	0.00620414494387977\\
198	0.00616316661368544\\
199	0.00612245193215191\\
200	0.00608199663188481\\
201	0.00604179663935963\\
202	0.00600198841850371\\
203	0.00596260913716559\\
204	0.00592349602180997\\
205	0.0058846423475801\\
206	0.00584604216265884\\
207	0.00580769011518053\\
208	0.00576958132870889\\
209	0.0057317113111207\\
210	0.00569407588677062\\
211	0.0056566711451085\\
212	0.00561949340108799\\
213	0.00558253916415431\\
214	0.00554580511356518\\
215	0.00550928807845869\\
216	0.00547298502153064\\
217	0.00543689302549522\\
218	0.00540100928172452\\
219	0.005365570025313\\
220	0.00533037530409505\\
221	0.00529541874942721\\
222	0.0052608589219534\\
223	0.00522653891446815\\
224	0.00519245222313654\\
225	0.00515859316588771\\
226	0.00512495668578705\\
227	0.00509153813025722\\
228	0.0050583332125575\\
229	0.00502533800430396\\
230	0.00499254887249546\\
231	0.00495996243135236\\
232	0.0049275940229089\\
233	0.00489543007845797\\
234	0.00486346035314751\\
235	0.0048316819305356\\
236	0.0048000920665313\\
237	0.00476868816339802\\
238	0.00473746774966943\\
239	0.00470644729811063\\
240	0.00467563706578421\\
241	0.00464500732092079\\
242	0.00461455533983861\\
243	0.00458427857674522\\
244	0.00455417463160265\\
245	0.00452424122651454\\
246	0.00449447618800341\\
247	0.00446487743341485\\
248	0.00443544296027257\\
249	0.0044061708377768\\
250	0.00437705919988791\\
251	0.00434810623960717\\
252	0.00431931020417409\\
253	0.00429066939097815\\
254	0.00426218214403866\\
255	0.00423384685093924\\
256	0.00420566194013366\\
257	0.00417762587855912\\
258	0.00414973716950456\\
259	0.00412199435069511\\
260	0.00409439599252459\\
261	0.00406694069433792\\
262	0.00403962708337736\\
263	0.00401245381556021\\
264	0.00398541957420796\\
265	0.00395852306887999\\
266	0.00393176303429771\\
267	0.00390513822934719\\
268	0.00387864743615379\\
269	0.00385228945921743\\
270	0.00382606312460631\\
271	0.00379996727920075\\
272	0.00377400074657173\\
273	0.00374816833024269\\
274	0.00372249880943798\\
275	0.00369699960424067\\
276	0.00367163207176097\\
277	0.00364639394956505\\
278	0.00362128321418011\\
279	0.0035962980255652\\
280	0.00357162481667142\\
281	0.00354717473076275\\
282	0.00352286873430736\\
283	0.00349870328364152\\
284	0.0034746752460502\\
285	0.00345078181638488\\
286	0.00342702045475737\\
287	0.00340338883925828\\
288	0.00337988482953618\\
289	0.00335650752202595\\
290	0.00333329268739188\\
291	0.00331020170374186\\
292	0.00328723226902789\\
293	0.00326438229701933\\
294	0.00324164986292093\\
295	0.00321903319930841\\
296	0.00319653067092888\\
297	0.00317414075216518\\
298	0.00315186200970822\\
299	0.00312969308904927\\
300	0.00310763270380455\\
};
\addlegendentry{proj. grad.}

\addplot [color=mycolor4,thick]
  table[row sep=crcr]{%
1	0.913566325963014\\
2	0.218849328668934\\
3	0.164526163478926\\
4	0.130128048966431\\
5	0.106156352218396\\
6	0.0882711572545895\\
7	0.0746802846502666\\
8	0.0643235543752128\\
9	0.0561687259690167\\
10	0.0495858106906956\\
11	0.044140417908484\\
12	0.0394540445523935\\
13	0.0353943377265585\\
14	0.0318572879764421\\
15	0.0287013667792849\\
16	0.025849195310961\\
17	0.0233636771392012\\
18	0.0211548861742355\\
19	0.0192377649907493\\
20	0.0176044845494303\\
21	0.0161300858341509\\
22	0.0148421947533479\\
23	0.0137760917072895\\
24	0.012852954129331\\
25	0.0120370133821781\\
26	0.0112924440214385\\
27	0.0106222458087275\\
28	0.0100100868420507\\
29	0.00943408788235371\\
30	0.00888100758323779\\
31	0.0083420023806542\\
32	0.00782006817299993\\
33	0.00731261606345786\\
34	0.00682953865959297\\
35	0.00636710629762098\\
36	0.00592100987684578\\
37	0.00549071379014898\\
38	0.00507243932628344\\
39	0.00468186747560619\\
40	0.00431079480161154\\
41	0.00396197795050381\\
42	0.00363257740075622\\
43	0.00331888600716976\\
44	0.00301782211598623\\
45	0.00272715358616293\\
46	0.00244548234999158\\
47	0.00217335544460191\\
48	0.00192461249251542\\
49	0.00171923624182981\\
50	0.00154297012150522\\
51	0.00138453347243623\\
52	0.00123853248322875\\
53	0.0011011971244014\\
54	0.000974245019941286\\
55	0.000855206423029223\\
56	0.00074276180821148\\
57	0.00063611625317959\\
58	0.000534107396909836\\
59	0.000438513539178807\\
60	0.000352198472519575\\
61	0.000274113795892761\\
62	0.000210383898415651\\
63	0.000164612537699897\\
64	0.000139643326184938\\
65	0.000129070148457276\\
66	0.000124702752706162\\
67	0.000125111559833543\\
68	0.000128617315586844\\
69	0.000134619832347914\\
70	0.00014185949437171\\
71	0.000151384882515052\\
72	0.000159188491970724\\
73	0.000164788749032381\\
74	0.00016717775513225\\
75	0.000165878297266977\\
76	0.000162661211281061\\
77	0.000158781261981374\\
78	0.000154989082155129\\
79	0.000150860668789604\\
80	0.000146394515641961\\
81	0.000141460560462038\\
82	0.000136087055034734\\
83	0.000130086287245079\\
84	0.000123852023444196\\
85	0.000117190184965547\\
86	0.000110079276162614\\
87	0.000102614468179913\\
88	9.52454866111237e-05\\
89	8.81427235792428e-05\\
90	8.16144938376399e-05\\
91	7.64489949488488e-05\\
92	7.20835363341192e-05\\
93	6.84261674153011e-05\\
94	6.53623741583443e-05\\
95	6.27625648873644e-05\\
96	6.04929489442735e-05\\
97	5.84271638220972e-05\\
98	5.64555141065e-05\\
99	5.44984089218331e-05\\
100	5.27390441028988e-05\\
101	5.11189168254274e-05\\
102	4.9559024162097e-05\\
103	4.82099364876552e-05\\
104	4.67322484790263e-05\\
105	4.50079944284525e-05\\
106	4.29831805594397e-05\\
107	4.06785265493299e-05\\
108	3.819368585933e-05\\
109	3.55330342996594e-05\\
110	3.27546183546343e-05\\
111	2.99516796624724e-05\\
112	2.71986543553641e-05\\
113	2.45279382125327e-05\\
114	2.19363150569469e-05\\
115	1.9431884844736e-05\\
116	1.70657782052989e-05\\
117	1.48839010340506e-05\\
118	1.30108780750805e-05\\
119	1.13763533329068e-05\\
120	1.00081608437521e-05\\
121	8.95574850952926e-06\\
122	8.16481240378981e-06\\
123	7.66037025362858e-06\\
124	7.38267456518258e-06\\
125	7.30603700232427e-06\\
126	7.41526481522464e-06\\
127	7.64132406679864e-06\\
128	7.94254715189274e-06\\
129	8.35088879347443e-06\\
130	8.86870510770008e-06\\
131	9.31295242550482e-06\\
132	9.68005852002018e-06\\
133	9.95423973030457e-06\\
134	1.00958711300872e-05\\
135	1.00903647767391e-05\\
136	9.94073583346616e-06\\
137	9.66446075117977e-06\\
138	9.27405968373894e-06\\
139	8.82148975061246e-06\\
140	8.34444751415808e-06\\
141	7.87683415837646e-06\\
142	7.44531648951084e-06\\
143	7.06792743419143e-06\\
144	6.75463907780054e-06\\
145	6.50906464364723e-06\\
146	6.33022170057167e-06\\
147	6.21350425189255e-06\\
148	6.15066786814784e-06\\
149	6.13054271453393e-06\\
150	6.14000019821901e-06\\
151	6.16531835717464e-06\\
152	6.19325668300454e-06\\
153	6.21110270847616e-06\\
154	6.20766161100929e-06\\
155	6.17418762902892e-06\\
156	6.10475727038946e-06\\
157	5.99630957515054e-06\\
158	5.8485598522914e-06\\
159	5.66384755332997e-06\\
160	5.4468562837558e-06\\
161	5.20412806964948e-06\\
162	4.94337094452887e-06\\
163	4.67736702598233e-06\\
164	4.41977735484312e-06\\
165	4.17482054092776e-06\\
166	3.94826582268866e-06\\
167	3.74336815480575e-06\\
168	3.56096191228648e-06\\
169	3.39979973587243e-06\\
170	3.25701711029885e-06\\
171	3.12860232902848e-06\\
172	3.00997879699831e-06\\
173	2.89659902641577e-06\\
174	2.78443755892532e-06\\
175	2.67033291850762e-06\\
176	2.55022926852998e-06\\
177	2.42609050381259e-06\\
178	2.29712956533509e-06\\
179	2.16323932065565e-06\\
180	2.02533757344674e-06\\
181	1.89249410269873e-06\\
182	1.79020826171234e-06\\
183	1.71333498698387e-06\\
184	1.65241246766722e-06\\
185	1.60090234753007e-06\\
186	1.55362598049102e-06\\
187	1.50750574261661e-06\\
188	1.46139886997906e-06\\
189	1.41770270794227e-06\\
190	1.37773233828756e-06\\
191	1.34228522948986e-06\\
192	1.31241665202002e-06\\
193	1.288858813237e-06\\
194	1.2720292843528e-06\\
195	1.26202199031943e-06\\
196	1.2585804835619e-06\\
197	1.26108647078346e-06\\
198	1.26859349783553e-06\\
199	1.27991363181503e-06\\
200	1.29374259274541e-06\\
201	1.30879621589031e-06\\
202	1.32392963193752e-06\\
203	1.33821694249005e-06\\
204	1.35097979019842e-06\\
205	1.36176536334395e-06\\
206	1.37028578581656e-06\\
207	1.3763391501281e-06\\
208	1.37973547528151e-06\\
209	1.38024761856415e-06\\
210	1.3775985866181e-06\\
211	1.37148553457176e-06\\
212	1.36163048632121e-06\\
213	1.34700640303832e-06\\
214	1.32820651578373e-06\\
215	1.30589305073325e-06\\
216	1.27979893861738e-06\\
217	1.25005936455966e-06\\
218	1.21713945392364e-06\\
219	1.18166245990144e-06\\
220	1.14424996249668e-06\\
221	1.1054377306724e-06\\
222	1.0656687039123e-06\\
223	1.02532526296585e-06\\
224	9.84758769356255e-07\\
225	9.44294608506979e-07\\
226	9.04215820012391e-07\\
227	8.64742411214696e-07\\
228	8.26013529733535e-07\\
229	7.88081803025122e-07\\
230	7.50939715598621e-07\\
231	7.14555888339449e-07\\
232	6.78912762872713e-07\\
233	6.44039389454537e-07\\
234	6.10035871109749e-07\\
235	5.77088021850095e-07\\
236	5.45471634074042e-07\\
237	5.15545965284678e-07\\
238	4.87736424786139e-07\\
239	4.62507341853977e-07\\
240	4.40326989797207e-07\\
241	4.21628280861813e-07\\
242	4.06769270762213e-07\\
243	3.95997624858847e-07\\
244	3.89422564347174e-07\\
245	3.86996774803111e-07\\
246	3.88509595594493e-07\\
247	3.93591705088146e-07\\
248	4.01730574623067e-07\\
249	4.12295213159455e-07\\
250	4.24568168008038e-07\\
251	4.37782384257908e-07\\
252	4.51160362387702e-07\\
253	4.63953075575084e-07\\
254	4.75476282150653e-07\\
255	4.85142171262332e-07\\
256	4.92484647436273e-07\\
257	4.9717699169595e-07\\
258	4.99041095506497e-07\\
259	4.98047961730759e-07\\
260	4.94309686111145e-07\\
261	4.8806366125323e-07\\
262	4.79650245457583e-07\\
263	4.69485571803589e-07\\
264	4.58031483944847e-07\\
265	4.45764743537552e-07\\
266	4.33147627580503e-07\\
267	4.20601828632176e-07\\
268	4.08487208318313e-07\\
269	3.97086473486429e-07\\
270	3.8659631148262e-07\\
271	3.77124979383803e-07\\
272	3.68695288604191e-07\\
273	3.61248626792703e-07\\
274	3.54658936579932e-07\\
275	3.48752524630125e-07\\
276	3.43328167448816e-07\\
277	3.38175782555079e-07\\
278	3.33092705673797e-07\\
279	3.27897154344357e-07\\
280	3.22438724107039e-07\\
281	3.16605845167642e-07\\
282	3.10330150247312e-07\\
283	3.03587751831688e-07\\
284	2.96397550037485e-07\\
285	2.88816863123348e-07\\
286	2.80934861208483e-07\\
287	2.72864445569759e-07\\
288	2.64733307599359e-07\\
289	2.56674915775738e-07\\
290	2.48820102756973e-07\\
291	2.41289782330727e-07\\
292	2.3418913801948e-07\\
293	2.27603430822884e-07\\
294	2.21595400106724e-07\\
295	2.16204106297767e-07\\
296	2.1144499721612e-07\\
297	2.07310962481791e-07\\
298	2.03774169389888e-07\\
299	2.00788518010114e-07\\
300	1.98292594735631e-07\\
};
\addlegendentry{acc. proj. grad.}

\addplot [color=black, dashed]
  table[row sep=crcr]{%
1	1\\
2	0.25\\
3	0.111111111111111\\
4	0.0625\\
5	0.04\\
6	0.0277777777777778\\
7	0.0204081632653061\\
8	0.015625\\
9	0.0123456790123457\\
10	0.01\\
11	0.00826446280991736\\
12	0.00694444444444444\\
13	0.00591715976331361\\
14	0.00510204081632653\\
15	0.00444444444444444\\
16	0.00390625\\
17	0.00346020761245675\\
18	0.00308641975308642\\
19	0.00277008310249307\\
20	0.0025\\
21	0.00226757369614512\\
22	0.00206611570247934\\
23	0.00189035916824197\\
24	0.00173611111111111\\
25	0.0016\\
26	0.0014792899408284\\
27	0.00137174211248285\\
28	0.00127551020408163\\
29	0.00118906064209275\\
30	0.00111111111111111\\
31	0.00104058272632674\\
32	0.0009765625\\
33	0.000918273645546373\\
34	0.000865051903114187\\
35	0.000816326530612245\\
36	0.000771604938271605\\
37	0.000730460189919649\\
38	0.000692520775623269\\
39	0.000657462195923734\\
40	0.000625\\
41	0.000594883997620464\\
42	0.000566893424036281\\
43	0.000540832882639265\\
44	0.000516528925619835\\
45	0.000493827160493827\\
46	0.000472589792060492\\
47	0.000452693526482571\\
48	0.000434027777777778\\
49	0.00041649312786339\\
50	0.0004\\
51	0.000384467512495194\\
52	0.000369822485207101\\
53	0.000355998576005696\\
54	0.000342935528120713\\
55	0.000330578512396694\\
56	0.000318877551020408\\
57	0.000307787011388119\\
58	0.000297265160523187\\
59	0.000287273771904625\\
60	0.000277777777777778\\
61	0.000268744961031981\\
62	0.000260145681581686\\
63	0.000251952632905014\\
64	0.000244140625\\
65	0.000236686390532544\\
66	0.000229568411386593\\
67	0.000222766763198931\\
68	0.000216262975778547\\
69	0.000210039907582441\\
70	0.000204081632653061\\
71	0.000198373338623289\\
72	0.000192901234567901\\
73	0.000187652467629949\\
74	0.000182615047479912\\
75	0.000177777777777778\\
76	0.000173130193905817\\
77	0.000168662506324844\\
78	0.000164365548980934\\
79	0.000160230732254446\\
80	0.00015625\\
81	0.000152415790275873\\
82	0.000148720999405116\\
83	0.000145158949049209\\
84	0.00014172335600907\\
85	0.00013840830449827\\
86	0.000135208220659816\\
87	0.000132117849121416\\
88	0.000129132231404959\\
89	0.000126246686024492\\
90	0.000123456790123457\\
91	0.000120758362516604\\
92	0.000118147448015123\\
93	0.000115620302925194\\
94	0.000113173381620643\\
95	0.000110803324099723\\
96	0.000108506944444444\\
97	0.000106281220108407\\
98	0.000104123281965848\\
99	0.000102030405060708\\
100	0.0001\\
101	9.80296049406921e-05\\
102	9.61168781237985e-05\\
103	9.42595909133754e-05\\
104	9.24556213017751e-05\\
105	9.0702947845805e-05\\
106	8.8999644001424e-05\\
107	8.73438728273212e-05\\
108	8.57338820301783e-05\\
109	8.4167999326656e-05\\
110	8.26446280991736e-05\\
111	8.11622433244055e-05\\
112	7.9719387755102e-05\\
113	7.83146683373796e-05\\
114	7.69467528470299e-05\\
115	7.56143667296786e-05\\
116	7.43162901307967e-05\\
117	7.30513551026372e-05\\
118	7.18184429761563e-05\\
119	7.06164818868724e-05\\
120	6.94444444444444e-05\\
121	6.83013455365071e-05\\
122	6.71862402579952e-05\\
123	6.60982219578293e-05\\
124	6.50364203954214e-05\\
125	6.4e-05\\
126	6.29881582262535e-05\\
127	6.2000124000248e-05\\
128	6.103515625e-05\\
129	6.00925425154738e-05\\
130	5.91715976331361e-05\\
131	5.82716624905309e-05\\
132	5.73921028466483e-05\\
133	5.65323082141444e-05\\
134	5.56916907997327e-05\\
135	5.48696844993141e-05\\
136	5.40657439446367e-05\\
137	5.32793435984869e-05\\
138	5.25099768956102e-05\\
139	5.17571554267377e-05\\
140	5.10204081632653e-05\\
141	5.02992807202857e-05\\
142	4.95933346558223e-05\\
143	4.89021468042447e-05\\
144	4.82253086419753e-05\\
145	4.75624256837099e-05\\
146	4.69131169074873e-05\\
147	4.62770142070434e-05\\
148	4.56537618699781e-05\\
149	4.50430160803567e-05\\
150	4.44444444444444e-05\\
151	4.38577255383536e-05\\
152	4.32825484764543e-05\\
153	4.2718612499466e-05\\
154	4.2165626581211e-05\\
155	4.16233090530697e-05\\
156	4.10913872452334e-05\\
157	4.05695971439004e-05\\
158	4.00576830636116e-05\\
159	3.95553973339662e-05\\
160	3.90625e-05\\
161	3.85787585355503e-05\\
162	3.81039475689681e-05\\
163	3.76378486205728e-05\\
164	3.7180249851279e-05\\
165	3.67309458218549e-05\\
166	3.62897372623022e-05\\
167	3.58564308508731e-05\\
168	3.54308390022676e-05\\
169	3.50127796645776e-05\\
170	3.46020761245675e-05\\
171	3.41985568209022e-05\\
172	3.3802055164954e-05\\
173	3.34124093688396e-05\\
174	3.30294622803541e-05\\
175	3.26530612244898e-05\\
176	3.22830578512397e-05\\
177	3.19193079894028e-05\\
178	3.1561671506123e-05\\
179	3.12100121719047e-05\\
180	3.08641975308642e-05\\
181	3.05240987759836e-05\\
182	3.01895906291511e-05\\
183	2.98605512257756e-05\\
184	2.95368620037807e-05\\
185	2.9218407596786e-05\\
186	2.89050757312984e-05\\
187	2.85967571277417e-05\\
188	2.82933454051607e-05\\
189	2.7994736989446e-05\\
190	2.77008310249307e-05\\
191	2.7411529289219e-05\\
192	2.71267361111111e-05\\
193	2.68463582914978e-05\\
194	2.65703050271017e-05\\
195	2.62984878369494e-05\\
196	2.60308204914619e-05\\
197	2.57672189440594e-05\\
198	2.5507601265177e-05\\
199	2.52518875785965e-05\\
200	2.5e-05\\
201	2.4751862577659e-05\\
202	2.4507401235173e-05\\
203	2.42665437161785e-05\\
204	2.40292195309496e-05\\
205	2.37953599048186e-05\\
206	2.35648977283439e-05\\
207	2.33377675091601e-05\\
208	2.31139053254438e-05\\
209	2.28932487809345e-05\\
210	2.26757369614512e-05\\
211	2.24613103928483e-05\\
212	2.2249911000356e-05\\
213	2.20414820692543e-05\\
214	2.18359682068303e-05\\
215	2.16333153055706e-05\\
216	2.14334705075446e-05\\
217	2.12363821699335e-05\\
218	2.1041999831664e-05\\
219	2.08502741811055e-05\\
220	2.06611570247934e-05\\
221	2.04746012571405e-05\\
222	2.02905608311014e-05\\
223	2.01089907297553e-05\\
224	1.99298469387755e-05\\
225	1.97530864197531e-05\\
226	1.95786670843449e-05\\
227	1.94065477692173e-05\\
228	1.92366882117575e-05\\
229	1.9069049026525e-05\\
230	1.89035916824197e-05\\
231	1.87402784805382e-05\\
232	1.85790725326992e-05\\
233	1.84199377406104e-05\\
234	1.82628387756593e-05\\
235	1.81077410593029e-05\\
236	1.79546107440391e-05\\
237	1.78034146949385e-05\\
238	1.76541204717181e-05\\
239	1.75066963113391e-05\\
240	1.73611111111111e-05\\
241	1.72173344122863e-05\\
242	1.70753363841268e-05\\
243	1.69350878084303e-05\\
244	1.67965600644988e-05\\
245	1.66597251145356e-05\\
246	1.65245554894573e-05\\
247	1.6391024275107e-05\\
248	1.62591050988554e-05\\
249	1.61287721165788e-05\\
250	1.6e-05\\
251	1.58727639243822e-05\\
252	1.57470395565634e-05\\
253	1.5622803043322e-05\\
254	1.5500031000062e-05\\
255	1.53787004998078e-05\\
256	1.52587890625e-05\\
257	1.51402746445821e-05\\
258	1.50231356288685e-05\\
259	1.49073508146867e-05\\
260	1.4792899408284e-05\\
261	1.46797610134907e-05\\
262	1.45679156226327e-05\\
263	1.44573436076855e-05\\
264	1.43480257116621e-05\\
265	1.42399430402278e-05\\
266	1.41330770535361e-05\\
267	1.40274095582769e-05\\
268	1.39229226999332e-05\\
269	1.38195989552383e-05\\
270	1.37174211248285e-05\\
271	1.36163723260849e-05\\
272	1.35164359861592e-05\\
273	1.34175958351783e-05\\
274	1.33198358996217e-05\\
275	1.32231404958678e-05\\
276	1.31274942239025e-05\\
277	1.30328819611881e-05\\
278	1.29392888566844e-05\\
279	1.28467003250215e-05\\
280	1.27551020408163e-05\\
281	1.26644799331315e-05\\
282	1.25748201800714e-05\\
283	1.24861092035111e-05\\
284	1.23983336639556e-05\\
285	1.23114804555248e-05\\
286	1.22255367010612e-05\\
287	1.21404897473564e-05\\
288	1.20563271604938e-05\\
289	1.19730367213036e-05\\
290	1.18906064209275e-05\\
291	1.18090244564896e-05\\
292	1.17282792268718e-05\\
293	1.16483593285886e-05\\
294	1.15692535517608e-05\\
295	1.1490950876185e-05\\
296	1.14134404674945e-05\\
297	1.1336711673412e-05\\
298	1.12607540200892e-05\\
299	1.11855572085323e-05\\
300	1.11111111111111e-05\\
};
\addlegendentry{$1/k^2$}

\end{axis}
\end{tikzpicture}%

%% file: media/resultCompressedSensingObjp0p8.tikz
%
%
\definecolor{mycolor1}{rgb}{0.00000,0.44700,0.74100}%
\definecolor{mycolor2}{rgb}{0.85000,0.32500,0.09800}%
\begin{tikzpicture}

\begin{axis}[%
width=0.951\figurewidth,
height=\figureheight,
at={(0\figurewidth,0\figureheight)},
scale only axis,
xmin=0,
xmax=500,
xlabel style={font=\color{white!15!black}},
xlabel={iterations},
xlabel near ticks,
ymode=log,
ymin=1e-05,
ymax=100,
yminorticks=true,
ylabel style={font=\color{white!15!black}},
ylabel={$|f(x_k)-f^*|$},
ylabel near ticks,
axis background/.style={fill=white},
legend style={legend cell align=left, align=left, draw=white!15!black}
]
\addplot [color=black,thick]
  table[row sep=crcr]{%
1	0.907539205035436\\
2	0.39358260897427\\
3	0.273658428962888\\
4	0.231988532710047\\
5	0.206436181192681\\
6	0.186707561750578\\
7	0.0746038801158262\\
8	0.0432349983636522\\
9	0.0399302526489924\\
10	0.0438195146097832\\
11	0.0487853359492036\\
12	0.0525663410443774\\
13	0.056249649516309\\
14	0.0577372197070251\\
15	0.0578229057693443\\
16	0.0560871700865482\\
17	0.0541499882918061\\
18	0.0518774089812703\\
19	0.0488273901170744\\
20	0.0463875678589874\\
21	0.0438867315862696\\
22	0.0413030991685769\\
23	0.0393239021904826\\
24	0.0377126852511874\\
25	0.0111312341716577\\
26	0.00135667629009871\\
27	0.00287241367923921\\
28	0.00479766442805892\\
29	0.00544607873913305\\
30	0.00563409737155958\\
31	0.00556789947615059\\
32	0.00549184861116458\\
33	0.00549412675441599\\
34	0.00538355726486405\\
35	0.00521517621177638\\
36	0.00492558635751668\\
37	0.00451616797626633\\
38	0.00404880707416481\\
39	0.00365722333294892\\
40	0.00326309872024632\\
41	0.00281817641597902\\
42	0.00227099630457739\\
43	0.0018579577130907\\
44	0.00155521658313714\\
45	0.00126561959179192\\
46	0.00104634504425339\\
47	0.000868153872331219\\
48	0.000632021847357086\\
49	0.000432381504019457\\
50	0.00030899230239264\\
51	0.000277677546238654\\
52	0.000291076502823624\\
53	0.000452650066424049\\
54	0.000584174824320443\\
55	0.000916959829247756\\
56	0.00104852562329753\\
57	0.0012571072237414\\
58	0.00146733755518917\\
59	0.00165299803069797\\
60	0.00183517080433769\\
61	0.00204575339432804\\
62	0.00230377387397724\\
63	0.0025211927639207\\
64	0.00275245779159193\\
65	0.00299922235081431\\
66	0.00327855643593926\\
67	0.00350444844158906\\
68	0.00367637122476344\\
69	0.00384716185844627\\
70	0.00399243141568333\\
71	0.00414798750795041\\
72	0.00421133228037666\\
73	0.00426715763233682\\
74	0.00430204882837162\\
75	0.0043592387664194\\
76	0.00437318136918318\\
77	0.00441177045629346\\
78	0.00448034444164327\\
79	0.00456540538054783\\
80	0.0046951835928063\\
81	0.00478477315783837\\
82	0.00495177648238584\\
83	0.00505553698840233\\
84	0.00508082074843768\\
85	0.00515707752909601\\
86	0.0051637552993853\\
87	0.00518796590145539\\
88	0.00516870804045372\\
89	0.0051306126138148\\
90	0.00507532145566566\\
91	0.0050050225246564\\
92	0.00492500434454933\\
93	0.00484958747828276\\
94	0.00476247377130546\\
95	0.00467524190954721\\
96	0.00459536038001858\\
97	0.00450901074524418\\
98	0.00442876304339573\\
99	0.00435644190253601\\
100	0.00429167202815161\\
101	0.00423298212892299\\
102	0.00417825622544651\\
103	0.0041252114637829\\
104	0.00407182106279439\\
105	0.00401662417098107\\
106	0.0039588935201692\\
107	0.00389866121191367\\
108	0.00383662879090324\\
109	0.00377400947298717\\
110	0.00371237666357742\\
111	0.00365367555315785\\
112	0.00360114806515018\\
113	0.00355891449659676\\
114	0.00351119329960568\\
115	0.00346986931488591\\
116	0.00344204856808241\\
117	0.00340326803313504\\
118	0.00336902181126216\\
119	0.00335004723499187\\
120	0.00330893282021375\\
121	0.00326674639709128\\
122	0.00322337167522721\\
123	0.00317891146189458\\
124	0.00313363003482263\\
125	0.00308788675849191\\
126	0.00304207308203054\\
127	0.00299656154304081\\
128	0.00295167126070137\\
129	0.00290765050008365\\
130	0.00286467385887275\\
131	0.00282284980240456\\
132	0.00278223369667086\\
133	0.00274284196938743\\
134	0.00270466422337294\\
135	0.00266767163467527\\
136	0.00263182141849896\\
137	0.00259705825825687\\
138	0.00256331421122815\\
139	0.00253050870255337\\
140	0.00249854988199426\\
141	0.00246733800287661\\
142	0.00243677077735837\\
143	0.00240675004176157\\
144	0.00237718866060484\\
145	0.00234801647530183\\
146	0.0023191842634911\\
147	0.00229066506057331\\
148	0.00226245271077117\\
149	0.00223455805039677\\
150	0.00220700357869134\\
151	0.00217981776671422\\
152	0.00215303025796259\\
153	0.00212666913878508\\
154	0.00210076126592535\\
155	0.00207533645397907\\
156	0.00205043635439007\\
157	0.00202612949158261\\
158	0.00200253606161941\\
159	0.0019798722896144\\
160	0.00195854316617225\\
161	0.00193938260937477\\
162	0.0019245013065615\\
163	0.00192305922350097\\
164	0.0019002036205713\\
165	0.00187766220834494\\
166	0.00185551826148084\\
167	0.00183385679774958\\
168	0.00181275002287619\\
169	0.00179224470622998\\
170	0.00177235390461322\\
171	0.00175305441710425\\
172	0.00173429013826435\\
173	0.001715980348778\\
174	0.00169803115906359\\
175	0.00168034792988114\\
176	0.00166284655761613\\
177	0.00164546195938229\\
178	0.00162815278649907\\
179	0.00161090216756797\\
180	0.0015937149767794\\
181	0.00157661262120695\\
182	0.00155962658161499\\
183	0.00154279192355251\\
184	0.00152614176812901\\
185	0.00150970335537334\\
186	0.00149349593784486\\
187	0.0014775303885831\\
188	0.00146181015182455\\
189	0.00144633303288301\\
190	0.00143109331146196\\
191	0.00141608374431539\\
192	0.0014012971602965\\
193	0.00138672750355258\\
194	0.00137237031604077\\
195	0.00135822274785694\\
196	0.0013442832356796\\
197	0.00133055099994619\\
198	0.00131702549165691\\
199	0.00130370588420132\\
200	0.0012905906673732\\
201	0.00127767736896331\\
202	0.00126496240832405\\
203	0.00125244107599674\\
204	0.00124010763078811\\
205	0.00122795550625554\\
206	0.00121597761850511\\
207	0.00120416676416494\\
208	0.00119251609106007\\
209	0.00118101961606372\\
210	0.0011696727577154\\
211	0.00115847284881816\\
212	0.00114741959941073\\
213	0.0011365154955747\\
214	0.00112576614621228\\
215	0.00111518063030346\\
216	0.00110477195591491\\
217	0.00109455783125102\\
218	0.00108456209611806\\
219	0.00107481743790295\\
220	0.00106537058724267\\
221	0.0010562924928541\\
222	0.00104769927637944\\
223	0.00103979925013981\\
224	0.0010330136845603\\
225	0.00102836340564213\\
226	0.0010293614789304\\
227	0.00102089044824646\\
228	0.00101248860016131\\
229	0.0010041327581557\\
230	0.000995807326906321\\
231	0.000987504705732761\\
232	0.000979224736552839\\
233	0.000970973387877851\\
234	0.000962760947483496\\
235	0.00095460002170995\\
236	0.000946503621531227\\
237	0.000938483562854159\\
238	0.000930549332897086\\
239	0.000922707489403155\\
240	0.00091496157773166\\
241	0.000907312483157587\\
242	0.000899759089197776\\
243	0.000892299090790269\\
244	0.000884929813189066\\
245	0.00087764890981819\\
246	0.000870454849125546\\
247	0.000863347144620193\\
248	0.000856326326676902\\
249	0.000849393693189093\\
250	0.000842550904228462\\
251	0.000835799501035949\\
252	0.000829140431533029\\
253	0.000822573654586154\\
254	0.000816097876398424\\
255	0.000809710448396212\\
256	0.000803407430816219\\
257	0.000797183803502835\\
258	0.000791033788018203\\
259	0.000784951234761437\\
260	0.000778930025901325\\
261	0.000772964448940122\\
262	0.00076704950514167\\
263	0.000761181129808998\\
264	0.00075535631522412\\
265	0.00074957313986318\\
266	0.000743830717620126\\
267	0.000738129087188602\\
268	0.000732469064149508\\
269	0.000726852077030866\\
270	0.000721280004513431\\
271	0.000715755025228415\\
272	0.000710279485501729\\
273	0.000704855785086261\\
274	0.00069948627722554\\
275	0.000694173177751457\\
276	0.000688918478332905\\
277	0.000683723861087915\\
278	0.000678590614880186\\
279	0.000673519556951821\\
280	0.000668510966310746\\
281	0.000663564536867984\\
282	0.000658679358330343\\
283	0.000653853931227726\\
284	0.00064908621941849\\
285	0.000644373739443923\\
286	0.000639713681825102\\
287	0.000635103055479432\\
288	0.000630538843487951\\
289	0.000626018156910931\\
290	0.00062153837344579\\
291	0.000617097249421385\\
292	0.000612692996664546\\
293	0.000608324319715806\\
294	0.000603990413163946\\
295	0.000599690922941777\\
296	0.00059542587877481\\
297	0.000591195607221984\\
298	0.000587000635696363\\
299	0.000582841597493521\\
300	0.000578719146354547\\
301	0.000574633886757178\\
302	0.00057058632335427\\
303	0.000566576830175291\\
304	0.000562605637751507\\
305	0.000558672834502449\\
306	0.000554778377702269\\
307	0.000550922109161694\\
308	0.000547103771330205\\
309	0.000543323020657146\\
310	0.000539579436503389\\
311	0.000535872525400181\\
312	0.000532201721766988\\
313	0.000528566387133469\\
314	0.000524965810352076\\
315	0.00052139921121396\\
316	0.000517865749344018\\
317	0.000514364539376394\\
318	0.000510894672361345\\
319	0.000507455242305095\\
320	0.000504045375855969\\
321	0.000500664262554382\\
322	0.000497311182828112\\
323	0.000493985531060585\\
324	0.000490686831544901\\
325	0.00048741474587931\\
326	0.00048416907124532\\
327	0.000480949729916485\\
328	0.000477756751158064\\
329	0.000474590247302101\\
330	0.000471450386162586\\
331	0.0004683373620694\\
332	0.000465251367665941\\
333	0.000462192568280872\\
334	0.000459161080217938\\
335	0.000456156953784271\\
336	0.000453180161372784\\
337	0.000450230590484725\\
338	0.000447308041264964\\
339	0.000444412227942169\\
340	0.000441542783507247\\
341	0.000438699267007157\\
342	0.000435881172935054\\
343	0.000433087942322321\\
344	0.000430318975245916\\
345	0.000427573644525283\\
346	0.000424851310383004\\
347	0.000422151335781841\\
348	0.000419473102041474\\
349	0.000416816024207609\\
350	0.0004141795655204\\
351	0.000411563250243963\\
352	0.000408966674095023\\
353	0.000406389511565191\\
354	0.00040383151957179\\
355	0.000401292537086605\\
356	0.000398772480662711\\
357	0.000396271336077279\\
358	0.000393789146599415\\
359	0.000391325998647071\\
360	0.000388882005783499\\
361	0.000386457292107571\\
362	0.000384051976097642\\
363	0.000381666155883658\\
364	0.000379299896755451\\
365	0.000376953221489288\\
366	0.000374626103816857\\
367	0.000372318465098895\\
368	0.000370030174026084\\
369	0.000367761048976187\\
370	0.000365510862521871\\
371	0.000363279347515484\\
372	0.000361066204173107\\
373	0.000358871107631265\\
374	0.000356693715540099\\
375	0.00035453367537022\\
376	0.000352390631226445\\
377	0.000350264230066465\\
378	0.000348154127301601\\
379	0.000346059991804534\\
380	0.000343981510364418\\
381	0.000341918391615065\\
382	0.00033987036942691\\
383	0.000337837205709315\\
384	0.000335818692524662\\
385	0.000333814653386146\\
386	0.000331824943599813\\
387	0.000329849449525301\\
388	0.000327888086672223\\
389	0.000325940796610717\\
390	0.000324007542755642\\
391	0.000322088305170149\\
392	0.000320183074617899\\
393	0.000318291846164926\\
394	0.000316414612680801\\
395	0.000314551358611512\\
396	0.000312702054388502\\
397	0.00031086665180001\\
398	0.000309045080587073\\
399	0.000307237246441107\\
400	0.000305443030484161\\
401	0.000303662290210886\\
402	0.00030189486177647\\
403	0.000300140563431269\\
404	0.000298399199838602\\
405	0.000296670566970357\\
406	0.000294954457257475\\
407	0.000293250664678602\\
408	0.000291558989497011\\
409	0.000289879242400832\\
410	0.00028821124785687\\
411	0.000286554846552583\\
412	0.000284909896864669\\
413	0.000283276275354466\\
414	0.000281653876345322\\
415	0.000280042610682438\\
416	0.00027844240380921\\
417	0.000276853193317623\\
418	0.000275274926140384\\
419	0.000273707555553932\\
420	0.000272151038152714\\
421	0.000270605330940647\\
422	0.000269070388664142\\
423	0.00026754616148846\\
424	0.000266032593091681\\
425	0.000264529619225974\\
426	0.000263037166769942\\
427	0.00026155515327239\\
428	0.000260083486967306\\
429	0.000258622067222904\\
430	0.00025717078537231\\
431	0.000255729525864858\\
432	0.000254298167668353\\
433	0.00025287658585107\\
434	0.000251464653270746\\
435	0.000250062242300831\\
436	0.000248669226528424\\
437	0.000247285482364938\\
438	0.000245910890517198\\
439	0.000244545337274882\\
440	0.000243188715578942\\
441	0.000241840925843053\\
442	0.000240501876509322\\
443	0.000239171484327925\\
444	0.000237849674357044\\
445	0.000236536379690091\\
446	0.000235231540923109\\
447	0.000233935105384372\\
448	0.000232647026156441\\
449	0.000231367260927508\\
450	0.000230095770716133\\
451	0.000228832518518941\\
452	0.00022757746793443\\
453	0.00022633058181851\\
454	0.000225091821025888\\
455	0.000223861143289492\\
456	0.000222638502283278\\
457	0.000221423846906337\\
458	0.000220217120815042\\
459	0.000219018262218731\\
460	0.000217827203941159\\
461	0.000216643873736792\\
462	0.000215468194839342\\
463	0.000214300086708658\\
464	0.000213139465933571\\
465	0.000211986247242555\\
466	0.000210840344570496\\
467	0.000209701672130457\\
468	0.000208570145441379\\
469	0.000207445682268222\\
470	0.000206328203438213\\
471	0.000205217633505372\\
472	0.000204113901244439\\
473	0.000203016939965129\\
474	0.000201926687646072\\
475	0.000200843086895573\\
476	0.000199766084753744\\
477	0.000198695632354666\\
478	0.000197631684471867\\
479	0.000196574198972326\\
480	0.000195523136205369\\
481	0.000194478458352904\\
482	0.000193440128766324\\
483	0.000192408111314612\\
484	0.000191382369765314\\
485	0.000190362867218763\\
486	0.000189349565612515\\
487	0.000188342425311109\\
488	0.000187341404792339\\
489	0.000186346460438889\\
490	0.000185357546440048\\
491	0.000184374614804864\\
492	0.000183397615484174\\
493	0.000182426496594796\\
494	0.000181461204735693\\
495	0.000180501685381272\\
496	0.000179547883334928\\
497	0.000178599743222072\\
498	0.000177657210000652\\
499	0.000176720229466042\\
500	0.00017578874872725\\
};
\addlegendentry{Alg.~3 $p=0.8$}

\addplot [color=red,thick]
  table[row sep=crcr]{%
1	0.907327087669568\\
2	0.66926629012926\\
3	0.513696758296201\\
4	0.420271973475759\\
5	0.356736413766497\\
6	0.310861946373426\\
7	0.275161852097026\\
8	0.246278871555385\\
9	0.222416272855075\\
10	0.202539184859756\\
11	0.185379957986101\\
12	0.170540523999164\\
13	0.157491840117199\\
14	0.146057582420245\\
15	0.135989888706732\\
16	0.126939899490609\\
17	0.11885789072009\\
18	0.11162945604849\\
19	0.105030110599189\\
20	0.0989240218722269\\
21	0.093135860673671\\
22	0.0876381068998481\\
23	0.0824609404418295\\
24	0.0776041349443269\\
25	0.0729188321958541\\
26	0.0684402273901982\\
27	0.0640402490143665\\
28	0.0598195538870634\\
29	0.0556757585104326\\
30	0.0516884535578144\\
31	0.0479675733702995\\
32	0.0444813734652687\\
33	0.0412409080540029\\
34	0.0382727167810079\\
35	0.035541594039765\\
36	0.0330310073383411\\
37	0.0307066950028531\\
38	0.0285981080841002\\
39	0.0266136262863964\\
40	0.0248606808617784\\
41	0.0232293400615669\\
42	0.0217829774417735\\
43	0.0204163302800679\\
44	0.0191727450669179\\
45	0.0180911125059924\\
46	0.0170884705360923\\
47	0.0161710180453693\\
48	0.0153445434175888\\
49	0.0145671573736639\\
50	0.0138319297617201\\
51	0.0131633691002155\\
52	0.0125400014224166\\
53	0.0119716114986353\\
54	0.0114392644870438\\
55	0.010939853123446\\
56	0.0104708243940283\\
57	0.0100300404521782\\
58	0.00961566077729429\\
59	0.00921343708652588\\
60	0.00884739957958445\\
61	0.00850322625027361\\
62	0.00817957633632834\\
63	0.00787516870047791\\
64	0.00757752290520972\\
65	0.00730824776772011\\
66	0.00704397180871204\\
67	0.00679491003634911\\
68	0.00656008961050553\\
69	0.00634868907064756\\
70	0.00614936560849588\\
71	0.00596134677302811\\
72	0.00578390071356823\\
73	0.0055976657888924\\
74	0.00543962054747739\\
75	0.00529010038606513\\
76	0.00513963051264776\\
77	0.00499673000069473\\
78	0.00486932204322952\\
79	0.00473963378354768\\
80	0.00462407670722754\\
81	0.00451366629571184\\
82	0.00439198715845866\\
83	0.00429097410829664\\
84	0.0041941012662893\\
85	0.00409349612812783\\
86	0.00400443063609933\\
87	0.00391902389373802\\
88	0.00383719489416364\\
89	0.00375888202913637\\
90	0.00368403036564537\\
91	0.00361257989141623\\
92	0.00354445570502357\\
93	0.00346579882020835\\
94	0.00339741310076555\\
95	0.00333885536891877\\
96	0.00328308370995902\\
97	0.0032299034677515\\
98	0.00317267196431501\\
99	0.00312413187490485\\
100	0.0030775579184688\\
101	0.00303275529874037\\
102	0.00298343840983537\\
103	0.00294175831701437\\
104	0.0029014062248925\\
105	0.00286229199543383\\
106	0.00282435473820208\\
107	0.00278182648658571\\
108	0.00274623319317179\\
109	0.00271177268465386\\
110	0.00267294076035103\\
111	0.00264085493638144\\
112	0.00260994317060181\\
113	0.00258021106056663\\
114	0.00254640379095198\\
115	0.00251906371470422\\
116	0.00249284527237119\\
117	0.00246265407380852\\
118	0.0024385804917933\\
119	0.00241545431585121\\
120	0.00239320157686683\\
121	0.00236691504671018\\
122	0.00234619478305072\\
123	0.00232609246793919\\
124	0.00230650322912097\\
125	0.00228732915106416\\
126	0.00226848983537143\\
127	0.00225428699941512\\
128	0.00223381787100553\\
129	0.00220928515582633\\
130	0.00217844251834224\\
131	0.00213450582887935\\
132	0.00208849195022948\\
133	0.00203698961072262\\
134	0.0019828493800557\\
135	0.00192936365961771\\
136	0.00187961921693414\\
137	0.00183599885096478\\
138	0.00179411605872658\\
139	0.00176574603627407\\
140	0.0017444423744856\\
141	0.00172873790219575\\
142	0.00171675427857998\\
143	0.00170655135903417\\
144	0.00169643100716233\\
145	0.00168515219735695\\
146	0.00167203567062213\\
147	0.00165695822415363\\
148	0.0016402552235979\\
149	0.00162256252944946\\
150	0.00160463448297254\\
151	0.00158717303269656\\
152	0.00157166017966971\\
153	0.00155188750000585\\
154	0.00153889458792277\\
155	0.00152618354464699\\
156	0.0015132966457121\\
157	0.00149992609816074\\
158	0.00148591155418549\\
159	0.00146574798314922\\
160	0.00145049579997465\\
161	0.00143479351043907\\
162	0.0014188213029821\\
163	0.00140278324597677\\
164	0.00138688884672219\\
165	0.00137133538271557\\
166	0.00135629332525808\\
167	0.00134189635867252\\
168	0.00132823659970642\\
169	0.00131536474236464\\
170	0.00130329414528966\\
171	0.0012920074427583\\
172	0.00128146413558134\\
173	0.00127160778350285\\
174	0.00126237180606411\\
175	0.00125368340385309\\
176	0.00124546562714931\\
177	0.00123763804625089\\
178	0.00123011674690527\\
179	0.001222814452019\\
180	0.00121564146290145\\
181	0.00120349690993031\\
182	0.0011963285615507\\
183	0.00118903364829388\\
184	0.00118154703682381\\
185	0.00117382200554224\\
186	0.00116583437276453\\
187	0.00115758498797468\\
188	0.00114910012194468\\
189	0.00114042956903672\\
190	0.00113164258591518\\
191	0.00112282207739448\\
192	0.00111405765944791\\
193	0.00110543835244494\\
194	0.0010970456724758\\
195	0.00108894780058095\\
196	0.00108119533923011\\
197	0.00107381894334211\\
198	0.00106682887508942\\
199	0.00106021631208879\\
200	0.00105395606541897\\
201	0.00104801025535731\\
202	0.0010423324553051\\
203	0.00103687184333167\\
204	0.00103157698222586\\
205	0.00102639896320448\\
206	0.00102129377385052\\
207	0.00101622386754573\\
208	0.00101115900440031\\
209	0.00100134168863171\\
210	0.000996236157819283\\
211	0.000986380906962027\\
212	0.000981198698336018\\
213	0.000975975207617824\\
214	0.000970718620115411\\
215	0.000965440940507788\\
216	0.00096015725896985\\
217	0.00095023452430271\\
218	0.000944998907393517\\
219	0.000939811436428581\\
220	0.00093469015008929\\
221	0.000925037696225392\\
222	0.00092010216150629\\
223	0.000915274044254151\\
224	0.000910561021834651\\
225	0.000901395268261702\\
226	0.000896929478480898\\
227	0.000892582669101026\\
228	0.000888350564384965\\
229	0.000884226953983552\\
230	0.000880204370119226\\
231	0.000876274621310574\\
232	0.000872429215887824\\
233	0.000864167230287798\\
234	0.000860475041976829\\
235	0.000856842599288218\\
236	0.000853261743065473\\
237	0.00084972443554066\\
238	0.00084622285802525\\
239	0.000842749514122075\\
240	0.000839297407844728\\
241	0.000835860291562726\\
242	0.000832432961024155\\
243	0.0008290115612285\\
244	0.000825593859816912\\
245	0.000822179444935919\\
246	0.000818769811949355\\
247	0.000815368316515467\\
248	0.000811979988139869\\
249	0.000808611215663212\\
250	0.000805269331574496\\
251	0.00080196213334219\\
252	0.000798697385698417\\
253	0.000795482347514358\\
254	0.00079232336105591\\
255	0.000789225531305808\\
256	0.000786192510521195\\
257	0.00078322639031759\\
258	0.000780327692246488\\
259	0.000777495439547816\\
260	0.000770418268198838\\
261	0.000767711485803537\\
262	0.000765061336335252\\
263	0.000762462582774301\\
264	0.000759909269993848\\
265	0.000753094553970293\\
266	0.000750614673777093\\
267	0.000748161021364211\\
268	0.000745726889641897\\
269	0.000743305814588105\\
270	0.000740891892363441\\
271	0.000734196123400687\\
272	0.000731786242622727\\
273	0.000729372485943659\\
274	0.000726953700766728\\
275	0.000724530325014712\\
276	0.000722104439924048\\
277	0.000719679688625923\\
278	0.000713006446644317\\
279	0.000714863747952051\\
280	0.000708200111478836\\
281	0.000705690874798239\\
282	0.000702788734905516\\
283	0.000698766882212471\\
284	0.000692385465674789\\
285	0.000677336753595237\\
286	0.000661773115098978\\
287	0.000640681217766463\\
288	0.000614325543695421\\
289	0.000583753201867475\\
290	0.000550591524127743\\
291	0.00051679396047823\\
292	0.000484369827024368\\
293	0.000449868679041796\\
294	0.000419966588141991\\
295	0.0004009296363511\\
296	0.000387917971695366\\
297	0.000380823588629207\\
298	0.00037909554511921\\
299	0.000381849159954767\\
300	0.000387990872896137\\
301	0.000396341384764983\\
302	0.000405743977644816\\
303	0.000415149695164525\\
304	0.000423675840498971\\
305	0.000430638321321017\\
306	0.000435561302096562\\
307	0.000438169265434172\\
308	0.000438367034108295\\
309	0.000436212836710184\\
310	0.000433666255885578\\
311	0.000430793711265498\\
312	0.000427269479043573\\
313	0.000422932951073929\\
314	0.000417769074146162\\
315	0.000411874017151299\\
316	0.000405414367029004\\
317	0.000398587294087755\\
318	0.000391587345351059\\
319	0.000384583181839898\\
320	0.000377705179758049\\
321	0.000371042791128633\\
322	0.00036464918638047\\
323	0.000353266272954752\\
324	0.000347471517725911\\
325	0.00034198025571457\\
326	0.000336788449300466\\
327	0.000331891275063119\\
328	0.000327283839364041\\
329	0.000322960122290014\\
330	0.000318911373758874\\
331	0.000315124889981774\\
332	0.000306341091417785\\
333	0.000303036367495145\\
334	0.00029993643303765\\
335	0.000297018175423009\\
336	0.000294259554253924\\
337	0.000291640324088288\\
338	0.000289141734521067\\
339	0.000286745327343801\\
340	0.00028443121667888\\
341	0.000282176419398426\\
342	0.00027995383916205\\
343	0.000277732390888037\\
344	0.000275478511636326\\
345	0.00027315899403178\\
346	0.000270744768424709\\
347	0.000268215017176115\\
348	0.000265560880227002\\
349	0.000262788031298858\\
350	0.00025991756479004\\
351	0.000256984903155214\\
352	0.000254036761559738\\
353	0.000251126529857863\\
354	0.000248308693170054\\
355	0.000245633066750029\\
356	0.000243139644707862\\
357	0.000240854755949867\\
358	0.000238789007534108\\
359	0.000236937215790562\\
360	0.000235280228215191\\
361	0.000233788273283842\\
362	0.000232425281278365\\
363	0.000231153522500928\\
364	0.000229937917502827\\
365	0.000228749477210148\\
366	0.000227567504794647\\
367	0.000226380108795918\\
368	0.000225182644786822\\
369	0.000223976366488481\\
370	0.000222766414688729\\
371	0.000221559451524791\\
372	0.000220361354952433\\
373	0.000219175334383027\\
374	0.000218000726466926\\
375	0.000216832600961959\\
376	0.000215662173282604\\
377	0.000214477903144336\\
378	0.000213267072886171\\
379	0.000212017592773698\\
380	0.000210719774791814\\
381	0.000209367845802437\\
382	0.000207961025788935\\
383	0.000206504065537455\\
384	0.000205007209072841\\
385	0.000203485610088032\\
386	0.00020195828223791\\
387	0.000200446697684651\\
388	0.000194060521087469\\
389	0.000192653657228488\\
390	0.000191325086992381\\
391	0.000190089581855426\\
392	0.000188956564568384\\
393	0.000187929463004671\\
394	0.000187005316580552\\
395	0.000181313398201242\\
396	0.000180566312058624\\
397	0.000179876575617346\\
398	0.00017922053781517\\
399	0.00017857253964341\\
400	0.00017790680054371\\
401	0.000177199222469636\\
402	0.000176429087810771\\
403	0.000175580517342242\\
404	0.000174643571225656\\
405	0.00017361490442554\\
406	0.000172497925292794\\
407	0.000171302448809583\\
408	0.000170043879619594\\
409	0.000168741999841147\\
410	0.000162582634070251\\
411	0.000161260028911587\\
412	0.000159964008430664\\
413	0.000158716390844272\\
414	0.0001575354078809\\
415	0.000156434763105031\\
416	0.000155422939559191\\
417	0.00015450290899187\\
418	0.000153672233521068\\
419	0.000152923521458329\\
420	0.000152245176466684\\
421	0.000151622364360759\\
422	0.000151038114797391\\
423	0.00015047447516024\\
424	0.000149913639853101\\
425	0.000149338988441867\\
426	0.000148735979041996\\
427	0.000148092857625638\\
428	0.000147401158376583\\
429	0.000146655984036128\\
430	0.000145856067820393\\
431	0.000145003629640981\\
432	0.000144104048845281\\
433	0.000143165383393167\\
434	0.000142197771196699\\
435	0.00014121275315639\\
436	0.000140222559133691\\
437	0.000139239397649277\\
438	0.000138274787529247\\
439	0.000137338965200655\\
440	0.000136440395163407\\
441	0.000135585403777002\\
442	0.000134777948419838\\
443	0.000134019525881849\\
444	0.000133309216079815\\
445	0.000132643850309054\\
446	0.000132018287604471\\
447	0.000131425778569267\\
448	0.000130858393287452\\
449	0.000130307488583831\\
450	0.000129764189768823\\
451	0.000129219862902132\\
452	0.000128666555326024\\
453	0.000128097384592135\\
454	0.000127506858822455\\
455	0.000126891114936554\\
456	0.000126248065012474\\
457	0.000125577445289585\\
458	0.000124880766911452\\
459	0.000124161172327409\\
460	0.000123423206134345\\
461	0.000122672513798855\\
462	0.000121915485841029\\
463	0.000121158868363598\\
464	0.000120409362964882\\
465	0.000119673239833709\\
466	0.00011417246972938\\
467	0.000113478489140735\\
468	0.000112810845394563\\
469	0.000112171403202584\\
470	0.000111560409271235\\
471	0.000110976682389076\\
472	0.000110417747821667\\
473	0.000109880045730764\\
474	0.00010935919724263\\
475	0.000104071989690415\\
476	0.000103570240434336\\
477	0.000103070480784897\\
478	0.000102568312986837\\
479	0.000102059862639732\\
480	0.000101542102694989\\
481	0.00010101298062349\\
482	0.000100471476523532\\
483	9.99175928570444e-05\\
484	9.93522810074973e-05\\
485	9.87773136962768e-05\\
486	9.81951152519318e-05\\
487	9.7608563654825e-05\\
488	9.70207791301646e-05\\
489	9.64349138521704e-05\\
490	9.58539561576131e-05\\
491	9.52805606914738e-05\\
492	9.47169133277883e-05\\
493	9.41646367344125e-05\\
494	9.36247393172294e-05\\
495	9.30976072061902e-05\\
496	9.2583036132611e-05\\
497	9.20802976662402e-05\\
498	9.15882324574631e-05\\
499	9.11053619496397e-05\\
500	9.06300095211907e-05\\
};
\addlegendentry{Alg.~4 $p=0.8$}

\addplot [color=black, dashed,thick]
  table[row sep=crcr]{%
1	100\\
2	25\\
3	11.1111111111111\\
4	6.25\\
5	4\\
6	2.77777777777778\\
7	2.04081632653061\\
8	1.5625\\
9	1.23456790123457\\
10	1\\
11	0.826446280991736\\
12	0.694444444444444\\
13	0.591715976331361\\
14	0.510204081632653\\
15	0.444444444444444\\
16	0.390625\\
17	0.346020761245675\\
18	0.308641975308642\\
19	0.277008310249307\\
20	0.25\\
21	0.226757369614512\\
22	0.206611570247934\\
23	0.189035916824197\\
24	0.173611111111111\\
25	0.16\\
26	0.14792899408284\\
27	0.137174211248285\\
28	0.127551020408163\\
29	0.118906064209275\\
30	0.111111111111111\\
31	0.104058272632674\\
32	0.09765625\\
33	0.0918273645546373\\
34	0.0865051903114187\\
35	0.0816326530612245\\
36	0.0771604938271605\\
37	0.0730460189919649\\
38	0.0692520775623269\\
39	0.0657462195923734\\
40	0.0625\\
41	0.0594883997620464\\
42	0.0566893424036281\\
43	0.0540832882639264\\
44	0.0516528925619835\\
45	0.0493827160493827\\
46	0.0472589792060491\\
47	0.0452693526482571\\
48	0.0434027777777778\\
49	0.041649312786339\\
50	0.04\\
51	0.0384467512495194\\
52	0.0369822485207101\\
53	0.0355998576005696\\
54	0.0342935528120713\\
55	0.0330578512396694\\
56	0.0318877551020408\\
57	0.0307787011388119\\
58	0.0297265160523187\\
59	0.0287273771904625\\
60	0.0277777777777778\\
61	0.0268744961031981\\
62	0.0260145681581686\\
63	0.0251952632905014\\
64	0.0244140625\\
65	0.0236686390532544\\
66	0.0229568411386593\\
67	0.0222766763198931\\
68	0.0216262975778547\\
69	0.0210039907582441\\
70	0.0204081632653061\\
71	0.0198373338623289\\
72	0.0192901234567901\\
73	0.0187652467629949\\
74	0.0182615047479912\\
75	0.0177777777777778\\
76	0.0173130193905817\\
77	0.0168662506324844\\
78	0.0164365548980934\\
79	0.0160230732254446\\
80	0.015625\\
81	0.0152415790275873\\
82	0.0148720999405116\\
83	0.0145158949049209\\
84	0.014172335600907\\
85	0.013840830449827\\
86	0.0135208220659816\\
87	0.0132117849121416\\
88	0.0129132231404959\\
89	0.0126246686024492\\
90	0.0123456790123457\\
91	0.0120758362516604\\
92	0.0118147448015123\\
93	0.0115620302925194\\
94	0.0113173381620643\\
95	0.0110803324099723\\
96	0.0108506944444444\\
97	0.0106281220108407\\
98	0.0104123281965848\\
99	0.0102030405060708\\
100	0.01\\
101	0.00980296049406921\\
102	0.00961168781237985\\
103	0.00942595909133754\\
104	0.00924556213017752\\
105	0.0090702947845805\\
106	0.0088999644001424\\
107	0.00873438728273212\\
108	0.00857338820301783\\
109	0.0084167999326656\\
110	0.00826446280991736\\
111	0.00811622433244055\\
112	0.0079719387755102\\
113	0.00783146683373796\\
114	0.00769467528470299\\
115	0.00756143667296786\\
116	0.00743162901307967\\
117	0.00730513551026372\\
118	0.00718184429761563\\
119	0.00706164818868724\\
120	0.00694444444444444\\
121	0.00683013455365071\\
122	0.00671862402579952\\
123	0.00660982219578293\\
124	0.00650364203954214\\
125	0.0064\\
126	0.00629881582262535\\
127	0.0062000124000248\\
128	0.006103515625\\
129	0.00600925425154738\\
130	0.00591715976331361\\
131	0.00582716624905309\\
132	0.00573921028466483\\
133	0.00565323082141444\\
134	0.00556916907997327\\
135	0.00548696844993141\\
136	0.00540657439446367\\
137	0.00532793435984869\\
138	0.00525099768956102\\
139	0.00517571554267377\\
140	0.00510204081632653\\
141	0.00502992807202857\\
142	0.00495933346558223\\
143	0.00489021468042447\\
144	0.00482253086419753\\
145	0.00475624256837099\\
146	0.00469131169074873\\
147	0.00462770142070434\\
148	0.00456537618699781\\
149	0.00450430160803567\\
150	0.00444444444444444\\
151	0.00438577255383536\\
152	0.00432825484764543\\
153	0.0042718612499466\\
154	0.0042165626581211\\
155	0.00416233090530697\\
156	0.00410913872452334\\
157	0.00405695971439004\\
158	0.00400576830636116\\
159	0.00395553973339662\\
160	0.00390625\\
161	0.00385787585355503\\
162	0.00381039475689681\\
163	0.00376378486205728\\
164	0.0037180249851279\\
165	0.00367309458218549\\
166	0.00362897372623022\\
167	0.00358564308508731\\
168	0.00354308390022676\\
169	0.00350127796645776\\
170	0.00346020761245675\\
171	0.00341985568209022\\
172	0.0033802055164954\\
173	0.00334124093688396\\
174	0.00330294622803541\\
175	0.00326530612244898\\
176	0.00322830578512397\\
177	0.00319193079894028\\
178	0.0031561671506123\\
179	0.00312100121719047\\
180	0.00308641975308642\\
181	0.00305240987759836\\
182	0.00301895906291511\\
183	0.00298605512257756\\
184	0.00295368620037807\\
185	0.0029218407596786\\
186	0.00289050757312984\\
187	0.00285967571277417\\
188	0.00282933454051607\\
189	0.0027994736989446\\
190	0.00277008310249307\\
191	0.0027411529289219\\
192	0.00271267361111111\\
193	0.00268463582914978\\
194	0.00265703050271017\\
195	0.00262984878369494\\
196	0.00260308204914619\\
197	0.00257672189440594\\
198	0.0025507601265177\\
199	0.00252518875785965\\
200	0.0025\\
201	0.0024751862577659\\
202	0.0024507401235173\\
203	0.00242665437161785\\
204	0.00240292195309496\\
205	0.00237953599048186\\
206	0.00235648977283439\\
207	0.00233377675091601\\
208	0.00231139053254438\\
209	0.00228932487809345\\
210	0.00226757369614512\\
211	0.00224613103928483\\
212	0.0022249911000356\\
213	0.00220414820692543\\
214	0.00218359682068303\\
215	0.00216333153055706\\
216	0.00214334705075446\\
217	0.00212363821699335\\
218	0.0021041999831664\\
219	0.00208502741811055\\
220	0.00206611570247934\\
221	0.00204746012571405\\
222	0.00202905608311014\\
223	0.00201089907297553\\
224	0.00199298469387755\\
225	0.00197530864197531\\
226	0.00195786670843449\\
227	0.00194065477692173\\
228	0.00192366882117575\\
229	0.0019069049026525\\
230	0.00189035916824197\\
231	0.00187402784805382\\
232	0.00185790725326992\\
233	0.00184199377406104\\
234	0.00182628387756593\\
235	0.00181077410593029\\
236	0.00179546107440391\\
237	0.00178034146949385\\
238	0.00176541204717181\\
239	0.00175066963113391\\
240	0.00173611111111111\\
241	0.00172173344122863\\
242	0.00170753363841268\\
243	0.00169350878084303\\
244	0.00167965600644988\\
245	0.00166597251145356\\
246	0.00165245554894573\\
247	0.0016391024275107\\
248	0.00162591050988554\\
249	0.00161287721165788\\
250	0.0016\\
251	0.00158727639243822\\
252	0.00157470395565634\\
253	0.0015622803043322\\
254	0.0015500031000062\\
255	0.00153787004998078\\
256	0.00152587890625\\
257	0.00151402746445821\\
258	0.00150231356288685\\
259	0.00149073508146867\\
260	0.0014792899408284\\
261	0.00146797610134907\\
262	0.00145679156226327\\
263	0.00144573436076855\\
264	0.00143480257116621\\
265	0.00142399430402278\\
266	0.00141330770535361\\
267	0.00140274095582769\\
268	0.00139229226999332\\
269	0.00138195989552383\\
270	0.00137174211248285\\
271	0.00136163723260849\\
272	0.00135164359861592\\
273	0.00134175958351783\\
274	0.00133198358996217\\
275	0.00132231404958678\\
276	0.00131274942239025\\
277	0.00130328819611881\\
278	0.00129392888566844\\
279	0.00128467003250215\\
280	0.00127551020408163\\
281	0.00126644799331315\\
282	0.00125748201800714\\
283	0.00124861092035111\\
284	0.00123983336639556\\
285	0.00123114804555248\\
286	0.00122255367010612\\
287	0.00121404897473564\\
288	0.00120563271604938\\
289	0.00119730367213036\\
290	0.00118906064209275\\
291	0.00118090244564896\\
292	0.00117282792268718\\
293	0.00116483593285886\\
294	0.00115692535517608\\
295	0.0011490950876185\\
296	0.00114134404674945\\
297	0.0011336711673412\\
298	0.00112607540200892\\
299	0.00111855572085323\\
300	0.00111111111111111\\
301	0.00110374057681483\\
302	0.00109644313845884\\
303	0.00108921783267436\\
304	0.00108206371191136\\
305	0.00107497984412792\\
306	0.00106796531248665\\
307	0.00106101921505798\\
308	0.00105414066453027\\
309	0.00104732878792639\\
310	0.00104058272632674\\
311	0.00103390163459848\\
312	0.00102728468113083\\
313	0.00102073104757627\\
314	0.00101423992859751\\
315	0.00100781053162006\\
316	0.00100144207659029\\
317	0.000995133795738837\\
318	0.000988884933349156\\
319	0.000982694745531196\\
320	0.0009765625\\
321	0.000970487475859124\\
322	0.000964468963388758\\
323	0.000958506263838434\\
324	0.000952598689224204\\
325	0.000946745562130178\\
326	0.000940946215514321\\
327	0.0009351999925184\\
328	0.000929506246281975\\
329	0.00092386433976035\\
330	0.000918273645546373\\
331	0.000912733545696005\\
332	0.000907243431557555\\
333	0.000901802703604505\\
334	0.000896410771271828\\
335	0.000891067052795723\\
336	0.000885770975056689\\
337	0.000880521973425847\\
338	0.000875319491614439\\
339	0.00087016298152644\\
340	0.000865051903114187\\
341	0.000859985724236978\\
342	0.000854963920522554\\
343	0.000849985975231409\\
344	0.000845051379123851\\
345	0.000840159630329763\\
346	0.00083531023422099\\
347	0.000830502703286299\\
348	0.000825736557008852\\
349	0.000821011321746127\\
350	0.000816326530612245\\
351	0.000811681723362635\\
352	0.000807076446280992\\
353	0.00080251025206847\\
354	0.00079798269973507\\
355	0.000793493354493156\\
356	0.000789041787653074\\
357	0.000784627576520804\\
358	0.000780250304297619\\
359	0.000775909559981688\\
360	0.000771604938271605\\
361	0.000767336039471766\\
362	0.000763102469399591\\
363	0.000758903839294523\\
364	0.000754739765728777\\
365	0.000750609870519797\\
366	0.000746513780644391\\
367	0.000742451128154489\\
368	0.000738421550094518\\
369	0.000734424688420326\\
370	0.000730460189919649\\
371	0.000726527706134073\\
372	0.00072262689328246\\
373	0.000718757412185813\\
374	0.000714918928193543\\
375	0.000711111111111111\\
376	0.000707333635129018\\
377	0.000703586178753105\\
378	0.00069986842473615\\
379	0.000696180060010721\\
380	0.000692520775623269\\
381	0.000688890266669422\\
382	0.000685288232230476\\
383	0.000681714375311032\\
384	0.000678168402777778\\
385	0.000674650025299376\\
386	0.000671158957287444\\
387	0.000667694916838598\\
388	0.000664257625677543\\
389	0.000660846809101182\\
390	0.000657462195923734\\
391	0.000654103518422826\\
392	0.000650770512286547\\
393	0.000647462916561454\\
394	0.000644180473601484\\
395	0.000640922929017786\\
396	0.000637690031629426\\
397	0.00063448153341497\\
398	0.000631297189464912\\
399	0.000628136757934938\\
400	0.000625\\
401	0.000621886679809205\\
402	0.000618796564441474\\
403	0.000615729423861978\\
404	0.000612685030879326\\
405	0.00060966316110349\\
406	0.000606663592904463\\
407	0.000603686107371611\\
408	0.000600730488273741\\
409	0.000597796522019835\\
410	0.000594883997620464\\
411	0.000591992706649854\\
412	0.000589122443208596\\
413	0.00058627300388699\\
414	0.000583444187729002\\
415	0.000580635796196835\\
416	0.000577847633136095\\
417	0.000575079504741531\\
418	0.000572331219523363\\
419	0.000569602588274161\\
420	0.000566893424036281\\
421	0.000564203542069837\\
422	0.000561532759821208\\
423	0.000558880896892063\\
424	0.0005562477750089\\
425	0.00055363321799308\\
426	0.000551037051731358\\
427	0.000548459104146899\\
428	0.000545899205170757\\
429	0.00054335718671383\\
430	0.000540832882639265\\
431	0.00053832612873531\\
432	0.000535836762688615\\
433	0.000533364624057945\\
434	0.000530909554248338\\
435	0.000528471396485665\\
436	0.0005260499957916\\
437	0.000523645198958993\\
438	0.000521256854527637\\
439	0.000518884812760415\\
440	0.000516528925619835\\
441	0.000514189046744926\\
442	0.000511865031428513\\
443	0.000509556736594836\\
444	0.000507264020777534\\
445	0.000504986744097967\\
446	0.000502724768243882\\
447	0.000500477956448408\\
448	0.000498246173469388\\
449	0.00049602928556902\\
450	0.000493827160493827\\
451	0.000491639667454929\\
452	0.000489466677108622\\
453	0.000487308061537262\\
454	0.000485163694230433\\
455	0.000483033450066417\\
456	0.000480917205293937\\
457	0.000478814837514185\\
458	0.000476726225663126\\
459	0.000474651249994067\\
460	0.000472589792060492\\
461	0.000470541734699159\\
462	0.000468506962013456\\
463	0.000466485359356997\\
464	0.000464476813317479\\
465	0.000462481211700775\\
466	0.000460498443515261\\
467	0.000458528398956389\\
468	0.000456570969391482\\
469	0.000454626047344757\\
470	0.000452693526482571\\
471	0.000450773301598893\\
472	0.000448865268600977\\
473	0.00044696932449526\\
474	0.000445085367373462\\
475	0.000443213296398892\\
476	0.000441353011792952\\
477	0.000439504414821847\\
478	0.000437667407783477\\
479	0.000435841893994535\\
480	0.000434027777777778\\
481	0.000432224964449497\\
482	0.000430433360307157\\
483	0.000428652872617226\\
484	0.000426883409603169\\
485	0.000425124880433627\\
486	0.000423377195210757\\
487	0.000421640264958742\\
488	0.00041991400161247\\
489	0.000418198318006365\\
490	0.00041649312786339\\
491	0.000414798345784197\\
492	0.000413113887236433\\
493	0.000411439668544203\\
494	0.000409775606877674\\
495	0.000408121620242832\\
496	0.000406477627471384\\
497	0.000404843548210794\\
498	0.000403219302914469\\
499	0.000401604812832077\\
500	0.0004\\
};
\addlegendentry{$100/k^2$}

\end{axis}
\end{tikzpicture}%

%% file: media/resultCompressedSensingConstp0p8.tikz
%
%
\definecolor{mycolor1}{rgb}{0.00000,0.44700,0.74100}%
\definecolor{mycolor2}{rgb}{0.85000,0.32500,0.09800}%
\begin{tikzpicture}

\begin{axis}[%
width=0.951\figurewidth,
height=\figureheight,
at={(0\figurewidth,0\figureheight)},
scale only axis,
xmin=0,
xmax=500,
xlabel style={font=\color{white!15!black}},
xlabel={iterations},
xlabel near ticks,
ymode=log,
ymin=0.0001,
ymax=100,
yminorticks=true,
ylabel style={font=\color{white!15!black}},
ylabel={$|\nu-\sum_{i=1}^n (\bar{x}_{ki})_\Delta^p|$},
ylabel near ticks,
axis background/.style={fill=white},
legend style={legend cell align=left, align=left, draw=white!15!black}
]
\addplot [color=black,thick]
  table[row sep=crcr]{%
1	13\\
2	10.4915694705228\\
3	3.77177483625336\\
4	1.2269378918164\\
5	0.0241637844038634\\
6	0.567696712527381\\
7	9.87839970075877\\
8	7.39098009531555\\
9	5.52655552577425\\
10	4.1578354174006\\
11	3.23088487902934\\
12	2.49603185323626\\
13	1.94314191420547\\
14	1.52221683351028\\
15	1.17142349824829\\
16	0.902851524596824\\
17	0.675882421247387\\
18	0.529202600258141\\
19	0.362854274880808\\
20	0.241560282786134\\
21	0.15260624224649\\
22	0.0506527423585341\\
23	0.00280713209117778\\
24	0.0780870810744991\\
25	6.47865920242707\\
26	5.81626878856825\\
27	5.2325156968884\\
28	4.75529004608795\\
29	4.33584096089431\\
30	3.96604835619267\\
31	3.66248980155617\\
32	3.38978669159724\\
33	3.13008387413717\\
34	2.91975834466219\\
35	2.7143748843304\\
36	2.53919049849511\\
37	2.38064105067655\\
38	2.23186958990637\\
39	2.08459301646804\\
40	1.96537998936038\\
41	1.84409498371193\\
42	1.74613605649796\\
43	1.64941692137498\\
44	1.55545391183307\\
45	1.46843011896082\\
46	1.39574174040336\\
47	1.32928403961341\\
48	1.26863333003003\\
49	1.21520850999766\\
50	1.16407886391345\\
51	1.11134269997281\\
52	1.06199487114208\\
53	1.01393193080115\\
54	0.968456097958998\\
55	0.924838989875572\\
56	0.888191605374758\\
57	0.846008139575614\\
58	0.803076732326458\\
59	0.769135988131719\\
60	0.741325863044539\\
61	0.713539105554562\\
62	0.681331481077211\\
63	0.657620210156645\\
64	0.634424919573008\\
65	0.610664971867021\\
66	0.581173239487084\\
67	0.557086346512436\\
68	0.537058103119995\\
69	0.517590870439967\\
70	0.497468944549544\\
71	0.466659083405018\\
72	0.450204930240625\\
73	0.429627534117615\\
74	0.408577376207536\\
75	0.394874466808666\\
76	0.381363394224344\\
77	0.364013586095961\\
78	0.342797297039788\\
79	0.322046370470955\\
80	0.297076969318506\\
81	0.278674216069485\\
82	0.252373856737092\\
83	0.231911325992581\\
84	0.219962845316498\\
85	0.205057842689189\\
86	0.196551841132865\\
87	0.183157717511178\\
88	0.178204512785237\\
89	0.173322135831655\\
90	0.168510498327741\\
91	0.163651231473465\\
92	0.158157762917972\\
93	0.149982518419089\\
94	0.146257171490172\\
95	0.141459997349586\\
96	0.138332263863708\\
97	0.135299257614475\\
98	0.132359294874012\\
99	0.129515277665861\\
100	0.126760500981489\\
101	0.12408945611799\\
102	0.121497821454206\\
103	0.118982293105405\\
104	0.116540267781432\\
105	0.114169402113395\\
106	0.111867068904529\\
107	0.109629674725833\\
108	0.107451601375155\\
109	0.105322839990837\\
110	0.10322135402891\\
111	0.101075119773754\\
112	0.0978647730896036\\
113	0.0960516186415521\\
114	0.0941105489721672\\
115	0.0918656637726682\\
116	0.0891215668068519\\
117	0.0873113554359886\\
118	0.085022204543616\\
119	0.0824539179733591\\
120	0.0810601781067191\\
121	0.0797008460396811\\
122	0.0783752580350165\\
123	0.0770827223929148\\
124	0.0758224618516761\\
125	0.0745935925687189\\
126	0.0733951335552228\\
127	0.0722260368194705\\
128	0.0710852274335719\\
129	0.069971643684285\\
130	0.0688842697174881\\
131	0.0678221559919631\\
132	0.0667844258706331\\
133	0.0657702693731032\\
134	0.0647789271942936\\
135	0.0638096693848785\\
136	0.0628617735200079\\
137	0.0619345067827454\\
138	0.0610271152791189\\
139	0.0601388222921979\\
140	0.0592688353563398\\
141	0.0584163603110667\\
142	0.0575806191843674\\
143	0.0567608680839661\\
144	0.0559564113369834\\
145	0.0551666088479085\\
146	0.0543908748334164\\
147	0.0536286674328366\\
148	0.0528794698460525\\
149	0.0521427642942542\\
150	0.051417999961943\\
151	0.0507045549104387\\
152	0.0500016894320152\\
153	0.0493084838028545\\
154	0.0486237453783968\\
155	0.0479458546758162\\
156	0.0472724876609751\\
157	0.0466000744493923\\
158	0.0459226478404428\\
159	0.0452290909391201\\
160	0.044495335544819\\
161	0.0436552204188143\\
162	0.0424082875209001\\
163	0.0394145422769599\\
164	0.038930393705223\\
165	0.0384554668654052\\
166	0.0379895517634875\\
167	0.0375324169859999\\
168	0.0370838077615014\\
169	0.0366434498626127\\
170	0.0362110580057656\\
171	0.0357863467287994\\
172	0.0353690415063271\\
173	0.034958888081371\\
174	0.034555658558568\\
175	0.0341591535760492\\
176	0.0337692006853496\\
177	0.0333856497652731\\
178	0.0330083667576904\\
179	0.0326372271782348\\
180	0.0322721107207024\\
181	0.0319128978949609\\
182	0.0315594691084619\\
183	0.0312117060364333\\
184	0.030869494636527\\
185	0.0305327288417758\\
186	0.0302013138613919\\
187	0.0298751681404518\\
188	0.029554223337948\\
189	0.0292384221070133\\
190	0.0289277139108747\\
191	0.0286220494935593\\
192	0.0283213748744913\\
193	0.0280256258097061\\
194	0.0277347235562151\\
195	0.0274485725223714\\
196	0.0271670600452475\\
197	0.0268900581793311\\
198	0.0266174270789727\\
199	0.0263490193615482\\
200	0.026084684772969\\
201	0.0258242745337276\\
202	0.0255676448908316\\
203	0.0253146595911889\\
204	0.0250651911799506\\
205	0.02481912117461\\
206	0.0245763392527445\\
207	0.0243367416122495\\
208	0.024100228625417\\
209	0.0238667018232759\\
210	0.0236360601246732\\
211	0.0234081950642032\\
212	0.0231829845569561\\
213	0.0229602844193595\\
214	0.0227399163475958\\
215	0.0225216501484879\\
216	0.0223051763372366\\
217	0.022090061942049\\
218	0.0218756756247794\\
219	0.0216610534892223\\
220	0.0214446420680694\\
221	0.0212237639391535\\
222	0.0209933816141203\\
223	0.0207427847264239\\
224	0.0204445024634626\\
225	0.019999199493002\\
226	0.0181099801797786\\
227	0.0179499957829569\\
228	0.0177920408241837\\
229	0.0176360855194437\\
230	0.0174821073145713\\
231	0.017330088725257\\
232	0.0171800151234876\\
233	0.0170318727068004\\
234	0.0168856468269964\\
235	0.0167413207847933\\
236	0.0165988751260656\\
237	0.0164582874124533\\
238	0.0163195323900147\\
239	0.016182582447793\\
240	0.0160474082449082\\
241	0.0159139793881909\\
242	0.0157822650598443\\
243	0.0156522345214908\\
244	0.015523857452414\\
245	0.0153971041114502\\
246	0.0152719453397697\\
247	0.0151483524426963\\
248	0.0150262970013075\\
249	0.0149057506686017\\
250	0.0147866850005064\\
251	0.014669071361624\\
252	0.0145528809303621\\
253	0.0144380848113136\\
254	0.0143246542461979\\
255	0.0142125609008928\\
256	0.0141017771959695\\
257	0.0139922766435976\\
258	0.0138840341536274\\
259	0.0137770262764953\\
260	0.0136712313590613\\
261	0.0135666295996786\\
262	0.0134632029999466\\
263	0.0133609352206249\\
264	0.0132598113575866\\
265	0.0131598176591883\\
266	0.0130609412092728\\
267	0.0129631696002414\\
268	0.0128664906186339\\
269	0.0127708919618819\\
270	0.0126763610007649\\
271	0.0125828845970551\\
272	0.0124904489816491\\
273	0.01239903969417\\
274	0.01230864158185\\
275	0.0122192388527352\\
276	0.0121308151762506\\
277	0.0120433538225057\\
278	0.0119568378309094\\
279	0.0118712501977742\\
280	0.0117865740725688\\
281	0.0117027929524262\\
282	0.0116198908652771\\
283	0.011537852532609\\
284	0.011456663504316\\
285	0.0113763102597681\\
286	0.0112967802706679\\
287	0.0112180620236809\\
288	0.011140145002516\\
289	0.0110630196311393\\
290	0.0109866771816486\\
291	0.010911109651763\\
292	0.0108363096183237\\
293	0.0107622700739829\\
294	0.0106889842550904\\
295	0.0106164454688559\\
296	0.0105446469280725\\
297	0.0104735816011074\\
298	0.0104032420842114\\
299	0.0103336205021508\\
300	0.0102647084418483\\
301	0.0101964969220902\\
302	0.0101289764005146\\
303	0.0100621368173748\\
304	0.00999596767339198\\
305	0.00993045813737115\\
306	0.0098655971776187\\
307	0.00980137370971965\\
308	0.00973777675252341\\
309	0.00967479558358311\\
310	0.0096124198853761\\
311	0.00955063987420285\\
312	0.00948944640468952\\
313	0.00942883104414061\\
314	0.00936878611286704\\
315	0.0093093046884842\\
316	0.00925038057426093\\
317	0.00919200823354605\\
318	0.00913418269436808\\
319	0.00907689942963318\\
320	0.00902015421993452\\
321	0.00896394300666595\\
322	0.0089082617435431\\
323	0.00885310625474894\\
324	0.00879847210722257\\
325	0.00874435450369776\\
326	0.00869074820188595\\
327	0.00863764746348415\\
328	0.00858504603497448\\
329	0.00853293716037502\\
330	0.00848131362422658\\
331	0.00843016782150385\\
332	0.00837949184988901\\
333	0.00832927761846487\\
334	0.008279516966642\\
335	0.00823020178649575\\
336	0.00818132414210594\\
337	0.008132876379884\\
338	0.00808485122470172\\
339	0.00803724185782916\\
340	0.00799004197376255\\
341	0.00794324581449452\\
342	0.00789684818095623\\
343	0.00785084442256649\\
344	0.00780523040687763\\
345	0.00776000247215652\\
346	0.00771515736618537\\
347	0.00767069217515787\\
348	0.00762660424640449\\
349	0.00758289110877737\\
350	0.00753955039413258\\
351	0.00749657976306799\\
352	0.00745397683735897\\
353	0.00741173914122278\\
354	0.00736986405272057\\
355	0.00732834876620828\\
356	0.00728719026613295\\
357	0.00724638531199875\\
358	0.00720593043402193\\
359	0.00716582193849986\\
360	0.0071260559218848\\
361	0.00708662829224646\\
362	0.00704753479674278\\
363	0.00700877105367805\\
364	0.00697033258778879\\
365	0.00693221486729713\\
366	0.00689441334146822\\
367	0.00685692347753253\\
368	0.00681974079574507\\
369	0.00678286090176183\\
370	0.00674627951543588\\
371	0.00670999249537486\\
372	0.00667399585887032\\
373	0.00663828579671637\\
374	0.00660285868292\\
375	0.00656771107927249\\
376	0.00653283973491341\\
377	0.00649824158134642\\
378	0.00646391372326078\\
379	0.00642985342580255\\
380	0.00639605809889742\\
381	0.00636252527943968\\
382	0.00632925261190698\\
383	0.00629623782833236\\
384	0.00626347872818379\\
385	0.00623097315888215\\
386	0.00619871899752256\\
387	0.00616671413428906\\
388	0.00613495645799119\\
389	0.00610344384398911\\
390	0.00607217414477489\\
391	0.00604114518323167\\
392	0.00601035474861683\\
393	0.00597980059518526\\
394	0.00594948044324119\\
395	0.00591939198238621\\
396	0.00588953287664792\\
397	0.00585990077114694\\
398	0.0058304932998627\\
399	0.00580130809416382\\
400	0.00577234279160743\\
401	0.00574359504472752\\
402	0.00571506252932916\\
403	0.00568674295207676\\
404	0.0056586340570612\\
405	0.00563073363116426\\
406	0.00560303950807235\\
407	0.0055755495708928\\
408	0.00554826175337556\\
409	0.00552117403975236\\
410	0.00549428446337162\\
411	0.00546759110424433\\
412	0.00544109208566899\\
413	0.00541478557021681\\
414	0.00538866975520653\\
415	0.00536274286795768\\
416	0.00533700316100401\\
417	0.00531144890747995\\
418	0.00528607839685124\\
419	0.00526088993110627\\
420	0.00523588182159703\\
421	0.00521105238656502\\
422	0.00518639994939644\\
423	0.00516192283767915\\
424	0.00513761938300195\\
425	0.00511348792149642\\
426	0.00508952679502495\\
427	0.00506573435290524\\
428	0.00504210895413337\\
429	0.00501864896990645\\
430	0.0049953527863171\\
431	0.00497221880711964\\
432	0.00494924545636381\\
433	0.00492643118084082\\
434	0.00490377445213681\\
435	0.00488127376825141\\
436	0.00485892765469721\\
437	0.00483673466499319\\
438	0.0048146933805638\\
439	0.0047928024100691\\
440	0.00477106038815732\\
441	0.00474946597371732\\
442	0.00472801784776451\\
443	0.00470671471100488\\
444	0.0046855552812488\\
445	0.00466453829074766\\
446	0.00464366248366507\\
447	0.00462292661367639\\
448	0.00460232944194908\\
449	0.00458186973545835\\
450	0.00456154626581425\\
451	0.00454135780856885\\
452	0.00452130314305561\\
453	0.00450138105275738\\
454	0.00448159032614535\\
455	0.00446192975795019\\
456	0.00444239815078731\\
457	0.00442299431702868\\
458	0.00440371708088096\\
459	0.00438456528045789\\
460	0.00436553776985966\\
461	0.0043466334210785\\
462	0.00432785112567721\\
463	0.00430918979614064\\
464	0.00429064836683327\\
465	0.00427222579461681\\
466	0.00425392105890905\\
467	0.0042357331614106\\
468	0.00421766112531455\\
469	0.00419970399420763\\
470	0.0041818608305936\\
471	0.00416413071414956\\
472	0.00414651273983436\\
473	0.00412900601583346\\
474	0.00411160966155212\\
475	0.00409432280563529\\
476	0.0040771445841204\\
477	0.00406007413885882\\
478	0.00404311061615215\\
479	0.00402625316575603\\
480	0.00400950094017701\\
481	0.0039928530943573\\
482	0.00397630878567255\\
483	0.00395986717425562\\
484	0.00394352742360767\\
485	0.00392728870141292\\
486	0.00391115018051673\\
487	0.00389511104003354\\
488	0.00387917046641234\\
489	0.00386332765453833\\
490	0.00384758180865583\\
491	0.00383193214316967\\
492	0.00381637788323805\\
493	0.00380091826509349\\
494	0.00378555253613055\\
495	0.00377027995471328\\
496	0.00375509978971444\\
497	0.00374001131985986\\
498	0.00372501383279811\\
499	0.00371010662408911\\
500	0\\
};
\addlegendentry{Alg.~3 $p=0.8$}

\addplot [color=red,thick]
  table[row sep=crcr]{%
1	13\\
2	6.16010171835132\\
3	3.24972494313784\\
4	1.99441893122169\\
5	1.34696399215365\\
6	0.968782201305098\\
7	0.730556119486546\\
8	0.570405746633071\\
9	0.457534496310275\\
10	0.375124190012651\\
11	0.313196919846663\\
12	0.265476588655689\\
13	0.227819201506443\\
14	0.197640556743896\\
15	0.173113276537342\\
16	0.152881277946509\\
17	0.13600525221631\\
18	0.121772593339821\\
19	0.109678927570046\\
20	0.0993183923981785\\
21	0.0903753260223559\\
22	0.0825913767081997\\
23	0.0757696937259901\\
24	0.0697670197164749\\
25	0.0644508236832855\\
26	0.0597284187207982\\
27	0.0555074711267163\\
28	0.0517226019935298\\
29	0.0483102062233199\\
30	0.0452229410442533\\
31	0.0424182640162085\\
32	0.0398660849422141\\
33	0.0375353596538245\\
34	0.0353943009201791\\
35	0.0334308895442218\\
36	0.0316261541618598\\
37	0.0299635454000426\\
38	0.0284282414873544\\
39	0.0270081133141235\\
40	0.0256912540898156\\
41	0.0244684833369359\\
42	0.0233307159224445\\
43	0.0222706399852737\\
44	0.0212810304763027\\
45	0.0203559324692161\\
46	0.0194899021573588\\
47	0.0186779830050762\\
48	0.0179157807224146\\
49	0.0171993452208709\\
50	0.0165250639357185\\
51	0.0158896733920027\\
52	0.0152902519774362\\
53	0.0147241309646182\\
54	0.014188882782711\\
55	0.0136823002120411\\
56	0.0132023695026623\\
57	0.0127472514035742\\
58	0.0123152633261478\\
59	0.0119048741574503\\
60	0.011514649709905\\
61	0.0111432955062841\\
62	0.0107896169027968\\
63	0.010452510164633\\
64	0.0101309671053864\\
65	0.00982402643322818\\
66	0.00953083348839406\\
67	0.00925056136943038\\
68	0.00898247169102118\\
69	0.00872587072905026\\
70	0.00848010704452988\\
71	0.00824458118944547\\
72	0.00801873254773616\\
73	0.00780207266778804\\
74	0.00759406222163811\\
75	0.00739424464538142\\
76	0.00720222193042541\\
77	0.00701757521998404\\
78	0.00683994379707023\\
79	0.00666898014561929\\
80	0.00650434311506454\\
81	0.00634572599183277\\
82	0.0061928726954025\\
83	0.00604546798513841\\
84	0.00590324478478563\\
85	0.0057659898243563\\
86	0.00563346265478484\\
87	0.00550544699606795\\
88	0.00538174439537182\\
89	0.00526216308482907\\
90	0.00514652195666157\\
91	0.00503464986529137\\
92	0.00492638497085506\\
93	0.00482160120693476\\
94	0.00472010383305302\\
95	0.00462177833212966\\
96	0.00452648832757199\\
97	0.00443411568446336\\
98	0.0043445496201881\\
99	0.00425767010824459\\
100	0.00417336640013111\\
101	0.00409154297613797\\
102	0.00401211020824715\\
103	0.00393496808046161\\
104	0.00386002468206129\\
105	0.00378720230297704\\
106	0.00371642140445563\\
107	0.00364761296935956\\
108	0.00358069648089528\\
109	0.00351559865586337\\
110	0.00345226618431196\\
111	0.00339062806263464\\
112	0.00333062000868752\\
113	0.0032721904047924\\
114	0.00321529128742106\\
115	0.00315986213357687\\
116	0.00310584847217771\\
117	0.00305321414432979\\
118	0.0030019057435391\\
119	0.00295187526522368\\
120	0.00290308550629172\\
121	0.00285549733564957\\
122	0.00280907043703764\\
123	0.00276376558675697\\
124	0.00271954855195352\\
125	0.00267638473143635\\
126	0.00263424082021736\\
127	0.0025953106765065\\
128	0.00259420551622268\\
129	0.00258840005546041\\
130	0.00258165311367722\\
131	0.00257488062109444\\
132	0.00255101895929086\\
133	0.00252476084035161\\
134	0.0024944701002694\\
135	0.00246170777717328\\
136	0.00242768167133161\\
137	0.0023932570410123\\
138	0.00235919263749861\\
139	0.00232552356844435\\
140	0.00229262868103067\\
141	0.002260615743305\\
142	0.00222950590417491\\
143	0.00219926568990603\\
144	0.00216982786089829\\
145	0.0021411078545756\\
146	0.00211301661186333\\
147	0.00208547030604784\\
148	0.00205839731292508\\
149	0.00203174265771413\\
150	0.00200547012936771\\
151	0.00197956225548485\\
152	0.00195371861763317\\
153	0.0019283779281403\\
154	0.00190353105563323\\
155	0.00187916138445435\\
156	0.00185526047691796\\
157	0.00183181502345875\\
158	0.00180881216275582\\
159	0.00178625371450565\\
160	0.00176411134136789\\
161	0.00174236776742595\\
162	0.00172102487758125\\
163	0.00170007224951287\\
164	0.00167949922051559\\
165	0.00165929540472883\\
166	0.00163945104758592\\
167	0.00161995718934403\\
168	0.00160080565401956\\
169	0.00158198891221381\\
170	0.0015634998813509\\
171	0.00154533172510241\\
172	0.00152747769931429\\
173	0.00150993107022412\\
174	0.00149268510831693\\
175	0.00147573314245799\\
176	0.00145906864699847\\
177	0.00144268533054792\\
178	0.00142657719827355\\
179	0.00141073856856363\\
180	0.00139516403699382\\
181	0.00137988349956\\
182	0.00136482257031589\\
183	0.0013500090673305\\
184	0.00133543884239713\\
185	0.00132110637211577\\
186	0.00130700593011762\\
187	0.00129313164126323\\
188	0.00127947758458534\\
189	0.00126603792749888\\
190	0.0012528070682862\\
191	0.00123977976315828\\
192	0.00122695121798422\\
193	0.00121431713180305\\
194	0.00120187368813147\\
195	0.00118961749851111\\
196	0.0011775455098457\\
197	0.00116565489109383\\
198	0.00115394291616615\\
199	0.00114240685836582\\
200	0.00113104390796762\\
201	0.00111985111997394\\
202	0.00110882539392561\\
203	0.00109796348333117\\
204	0.00108726202877279\\
205	0.00107671760674663\\
206	0.00106632678563627\\
207	0.0010560861806775\\
208	0.00104599250147307\\
209	0.00103604192126283\\
210	0.00102626417994946\\
211	0.00101659429382236\\
212	0.00100708649706405\\
213	0.000997682400272537\\
214	0.00098840857177409\\
215	0.00097926266668498\\
216	0.000970242418404585\\
217	0.000961356501763495\\
218	0.000952601069914925\\
219	0.000943944454478783\\
220	0.00093540496137344\\
221	0.000926995732798773\\
222	0.00091870078071226\\
223	0.000910500274199776\\
224	0.00090240895733687\\
225	0.000894441414854286\\
226	0.000886576606010117\\
227	0.000878801390404409\\
228	0.000871128023941063\\
229	0.000863554760922758\\
230	0.000856079891372472\\
231	0.000848701741470052\\
232	0.000841418673995786\\
233	0.000834247519719514\\
234	0.000827159470825498\\
235	0.000820151975576888\\
236	0.000813233370546416\\
237	0.00080640217950103\\
238	0.000799656947420319\\
239	0.000792996236228034\\
240	0.000786418621428656\\
241	0.000779922690339302\\
242	0.00077350704243441\\
243	0.000767170291858167\\
244	0.000760911072076362\\
245	0.000754728041974198\\
246	0.000748619892699739\\
247	0.000742585354287367\\
248	0.000736623201080047\\
249	0.000730732255188854\\
250	0.00072491138739892\\
251	0.00071915951540736\\
252	0.000713475599443097\\
253	0.000707858635888153\\
254	0.000702307649573825\\
255	0.000696821685645221\\
256	0.000691399801923606\\
257	0.000686041062516358\\
258	0.000680744533249324\\
259	0.000675509279275566\\
260	0.000670335918270155\\
261	0.000665228280794551\\
262	0.000660171151499703\\
263	0.000655171574237307\\
264	0.00065022864052252\\
265	0.000645342126619056\\
266	0.000640516826552431\\
267	0.00063573849427899\\
268	0.000631013404703507\\
269	0.0006263407497216\\
270	0.000621719742900546\\
271	0.000617151274624861\\
272	0.00061263693974957\\
273	0.000608166283593524\\
274	0.000603744316016161\\
275	0.000599370332899141\\
276	0.000595043641953744\\
277	0.000590763562266727\\
278	0.000586533005157513\\
279	0.00138066279774236\\
280	0.000576281034603596\\
281	0.000572211337056566\\
282	0.000575138875185476\\
283	0.000585196916512674\\
284	0.000600993832885685\\
285	0.000625794740834712\\
286	0.000635565484528103\\
287	0.000645131387425391\\
288	0.000650990382720359\\
289	0.000653665557203462\\
290	0.000653872713604779\\
291	0.000652290668516221\\
292	0.000649486673049288\\
293	0.000646263800938643\\
294	0.000642239319847256\\
295	0.000637944797902891\\
296	0.000633670380244276\\
297	0.000629463708103087\\
298	0.000625388133028801\\
299	0.000621473732461214\\
300	0.000617726442348809\\
301	0.000614135460430737\\
302	0.000610679198660036\\
303	0.000607330010638507\\
304	0.00060405787787432\\
305	0.000600833208856025\\
306	0.000597628883707943\\
307	0.00059442166120547\\
308	0.000591193050903362\\
309	0.000587929740591388\\
310	0.000584199644201935\\
311	0.000580489212179999\\
312	0.00057681605544524\\
313	0.000573178118278068\\
314	0.000569573640214677\\
315	0.000566001266815191\\
316	0.000562460077143834\\
317	0.000558949541756934\\
318	0.000555469433401387\\
319	0.000552019715619412\\
320	0.000548600432678785\\
321	0.000545211618952497\\
322	0.000541853238332481\\
323	0.000538545103500321\\
324	0.000535265948190158\\
325	0.000531997731314154\\
326	0.00052875938017688\\
327	0.000525550708444331\\
328	0.000522371499453576\\
329	0.00051922151143128\\
330	0.000516100471345093\\
331	0.000513008062344583\\
332	0.00050998407200923\\
333	0.000506950227605905\\
334	0.000503941076127147\\
335	0.00050095870085555\\
336	0.000498002542722087\\
337	0.000495072059967142\\
338	0.000492166760382051\\
339	0.000489286226767374\\
340	0.000486430130561322\\
341	0.000483598230903839\\
342	0.000480790359261012\\
343	0.000478006392621563\\
344	0.000475246220315674\\
345	0.000472509710908717\\
346	0.000469796685125451\\
347	0.000467106899781709\\
348	0.000464440045241175\\
349	0.000461795756537049\\
350	0.000459173635631763\\
351	0.000456573280495136\\
352	0.000453994315451019\\
353	0.000451436417426322\\
354	0.000448899333412009\\
355	0.000446382886404265\\
356	0.000443886969007795\\
357	0.000441411526092159\\
358	0.00043895652976858\\
359	0.000436521950977389\\
360	0.000434107732515968\\
361	0.000431713767799768\\
362	0.000429339888633085\\
363	0.000426985863562629\\
364	0.000424651406689076\\
365	0.000422336195126681\\
366	0.000420039891966845\\
367	0.000417762220291364\\
368	0.000415503109224544\\
369	0.000413262495184412\\
370	0.000411040318401855\\
371	0.000408836522310649\\
372	0.000406651047460495\\
373	0.000404483820470279\\
374	0.000402334739704815\\
375	0.000400203660001833\\
376	0.000398090379253709\\
377	0.000395994629410409\\
378	0.00039391607400277\\
379	0.000391854313356095\\
380	0.000389808897737313\\
381	0.00038777934742041\\
382	0.00038576517793782\\
383	0.000383765927948587\\
384	0.000381781186791628\\
385	0.000379810618770932\\
386	0.000377853981531118\\
387	0.000375911136379511\\
388	0.000374002869822733\\
389	0.000372114099458476\\
390	0.000370212528984747\\
391	0.000368325135112008\\
392	0.000366452131634089\\
393	0.000364593722440465\\
394	0.000362750070881331\\
395	0.000360950443337511\\
396	0.000359155125131141\\
397	0.000357355818978583\\
398	0.000355571100687083\\
399	0.00035380068410476\\
400	0.000352044208003881\\
401	0.000350301259307406\\
402	0.000348571399951203\\
403	0.000346854194702602\\
404	0.00034514923713639\\
405	0.000343456171397818\\
406	0.00034177470792713\\
407	0.00034010463217686\\
408	0.000338445805984917\\
409	0.000336798162254466\\
410	0.000335190016997666\\
411	0.000333585624285947\\
412	0.00033197143311172\\
413	0.00033036858688007\\
414	0.00032877716189561\\
415	0.000327197216330094\\
416	0.000325628779825615\\
417	0.000324071846973431\\
418	0.000322526374790751\\
419	0.000320992283750841\\
420	0.000319469461697301\\
421	0.000317957769875671\\
422	0.000316457050047562\\
423	0.000314967132007481\\
424	0.000313487840639545\\
425	0.00031201900211543\\
426	0.000310560448787647\\
427	0.000309112022754354\\
428	0.000307673578060601\\
429	0.000306244981742469\\
430	0.000304826113906642\\
431	0.000303416867123089\\
432	0.000302017145361974\\
433	0.000300626862712178\\
434	0.000299245941950473\\
435	0.000297874313147095\\
436	0.000296511912288494\\
437	0.000295158679983379\\
438	0.000293814560209951\\
439	0.000292479499124323\\
440	0.000291153443939485\\
441	0.000289836341871456\\
442	0.000288528139201376\\
443	0.000287228780446486\\
444	0.000285938207697539\\
445	0.000284656360152957\\
446	0.000283383173821997\\
447	0.000282118581386298\\
448	0.000280862512256078\\
449	0.000279614892721602\\
450	0.000278375646191473\\
451	0.000277144693496829\\
452	0.000275921953226191\\
453	0.000274707342093103\\
454	0.000273500775324551\\
455	0.000272302167112836\\
456	0.000271111431128068\\
457	0.000269928481116066\\
458	0.000268753231584061\\
459	0.000267585598583218\\
460	0.000266425500538403\\
461	0.000265272859088603\\
462	0.000264127599902256\\
463	0.000262989653359767\\
464	0.000261858955058869\\
465	0.000260735446054372\\
466	0.000259643015852733\\
467	0.000258557102536236\\
468	0.000257454663943134\\
469	0.000256359227111751\\
470	0.000255270753783093\\
471	0.000254189206653328\\
472	0.000253114547939359\\
473	0.000252046738021781\\
474	0.000250985734249518\\
475	0.000249956785536645\\
476	0.000248931047795636\\
477	0.000247889959352633\\
478	0.000246855463217184\\
479	0.000245827495281993\\
480	0.00024480598813871\\
481	0.000243790872201111\\
482	0.000242782076927767\\
483	0.000241779532144951\\
484	0.000240783169302664\\
485	0.000239792922628608\\
486	0.000238808730041895\\
487	0.000237830533815798\\
488	0.000236858280897527\\
489	0.000235891922961702\\
490	0.000234931416101605\\
491	0.000233976720254415\\
492	0.000233027798431748\\
493	0.000232084615801868\\
494	0.000231147138694557\\
495	0.000230215333616279\\
496	0.000229289166360842\\
497	0.000228368601268961\\
498	0.000227453600665471\\
499	0.000226544124525937\\
500	0\\
};
\addlegendentry{Alg.~4 $p=0.8$}

\addplot [color=black, thick, dashed]
  table[row sep=crcr]{%
1	100\\
2	25\\
3	11.1111111111111\\
4	6.25\\
5	4\\
6	2.77777777777778\\
7	2.04081632653061\\
8	1.5625\\
9	1.23456790123457\\
10	1\\
11	0.826446280991736\\
12	0.694444444444444\\
13	0.591715976331361\\
14	0.510204081632653\\
15	0.444444444444444\\
16	0.390625\\
17	0.346020761245675\\
18	0.308641975308642\\
19	0.277008310249307\\
20	0.25\\
21	0.226757369614512\\
22	0.206611570247934\\
23	0.189035916824197\\
24	0.173611111111111\\
25	0.16\\
26	0.14792899408284\\
27	0.137174211248285\\
28	0.127551020408163\\
29	0.118906064209275\\
30	0.111111111111111\\
31	0.104058272632674\\
32	0.09765625\\
33	0.0918273645546373\\
34	0.0865051903114187\\
35	0.0816326530612245\\
36	0.0771604938271605\\
37	0.0730460189919649\\
38	0.0692520775623269\\
39	0.0657462195923734\\
40	0.0625\\
41	0.0594883997620464\\
42	0.0566893424036281\\
43	0.0540832882639264\\
44	0.0516528925619835\\
45	0.0493827160493827\\
46	0.0472589792060491\\
47	0.0452693526482571\\
48	0.0434027777777778\\
49	0.041649312786339\\
50	0.04\\
51	0.0384467512495194\\
52	0.0369822485207101\\
53	0.0355998576005696\\
54	0.0342935528120713\\
55	0.0330578512396694\\
56	0.0318877551020408\\
57	0.0307787011388119\\
58	0.0297265160523187\\
59	0.0287273771904625\\
60	0.0277777777777778\\
61	0.0268744961031981\\
62	0.0260145681581686\\
63	0.0251952632905014\\
64	0.0244140625\\
65	0.0236686390532544\\
66	0.0229568411386593\\
67	0.0222766763198931\\
68	0.0216262975778547\\
69	0.0210039907582441\\
70	0.0204081632653061\\
71	0.0198373338623289\\
72	0.0192901234567901\\
73	0.0187652467629949\\
74	0.0182615047479912\\
75	0.0177777777777778\\
76	0.0173130193905817\\
77	0.0168662506324844\\
78	0.0164365548980934\\
79	0.0160230732254446\\
80	0.015625\\
81	0.0152415790275873\\
82	0.0148720999405116\\
83	0.0145158949049209\\
84	0.014172335600907\\
85	0.013840830449827\\
86	0.0135208220659816\\
87	0.0132117849121416\\
88	0.0129132231404959\\
89	0.0126246686024492\\
90	0.0123456790123457\\
91	0.0120758362516604\\
92	0.0118147448015123\\
93	0.0115620302925194\\
94	0.0113173381620643\\
95	0.0110803324099723\\
96	0.0108506944444444\\
97	0.0106281220108407\\
98	0.0104123281965848\\
99	0.0102030405060708\\
100	0.01\\
101	0.00980296049406921\\
102	0.00961168781237985\\
103	0.00942595909133754\\
104	0.00924556213017752\\
105	0.0090702947845805\\
106	0.0088999644001424\\
107	0.00873438728273212\\
108	0.00857338820301783\\
109	0.0084167999326656\\
110	0.00826446280991736\\
111	0.00811622433244055\\
112	0.0079719387755102\\
113	0.00783146683373796\\
114	0.00769467528470299\\
115	0.00756143667296786\\
116	0.00743162901307967\\
117	0.00730513551026372\\
118	0.00718184429761563\\
119	0.00706164818868724\\
120	0.00694444444444444\\
121	0.00683013455365071\\
122	0.00671862402579952\\
123	0.00660982219578293\\
124	0.00650364203954214\\
125	0.0064\\
126	0.00629881582262535\\
127	0.0062000124000248\\
128	0.006103515625\\
129	0.00600925425154738\\
130	0.00591715976331361\\
131	0.00582716624905309\\
132	0.00573921028466483\\
133	0.00565323082141444\\
134	0.00556916907997327\\
135	0.00548696844993141\\
136	0.00540657439446367\\
137	0.00532793435984869\\
138	0.00525099768956102\\
139	0.00517571554267377\\
140	0.00510204081632653\\
141	0.00502992807202857\\
142	0.00495933346558223\\
143	0.00489021468042447\\
144	0.00482253086419753\\
145	0.00475624256837099\\
146	0.00469131169074873\\
147	0.00462770142070434\\
148	0.00456537618699781\\
149	0.00450430160803567\\
150	0.00444444444444444\\
151	0.00438577255383536\\
152	0.00432825484764543\\
153	0.0042718612499466\\
154	0.0042165626581211\\
155	0.00416233090530697\\
156	0.00410913872452334\\
157	0.00405695971439004\\
158	0.00400576830636116\\
159	0.00395553973339662\\
160	0.00390625\\
161	0.00385787585355503\\
162	0.00381039475689681\\
163	0.00376378486205728\\
164	0.0037180249851279\\
165	0.00367309458218549\\
166	0.00362897372623022\\
167	0.00358564308508731\\
168	0.00354308390022676\\
169	0.00350127796645776\\
170	0.00346020761245675\\
171	0.00341985568209022\\
172	0.0033802055164954\\
173	0.00334124093688396\\
174	0.00330294622803541\\
175	0.00326530612244898\\
176	0.00322830578512397\\
177	0.00319193079894028\\
178	0.0031561671506123\\
179	0.00312100121719047\\
180	0.00308641975308642\\
181	0.00305240987759836\\
182	0.00301895906291511\\
183	0.00298605512257756\\
184	0.00295368620037807\\
185	0.0029218407596786\\
186	0.00289050757312984\\
187	0.00285967571277417\\
188	0.00282933454051607\\
189	0.0027994736989446\\
190	0.00277008310249307\\
191	0.0027411529289219\\
192	0.00271267361111111\\
193	0.00268463582914978\\
194	0.00265703050271017\\
195	0.00262984878369494\\
196	0.00260308204914619\\
197	0.00257672189440594\\
198	0.0025507601265177\\
199	0.00252518875785965\\
200	0.0025\\
201	0.0024751862577659\\
202	0.0024507401235173\\
203	0.00242665437161785\\
204	0.00240292195309496\\
205	0.00237953599048186\\
206	0.00235648977283439\\
207	0.00233377675091601\\
208	0.00231139053254438\\
209	0.00228932487809345\\
210	0.00226757369614512\\
211	0.00224613103928483\\
212	0.0022249911000356\\
213	0.00220414820692543\\
214	0.00218359682068303\\
215	0.00216333153055706\\
216	0.00214334705075446\\
217	0.00212363821699335\\
218	0.0021041999831664\\
219	0.00208502741811055\\
220	0.00206611570247934\\
221	0.00204746012571405\\
222	0.00202905608311014\\
223	0.00201089907297553\\
224	0.00199298469387755\\
225	0.00197530864197531\\
226	0.00195786670843449\\
227	0.00194065477692173\\
228	0.00192366882117575\\
229	0.0019069049026525\\
230	0.00189035916824197\\
231	0.00187402784805382\\
232	0.00185790725326992\\
233	0.00184199377406104\\
234	0.00182628387756593\\
235	0.00181077410593029\\
236	0.00179546107440391\\
237	0.00178034146949385\\
238	0.00176541204717181\\
239	0.00175066963113391\\
240	0.00173611111111111\\
241	0.00172173344122863\\
242	0.00170753363841268\\
243	0.00169350878084303\\
244	0.00167965600644988\\
245	0.00166597251145356\\
246	0.00165245554894573\\
247	0.0016391024275107\\
248	0.00162591050988554\\
249	0.00161287721165788\\
250	0.0016\\
251	0.00158727639243822\\
252	0.00157470395565634\\
253	0.0015622803043322\\
254	0.0015500031000062\\
255	0.00153787004998078\\
256	0.00152587890625\\
257	0.00151402746445821\\
258	0.00150231356288685\\
259	0.00149073508146867\\
260	0.0014792899408284\\
261	0.00146797610134907\\
262	0.00145679156226327\\
263	0.00144573436076855\\
264	0.00143480257116621\\
265	0.00142399430402278\\
266	0.00141330770535361\\
267	0.00140274095582769\\
268	0.00139229226999332\\
269	0.00138195989552383\\
270	0.00137174211248285\\
271	0.00136163723260849\\
272	0.00135164359861592\\
273	0.00134175958351783\\
274	0.00133198358996217\\
275	0.00132231404958678\\
276	0.00131274942239025\\
277	0.00130328819611881\\
278	0.00129392888566844\\
279	0.00128467003250215\\
280	0.00127551020408163\\
281	0.00126644799331315\\
282	0.00125748201800714\\
283	0.00124861092035111\\
284	0.00123983336639556\\
285	0.00123114804555248\\
286	0.00122255367010612\\
287	0.00121404897473564\\
288	0.00120563271604938\\
289	0.00119730367213036\\
290	0.00118906064209275\\
291	0.00118090244564896\\
292	0.00117282792268718\\
293	0.00116483593285886\\
294	0.00115692535517608\\
295	0.0011490950876185\\
296	0.00114134404674945\\
297	0.0011336711673412\\
298	0.00112607540200892\\
299	0.00111855572085323\\
300	0.00111111111111111\\
301	0.00110374057681483\\
302	0.00109644313845884\\
303	0.00108921783267436\\
304	0.00108206371191136\\
305	0.00107497984412792\\
306	0.00106796531248665\\
307	0.00106101921505798\\
308	0.00105414066453027\\
309	0.00104732878792639\\
310	0.00104058272632674\\
311	0.00103390163459848\\
312	0.00102728468113083\\
313	0.00102073104757627\\
314	0.00101423992859751\\
315	0.00100781053162006\\
316	0.00100144207659029\\
317	0.000995133795738837\\
318	0.000988884933349156\\
319	0.000982694745531196\\
320	0.0009765625\\
321	0.000970487475859124\\
322	0.000964468963388758\\
323	0.000958506263838434\\
324	0.000952598689224204\\
325	0.000946745562130178\\
326	0.000940946215514321\\
327	0.0009351999925184\\
328	0.000929506246281975\\
329	0.00092386433976035\\
330	0.000918273645546373\\
331	0.000912733545696005\\
332	0.000907243431557555\\
333	0.000901802703604505\\
334	0.000896410771271828\\
335	0.000891067052795723\\
336	0.000885770975056689\\
337	0.000880521973425847\\
338	0.000875319491614439\\
339	0.00087016298152644\\
340	0.000865051903114187\\
341	0.000859985724236978\\
342	0.000854963920522554\\
343	0.000849985975231409\\
344	0.000845051379123851\\
345	0.000840159630329763\\
346	0.00083531023422099\\
347	0.000830502703286299\\
348	0.000825736557008852\\
349	0.000821011321746127\\
350	0.000816326530612245\\
351	0.000811681723362635\\
352	0.000807076446280992\\
353	0.00080251025206847\\
354	0.00079798269973507\\
355	0.000793493354493156\\
356	0.000789041787653074\\
357	0.000784627576520804\\
358	0.000780250304297619\\
359	0.000775909559981688\\
360	0.000771604938271605\\
361	0.000767336039471766\\
362	0.000763102469399591\\
363	0.000758903839294523\\
364	0.000754739765728777\\
365	0.000750609870519797\\
366	0.000746513780644391\\
367	0.000742451128154489\\
368	0.000738421550094518\\
369	0.000734424688420326\\
370	0.000730460189919649\\
371	0.000726527706134073\\
372	0.00072262689328246\\
373	0.000718757412185813\\
374	0.000714918928193543\\
375	0.000711111111111111\\
376	0.000707333635129018\\
377	0.000703586178753105\\
378	0.00069986842473615\\
379	0.000696180060010721\\
380	0.000692520775623269\\
381	0.000688890266669422\\
382	0.000685288232230476\\
383	0.000681714375311032\\
384	0.000678168402777778\\
385	0.000674650025299376\\
386	0.000671158957287444\\
387	0.000667694916838598\\
388	0.000664257625677543\\
389	0.000660846809101182\\
390	0.000657462195923734\\
391	0.000654103518422826\\
392	0.000650770512286547\\
393	0.000647462916561454\\
394	0.000644180473601484\\
395	0.000640922929017786\\
396	0.000637690031629426\\
397	0.00063448153341497\\
398	0.000631297189464912\\
399	0.000628136757934938\\
400	0.000625\\
401	0.000621886679809205\\
402	0.000618796564441474\\
403	0.000615729423861978\\
404	0.000612685030879326\\
405	0.00060966316110349\\
406	0.000606663592904463\\
407	0.000603686107371611\\
408	0.000600730488273741\\
409	0.000597796522019835\\
410	0.000594883997620464\\
411	0.000591992706649854\\
412	0.000589122443208596\\
413	0.00058627300388699\\
414	0.000583444187729002\\
415	0.000580635796196835\\
416	0.000577847633136095\\
417	0.000575079504741531\\
418	0.000572331219523363\\
419	0.000569602588274161\\
420	0.000566893424036281\\
421	0.000564203542069837\\
422	0.000561532759821208\\
423	0.000558880896892063\\
424	0.0005562477750089\\
425	0.00055363321799308\\
426	0.000551037051731358\\
427	0.000548459104146899\\
428	0.000545899205170757\\
429	0.00054335718671383\\
430	0.000540832882639265\\
431	0.00053832612873531\\
432	0.000535836762688615\\
433	0.000533364624057945\\
434	0.000530909554248338\\
435	0.000528471396485665\\
436	0.0005260499957916\\
437	0.000523645198958993\\
438	0.000521256854527637\\
439	0.000518884812760415\\
440	0.000516528925619835\\
441	0.000514189046744926\\
442	0.000511865031428513\\
443	0.000509556736594836\\
444	0.000507264020777534\\
445	0.000504986744097967\\
446	0.000502724768243882\\
447	0.000500477956448408\\
448	0.000498246173469388\\
449	0.00049602928556902\\
450	0.000493827160493827\\
451	0.000491639667454929\\
452	0.000489466677108622\\
453	0.000487308061537262\\
454	0.000485163694230433\\
455	0.000483033450066417\\
456	0.000480917205293937\\
457	0.000478814837514185\\
458	0.000476726225663126\\
459	0.000474651249994067\\
460	0.000472589792060492\\
461	0.000470541734699159\\
462	0.000468506962013456\\
463	0.000466485359356997\\
464	0.000464476813317479\\
465	0.000462481211700775\\
466	0.000460498443515261\\
467	0.000458528398956389\\
468	0.000456570969391482\\
469	0.000454626047344757\\
470	0.000452693526482571\\
471	0.000450773301598893\\
472	0.000448865268600977\\
473	0.00044696932449526\\
474	0.000445085367373462\\
475	0.000443213296398892\\
476	0.000441353011792952\\
477	0.000439504414821847\\
478	0.000437667407783477\\
479	0.000435841893994535\\
480	0.000434027777777778\\
481	0.000432224964449497\\
482	0.000430433360307157\\
483	0.000428652872617226\\
484	0.000426883409603169\\
485	0.000425124880433627\\
486	0.000423377195210757\\
487	0.000421640264958742\\
488	0.00041991400161247\\
489	0.000418198318006365\\
490	0.00041649312786339\\
491	0.000414798345784197\\
492	0.000413113887236433\\
493	0.000411439668544203\\
494	0.000409775606877674\\
495	0.000408121620242832\\
496	0.000406477627471384\\
497	0.000404843548210794\\
498	0.000403219302914469\\
499	0.000401604812832077\\
500	0.0004\\
};
\addlegendentry{$100/k^2$}

\end{axis}
\end{tikzpicture}%

%% file: media/imageDenoisingTraj.tikz
%
%
\definecolor{mycolor1}{rgb}{0.00000,0.44700,0.74100}%
\definecolor{mycolor2}{rgb}{0.85000,0.32500,0.09800}%
\definecolor{mycolor3}{rgb}{0,0.5,0}%
\definecolor{mycolor4}{rgb}{0.49400,0.18400,0.55600}%
\begin{tikzpicture}

\begin{axis}[%
width=0.951\figurewidth,
height=\figureheight,
at={(0\figurewidth,0\figureheight)},
scale only axis,
xmin=0,
xmax=100,
xlabel style={font=\color{white!15!black}},
xlabel={iterations},
xlabel near ticks,
ymode=log,
ymin=0.01,
ymax=100000,
yminorticks=true,
ylabel style={font=\color{white!15!black}},
ylabel={$|f(x_k)-f^*|$},
ylabel near ticks,
axis background/.style={fill=white},
legend style={legend cell align=left, align=left, draw=white!15!black}
]
\addplot [color=black,thick]
  table[row sep=crcr]{%
1	17381.5718998293\\
2	4429.59646080573\\
3	672.243090477502\\
4	45.6680851626179\\
5	12.1530866748448\\
6	7.06958547470898\\
7	4.83823005376949\\
8	3.6309792882649\\
9	2.89855916231102\\
10	2.40838260289512\\
11	2.05603651253789\\
12	1.7905234383335\\
13	1.58072057258212\\
14	1.4111659087894\\
15	1.26932363697948\\
16	1.15288585160385\\
17	1.05316736671612\\
18	0.965875634455812\\
19	0.888612040363695\\
20	0.820001873167792\\
21	0.758768348824247\\
22	0.703656874669777\\
23	0.653963733046736\\
24	0.60907401531603\\
25	0.568376219957288\\
26	0.531164875613118\\
27	0.497124369801865\\
28	0.465978645100761\\
29	0.437353224069061\\
30	0.411099724396233\\
31	0.386913021494046\\
32	0.364641467268561\\
33	0.34415297539008\\
34	0.325268385713626\\
35	0.307875477540862\\
36	0.291839377509642\\
37	0.277035846025153\\
38	0.263335499813892\\
39	0.250639734849453\\
40	0.238877562755059\\
41	0.227949019484291\\
42	0.217791526660267\\
43	0.208324305490385\\
44	0.199484324719227\\
45	0.191220994705465\\
46	0.183482434682546\\
47	0.176220652510016\\
48	0.169382232720254\\
49	0.162940564989436\\
50	0.156874147442867\\
51	0.151146701628138\\
52	0.145723873700855\\
53	0.140589382385914\\
54	0.135725566006624\\
55	0.13110606486357\\
56	0.126715542399094\\
57	0.122537222540872\\
58	0.118561804054508\\
59	0.114789009941553\\
60	0.111201807001829\\
61	0.107795625410615\\
62	0.104570230724904\\
63	0.101505268762512\\
64	0.0985957169893181\\
65	0.0958276300334582\\
66	0.0931924837233253\\
67	0.0906759918478352\\
68	0.0882761758405747\\
69	0.0859854779626747\\
70	0.0837977159818022\\
71	0.0817040828023787\\
72	0.07970497682634\\
73	0.0777899401719838\\
74	0.0759518928911819\\
75	0.0741913692823663\\
76	0.0725046123059516\\
77	0.0708840449277165\\
78	0.0693277311457368\\
79	0.0678332734696543\\
80	0.0663908050578124\\
81	0.0650045840146549\\
82	0.0636693770946109\\
83	0.0623789613731855\\
84	0.0611347324934834\\
85	0.0599385714889236\\
86	0.0587854284359089\\
87	0.0576725923695774\\
88	0.056596262484155\\
89	0.0555564857881679\\
90	0.0545514474250004\\
91	0.0535831441209594\\
92	0.0526473041110843\\
93	0.0517437295703442\\
94	0.0508680995720801\\
95	0.0500185661193073\\
96	0.0491945690219567\\
97	0.0483963922552963\\
98	0.0476224113725556\\
99	0.0468710474831414\\
100	0.0461425891362336\\
};
\addlegendentry{Alg.~3, $p=1$}

\addplot [color=red,thick]
  table[row sep=crcr]{%
1	17381.5718998293\\
2	4429.59646080573\\
3	672.312380445851\\
4	45.8262524288529\\
5	12.2941292150824\\
6	7.18032686594379\\
7	4.92491252084329\\
8	3.70915252121811\\
9	2.97160377557478\\
10	2.48006111441029\\
11	2.12783012647495\\
12	1.86178489368162\\
13	1.65288602096682\\
14	1.48347005723899\\
15	1.34193613068397\\
16	1.22089312384655\\
17	1.11572973993296\\
18	1.0233803115661\\
19	0.941732497383514\\
20	0.86915011735969\\
21	0.804283495461725\\
22	0.746045168135059\\
23	0.693544593255999\\
24	0.646018281171036\\
25	0.602834580170239\\
26	0.563455411979678\\
27	0.527422394502333\\
28	0.494368022811659\\
29	0.46399266282723\\
30	0.436050646911165\\
31	0.410321853090472\\
32	0.386616561535553\\
33	0.364765177062553\\
34	0.34461869265541\\
35	0.326037631854659\\
36	0.308894469053053\\
37	0.293073093797261\\
38	0.278463765036053\\
39	0.264960984310394\\
40	0.25246683788156\\
41	0.240894935724148\\
42	0.230160263584478\\
43	0.220184749420146\\
44	0.210896719374542\\
45	0.202229264371043\\
46	0.194120751985352\\
47	0.186515620736682\\
48	0.179365202538653\\
49	0.172627436718467\\
50	0.166265156593011\\
51	0.160246300851958\\
52	0.154543855325744\\
53	0.149133582369949\\
54	0.143994629399995\\
55	0.13910863425742\\
56	0.134460075667153\\
57	0.130035972157349\\
58	0.125824537038094\\
59	0.12181538509525\\
60	0.11799890820761\\
61	0.114366122434194\\
62	0.110907857887908\\
63	0.107615559232655\\
64	0.104481192921605\\
65	0.101496786524849\\
66	0.0986544722617412\\
67	0.095947051037517\\
68	0.0933672387105084\\
69	0.0909080654845682\\
70	0.0885626736035333\\
71	0.0863244174926995\\
72	0.0841869024947417\\
73	0.0821439872092106\\
74	0.080189792969247\\
75	0.0783186343889056\\
76	0.0765255699039719\\
77	0.0748058496252136\\
78	0.0731549551084407\\
79	0.0715688939230291\\
80	0.070043744283066\\
81	0.0685759472551808\\
82	0.0671623059365587\\
83	0.0657997404693507\\
84	0.0644854489009988\\
85	0.0632169701250966\\
86	0.0619921676408037\\
87	0.060808971632306\\
88	0.0596656746739258\\
89	0.0585605100503312\\
90	0.0574917716007987\\
91	0.0564579504276451\\
92	0.055457508651176\\
93	0.0544889479105935\\
94	0.0535509739343698\\
95	0.0526423746968218\\
96	0.0517619713649058\\
97	0.0509086500154307\\
98	0.0500813316671297\\
99	0.0492789941146098\\
100	0.0485007340553985\\
};
\addlegendentry{Alg.~4, $p=1$}

\addplot [color=mycolor3,thick]
  table[row sep=crcr]{%
1	17381.5718998293\\
2	62.4381403997062\\
3	21.9301029722601\\
4	10.4284349851849\\
5	6.17262256908218\\
6	4.23157493128457\\
7	3.18311717988135\\
8	2.53489901181399\\
9	2.09339999809622\\
10	1.77219785707287\\
11	1.52688091593422\\
12	1.33200129852884\\
13	1.17222409131773\\
14	1.03814049308252\\
15	0.923804095797094\\
16	0.825225303613424\\
17	0.739533013477764\\
18	0.664553438041355\\
19	0.598599749120628\\
20	0.540347130086319\\
21	0.488740334888762\\
22	0.442922467355456\\
23	0.402185605445592\\
24	0.36593960973991\\
25	0.333690089437102\\
26	0.305017587555539\\
27	0.279555879047089\\
28	0.25697237685561\\
29	0.236954713918189\\
30	0.219205217076648\\
31	0.203441832808615\\
32	0.189402256419747\\
33	0.176848122439797\\
34	0.165567443124541\\
35	0.1553749791929\\
36	0.146111166458917\\
37	0.137640429308569\\
38	0.129849421174152\\
39	0.122645309954372\\
40	0.115953947745137\\
41	0.109717721400017\\
42	0.103893012176713\\
43	0.0984473736525552\\
44	0.0933566615467456\\
45	0.0886023731427059\\
46	0.0841693926792411\\
47	0.0800442379804317\\
48	0.0762138082316203\\
49	0.0726645702040346\\
50	0.0693820956447027\\
51	0.0663508662672058\\
52	0.0635542798346131\\
53	0.0609748085437813\\
54	0.0585942723957777\\
55	0.0563941942885043\\
56	0.0543562027393547\\
57	0.0524624462872105\\
58	0.0506959839968601\\
59	0.0490411207712296\\
60	0.0474836642925698\\
61	0.0460110909517839\\
62	0.0446126190330432\\
63	0.0432791967204419\\
64	0.0420034188289674\\
65	0.0407793890763257\\
66	0.0396025446265159\\
67	0.0384694575850249\\
68	0.0373776253554364\\
69	0.0363252593336808\\
70	0.0353110799010674\\
71	0.034334125099239\\
72	0.0333935802903336\\
73	0.0324886358697684\\
74	0.0316183791190241\\
75	0.0307817242614304\\
76	0.0299773818011309\\
77	0.0292038647138285\\
78	0.0284595256430783\\
79	0.0277426165652608\\
80	0.0270513608786654\\
81	0.0263840277355037\\
82	0.0257389995740381\\
83	0.0251148259023836\\
84	0.0245102589988294\\
85	0.0239242698813286\\
86	0.0233560452963105\\
87	0.0228049683458078\\
88	0.0222705866073025\\
89	0.0217605112443678\\
90	0.0212658755284985\\
91	0.0207878016886365\\
92	0.0203270300717713\\
93	0.0198832876768677\\
94	0.0194565877626218\\
95	0.019046696919148\\
96	0.0186531699279884\\
97	0.0182755642341232\\
98	0.0179131707686386\\
99	0.0175651909520351\\
100	0.0172308704889542\\
};
\addlegendentry{acc. proj. grad.}

\addplot [color=mycolor4,thick]
  table[row sep=crcr]{%
1	17381.5718998293\\
2	62.4381403997062\\
3	21.9301029722601\\
4	12.2519449507993\\
5	8.34075008716498\\
6	6.31792994035335\\
7	5.10790744458788\\
8	4.31023428333921\\
9	3.74656739791145\\
10	3.32710601410485\\
11	3.0023150762715\\
12	2.74285469284702\\
13	2.53032364149396\\
14	2.3526281021593\\
15	2.20150519450467\\
16	2.07112283653913\\
17	1.95724963679084\\
18	1.85674223199056\\
19	1.76721729135512\\
20	1.68683512805647\\
21	1.61415306470475\\
22	1.54802370313237\\
23	1.48752286667169\\
24	1.43189760938458\\
25	1.38052807905872\\
26	1.33289912275043\\
27	1.28857885848801\\
28	1.2472023034455\\
29	1.20845872303688\\
30	1.17208175269546\\
31	1.13784160974031\\
32	1.10553889768647\\
33	1.07499963592807\\
34	1.046071241085\\
35	1.01861925385075\\
36	0.992524654586841\\
37	0.967681647420407\\
38	0.943995819840635\\
39	0.921382605295405\\
40	0.899765991856188\\
41	0.879077431931489\\
42	0.859254917193356\\
43	0.840242190013109\\
44	0.821988068277746\\
45	0.80444586484491\\
46	0.787572886366825\\
47	0.771329998978282\\
48	0.75568125055735\\
49	0.74059354104916\\
50	0.726036333784469\\
51	0.71198140189636\\
52	0.698402604895412\\
53	0.6852756912487\\
54	0.672578123454854\\
55	0.66028892264251\\
56	0.648388530164075\\
57	0.636858684027384\\
58	0.625682308318093\\
59	0.61484341402629\\
60	0.604327009910468\\
61	0.594119022217811\\
62	0.58420622223739\\
63	0.574576160797055\\
64	0.565217108929355\\
65	0.556118004029924\\
66	0.547268400915998\\
67	0.538658427265226\\
68	0.530278742977573\\
69	0.522120503057302\\
70	0.514175323659001\\
71	0.506435250982569\\
72	0.498892732737667\\
73	0.491540591929316\\
74	0.484372002743588\\
75	0.477380468336245\\
76	0.470559800348233\\
77	0.463904099990457\\
78	0.457407740556613\\
79	0.451065351237322\\
80	0.444871802121607\\
81	0.438822190283106\\
82	0.432911826858548\\
83	0.42713622503496\\
84	0.421491088870123\\
85	0.415972302877929\\
86	0.410575922316699\\
87	0.405298164124253\\
88	0.400135398448648\\
89	0.395084140728128\\
90	0.390141044277974\\
91	0.385302893345661\\
92	0.380566596599124\\
93	0.375929181015954\\
94	0.37138778614407\\
95	0.366939658706957\\
96	0.362582147528736\\
97	0.358312698756411\\
98	0.354128851358501\\
99	0.350028232880898\\
100	0.34600855544236\\
};
\addlegendentry{proj. grad.}

\addplot [color=black, dashed]
  table[row sep=crcr]{%
1	1000\\
2	250\\
3	111.111111111111\\
4	62.5\\
5	40\\
6	27.7777777777778\\
7	20.4081632653061\\
8	15.625\\
9	12.3456790123457\\
10	10\\
11	8.26446280991735\\
12	6.94444444444444\\
13	5.91715976331361\\
14	5.10204081632653\\
15	4.44444444444444\\
16	3.90625\\
17	3.46020761245675\\
18	3.08641975308642\\
19	2.77008310249307\\
20	2.5\\
21	2.26757369614512\\
22	2.06611570247934\\
23	1.89035916824197\\
24	1.73611111111111\\
25	1.6\\
26	1.4792899408284\\
27	1.37174211248285\\
28	1.27551020408163\\
29	1.18906064209275\\
30	1.11111111111111\\
31	1.04058272632674\\
32	0.9765625\\
33	0.918273645546373\\
34	0.865051903114187\\
35	0.816326530612245\\
36	0.771604938271605\\
37	0.730460189919649\\
38	0.692520775623269\\
39	0.657462195923734\\
40	0.625\\
41	0.594883997620464\\
42	0.566893424036281\\
43	0.540832882639264\\
44	0.516528925619835\\
45	0.493827160493827\\
46	0.472589792060492\\
47	0.452693526482571\\
48	0.434027777777778\\
49	0.41649312786339\\
50	0.4\\
51	0.384467512495194\\
52	0.369822485207101\\
53	0.355998576005696\\
54	0.342935528120713\\
55	0.330578512396694\\
56	0.318877551020408\\
57	0.307787011388119\\
58	0.297265160523187\\
59	0.287273771904625\\
60	0.277777777777778\\
61	0.268744961031981\\
62	0.260145681581686\\
63	0.251952632905014\\
64	0.244140625\\
65	0.236686390532544\\
66	0.229568411386593\\
67	0.222766763198931\\
68	0.216262975778547\\
69	0.210039907582441\\
70	0.204081632653061\\
71	0.198373338623289\\
72	0.192901234567901\\
73	0.187652467629949\\
74	0.182615047479912\\
75	0.177777777777778\\
76	0.173130193905817\\
77	0.168662506324844\\
78	0.164365548980934\\
79	0.160230732254446\\
80	0.15625\\
81	0.152415790275873\\
82	0.148720999405116\\
83	0.145158949049209\\
84	0.14172335600907\\
85	0.13840830449827\\
86	0.135208220659816\\
87	0.132117849121416\\
88	0.129132231404959\\
89	0.126246686024492\\
90	0.123456790123457\\
91	0.120758362516604\\
92	0.118147448015123\\
93	0.115620302925194\\
94	0.113173381620643\\
95	0.110803324099723\\
96	0.108506944444444\\
97	0.106281220108407\\
98	0.104123281965848\\
99	0.102030405060708\\
100	0.1\\
};
\addlegendentry{$10^3/k^2$}

\end{axis}
\end{tikzpicture}%

%% file: media/imageDenoisingConst.tikz
%
%
\definecolor{mycolor1}{rgb}{0.00000,0.44700,0.74100}%
\definecolor{mycolor2}{rgb}{0.85000,0.32500,0.09800}%
\begin{tikzpicture}

\begin{axis}[%
width=0.951\figurewidth,
height=\figureheight,
at={(0\figurewidth,0\figureheight)},
scale only axis,
xmin=0,
xmax=100,
xlabel style={font=\color{white!15!black}},
xlabel={iterations},
xlabel near ticks,
ymode=log,
ymin=0.1,
ymax=10000,
yminorticks=true,
ylabel style={font=\color{white!15!black}},
ylabel={$|\nu-\sum_{i=1}^n (\bar{x}_{ki})_\Delta^p|$},
ylabel near ticks,
axis background/.style={fill=white},
legend style={legend cell align=left, align=left, draw=white!15!black}
]
\addplot [color=black,thick]
  table[row sep=crcr]{%
1	6000\\
2	3742.41661699306\\
3	2283.97441468772\\
4	1490.66887854023\\
5	1047.53132100811\\
6	769.377794009609\\
7	578.606917285877\\
8	439.127620024747\\
9	331.570412779449\\
10	245.20761393658\\
11	173.663809851613\\
12	112.683621682394\\
13	59.6033339236383\\
14	12.5396525941114\\
15	29.8431165285474\\
16	26.3321616428659\\
17	23.4063659047446\\
18	20.9425379147801\\
19	18.848284123306\\
20	17.053209444893\\
21	15.502917677159\\
22	14.1548378791587\\
23	12.9752680558814\\
24	11.937246611431\\
25	11.0189968720919\\
26	10.2027748815461\\
27	9.47400524715828\\
28	8.82062557493539\\
29	8.23258386993347\\
30	7.70144942672526\\
31	7.22010883754029\\
32	6.78252648375446\\
33	6.38355433765592\\
34	6.01877980406747\\
35	5.68440314829243\\
36	5.37713811325851\\
37	5.09413084413357\\
38	4.83289336494348\\
39	4.59124869669587\\
40	4.367285345619\\
41	4.15931937680609\\
42	3.96586266160912\\
43	3.78559617697342\\
44	3.61734745799641\\
45	3.46007148156664\\
46	3.31283439725936\\
47	3.17479963069486\\
48	3.04521597229729\\
49	2.92340733340272\\
50	2.808763908572\\
51	2.70073452747315\\
52	2.59882001698529\\
53	2.50256742376733\\
54	2.41156497199582\\
55	2.32543765156475\\
56	2.24384334800693\\
57	2.16646943946711\\
58	2.09302979742632\\
59	2.02326213752695\\
60	1.95692567399973\\
61	1.89379903935394\\
62	1.83367843494882\\
63	1.77637598384038\\
64	1.72171826128583\\
65	1.66954498060832\\
66	1.61970781699751\\
67	1.57206935182049\\
68	1.52650212421918\\
69	1.48288777783499\\
70	1.44111629110023\\
71	1.40108528302363\\
72	1.36269938486228\\
73	1.32586967176336\\
74	1.29051314717257\\
75	1.25655227487947\\
76	1.22391455346818\\
77	1.19253212900536\\
78	1.16234144221371\\
79	1.1332829061356\\
80	1.10530061215533\\
81	1.07834206065728\\
82	1.05235791461509\\
83	1.02730177379873\\
84	1.00312996734302\\
85	0.979801363455519\\
86	0.957277194176485\\
87	0.93552089431814\\
88	0.914497952859319\\
89	0.894175776142278\\
90	0.874523561255222\\
91	0.855512179506608\\
92	0.837114068113611\\
93	0.819303130493604\\
94	0.802054643554258\\
95	0.785345171796789\\
96	0.769152487850838\\
97	0.753455498327956\\
98	0.738234175111246\\
99	0.723469491596011\\
100	0\\
};
\addlegendentry{Alg.~3, $p=1$}

\addplot [color=red,thick]
  table[row sep=crcr]{%
1	6000\\
2	3742.41661699306\\
3	2281.9756873241\\
4	1485.89110055136\\
5	1039.51926404203\\
6	757.64939123519\\
7	568.237043426403\\
8	441.962144887195\\
9	353.569715909748\\
10	289.284313017086\\
11	241.070260847544\\
12	203.982528409456\\
13	174.84216720814\\
14	151.52987824704\\
15	132.588643466134\\
16	116.989979528928\\
17	103.991092914582\\
18	93.0446620814817\\
19	83.7401958733286\\
20	75.7649391234723\\
21	68.8772173849785\\
22	62.8878941341165\\
23	57.6472362896037\\
24	53.0354573864452\\
25	48.9558068182546\\
26	45.3294507576469\\
27	42.0916328463851\\
28	39.188761615608\\
29	36.5761775079024\\
30	34.2164241203098\\
31	32.0778976127993\\
32	30.1337826059591\\
33	28.3612071585287\\
34	26.7405667494725\\
35	25.2549797078422\\
36	23.8898456695918\\
37	22.6324853711876\\
38	21.4718450957393\\
39	20.3982528409395\\
40	19.4032161170007\\
41	18.4792534447576\\
42	17.6197532845274\\
43	16.8188554079451\\
44	16.0713507231333\\
45	15.3725963438675\\
46	14.7184433079515\\
47	14.1051748367976\\
48	13.5294534148902\\
49	12.9882752782821\\
50	12.4789311497146\\
51	11.9989722593685\\
52	11.546180853343\\
53	11.1185445254186\\
54	10.714233815404\\
55	10.3315826077353\\
56	9.96907093729133\\
57	9.62530987049326\\
58	9.29902817996852\\
59	8.9890605739752\\
60	8.69433727646193\\
61	8.41387478368079\\
62	8.14676764768513\\
63	7.89218115869375\\
64	7.64934481535236\\
65	7.41754648760681\\
66	7.19612718948777\\
67	6.98447638978786\\
68	6.78202779876345\\
69	6.58825557593863\\
70	6.40267091183967\\
71	6.22481894206591\\
72	6.0542759573507\\
73	5.89064687744539\\
74	5.73356296071102\\
75	5.58267972491488\\
76	5.43767505673658\\
77	5.2982474911929\\
78	5.16411464332924\\
79	5.03501177724945\\
80	4.9106904988148\\
81	4.79091755984568\\
82	4.67547376322403\\
83	4.56415295934374\\
84	4.45676112499113\\
85	4.35311551740641\\
86	4.25304389632296\\
87	4.15638380777537\\
88	4.06298192446336\\
89	3.97269343725546\\
90	3.88538149357962\\
91	3.80091667850674\\
92	3.71917653488533\\
93	3.6400451192641\\
94	3.56341259043993\\
95	3.48917482816266\\
96	3.41723307913216\\
97	3.34749362855024\\
98	3.27986749462093\\
99	3.21427014473767\\
100	0\\
};
\addlegendentry{Alg.~4, $p=1$}

\addplot [color=black, dashed]
  table[row sep=crcr]{%
1	10000\\
2	2500\\
3	1111.11111111111\\
4	625\\
5	400\\
6	277.777777777778\\
7	204.081632653061\\
8	156.25\\
9	123.456790123457\\
10	100\\
11	82.6446280991736\\
12	69.4444444444444\\
13	59.1715976331361\\
14	51.0204081632653\\
15	44.4444444444444\\
16	39.0625\\
17	34.6020761245675\\
18	30.8641975308642\\
19	27.7008310249307\\
20	25\\
21	22.6757369614512\\
22	20.6611570247934\\
23	18.9035916824197\\
24	17.3611111111111\\
25	16\\
26	14.792899408284\\
27	13.7174211248285\\
28	12.7551020408163\\
29	11.8906064209275\\
30	11.1111111111111\\
31	10.4058272632674\\
32	9.765625\\
33	9.18273645546373\\
34	8.65051903114187\\
35	8.16326530612245\\
36	7.71604938271605\\
37	7.30460189919649\\
38	6.92520775623269\\
39	6.57462195923734\\
40	6.25\\
41	5.94883997620464\\
42	5.66893424036281\\
43	5.40832882639265\\
44	5.16528925619835\\
45	4.93827160493827\\
46	4.72589792060491\\
47	4.52693526482571\\
48	4.34027777777778\\
49	4.1649312786339\\
50	4\\
51	3.84467512495194\\
52	3.69822485207101\\
53	3.55998576005696\\
54	3.42935528120713\\
55	3.30578512396694\\
56	3.18877551020408\\
57	3.07787011388119\\
58	2.97265160523187\\
59	2.87273771904625\\
60	2.77777777777778\\
61	2.68744961031981\\
62	2.60145681581686\\
63	2.51952632905014\\
64	2.44140625\\
65	2.36686390532544\\
66	2.29568411386593\\
67	2.22766763198931\\
68	2.16262975778547\\
69	2.10039907582441\\
70	2.04081632653061\\
71	1.98373338623289\\
72	1.92901234567901\\
73	1.87652467629949\\
74	1.82615047479912\\
75	1.77777777777778\\
76	1.73130193905817\\
77	1.68662506324844\\
78	1.64365548980934\\
79	1.60230732254446\\
80	1.5625\\
81	1.52415790275873\\
82	1.48720999405116\\
83	1.45158949049209\\
84	1.4172335600907\\
85	1.3840830449827\\
86	1.35208220659816\\
87	1.32117849121416\\
88	1.29132231404959\\
89	1.26246686024492\\
90	1.23456790123457\\
91	1.20758362516604\\
92	1.18147448015123\\
93	1.15620302925194\\
94	1.13173381620643\\
95	1.10803324099723\\
96	1.08506944444444\\
97	1.06281220108407\\
98	1.04123281965848\\
99	1.02030405060708\\
100	1\\
};
\addlegendentry{$10^4/k^2$}

\end{axis}
\end{tikzpicture}%

%% file: conclusion.tex
\vspace*{-6pt}\section{Conclusion}\label{Sec:Conc}
\vspace*{-2pt}We have introduced a new type of accelerated optimization algorithm for constrained optimization problems. By imposing constraints on velocities, rather than on positions, the algorithms avoid projections or optimizations over the entire feasible set at each iteration. This not only has the potential to reduce execution time compared to Frank-Wolfe or projected gradient schemes, but more importantly, expands the range of potential applications, as constraints are not necessarily required to be convex or to have a simple structure. We have highlighted important analogies to non-smooth dynamical systems, and characterized the algorithm's behavior in continuous and discrete time.


%% file: appendixS2.tex
\section{Supplementary Material of 
Sec.~\ref{Sec:GF}}\label{App:Sec1}
\subsection{Non-asymptotic Rates for Gradient Descent}
\hchange{The following section contains a non-asymptotic convergence result for the gradient descent algorithm introduced in \cite{ownWorkC}, which only contained an asymptotic analysis. The pseudo-code of the algorithm is listed in Alg.~\ref{alg:CGDV} and requires knowledge of the smoothness constant $L_l$ of the Lagrangian. Sec.~\ref{Sec:GF2} below provides results in the situation where $L_l$ is not known.}

\begin{algorithm}
\caption{Constrained gradient descent with velocity projections}\label{alg:CGDV}
\begin{algorithmic}
\Require $x_0\in \mathbb{R}^n$, $L_l,\mu$ (smoothness and strong convexity constant of $l$)
\State $T \leq 1/L_l$, $\alpha \leq \mu$
\For{$k=0,1,\dots$}
    \State $x_{k+1}\gets \hchange{x_k+}T \argmin_{v\in V_\alpha(x_k)} |v+\nabla f(x_k)|^2$ (with $I=[n_\text{g}]$)
\EndFor
\end{algorithmic}
\end{algorithm}

\begin{theorem}
Let the function $f$ be $\mu$-strongly convex and \hchange{smooth}, and $g$ be smooth and concave. Let $l(x):=f(x)-{\lambda^*}\T g(x)$ denote the Lagrangian \hchange{with smoothness constant $L_l$}, where $\lambda^*$ denotes an optimal multiplier of \eqref{eq:fundProb}, and let $T\leq 1/L_l$, $\alpha\leq \mu$. Then, the iterates of Alg.~\ref{alg:CGDV} satisfy:
\begin{equation*}
\frac{\mu}{2} |x_k-x^*|^2 ~~ \leq ~~ l(x_k)-l(x^*) ~~\leq~~ (1-T \alpha)^k ~~(l(x_0)-l(x^*)),
\end{equation*} 
for all $k\geq 0$, where $x^*$ denotes the minimizer of \eqref{eq:fundProb}. For $T=1/L_l$ and $\alpha=\mu$, the convergence is linear at rate \hchange{$\mu/L_l$}. \label{Thm:GDV}
\end{theorem}

\begin{proof}
The proof hinges on the fact that $\alpha (x^*-x_k) \in V_\alpha(x_k)$, which follows from concavity of $g$. We define $v_k:=(x_{k+1}-x_k)/T$ to be the velocity. From the fact that $v_k=\argmin_{v\in V_\alpha(x_k)} |v+\nabla f(x_k)|^2$ we conclude
\begin{align*}
|v_k+\nabla f(x_k)|^2/2 \leq |\alpha (x^*-x_k) + \nabla f(x_k)|^2/2, 
\end{align*}
which can be rearranged to
\begin{align}
|v_k|^2/2+\nabla f(x_k)\T v_k &\leq \alpha^2 |x^*-x_k|^2/2 + \alpha \nabla f(x_k)\T (x^*-x_k)\nonumber\\
&\leq (\alpha^2-\alpha\mu) |x^*-x_k|^2/2 + \alpha (f(x^*)-f(x_k))\nonumber\\
&\leq \alpha (f(x^*)-f(x_k)), \label{eq:app2}
\end{align}
where we have used the strong convexity of $f$ in the second step and the fact that $\alpha \leq \mu$ in the third step. Next we use the smoothness of $l$, which yields
\begin{align*}
l(x_{k+1})-l(x_k) &\leq T \nabla l(x_k)\T v_k + \frac{T^2 L_l}{2} |v_k|^2\\
&\leq T \nabla f(x_k)\T v_k - T {\lambda^*}\T \nabla g(x_k)\T v_k + \frac{T^2 L_l}{2} |v_k|^2\\
&\leq \alpha T (f(x^*)-f(x_k))- T {\lambda^*}\T \nabla g(x_k)\T v_k + T \frac{TL_l - 1}{2} |v_k|^2, 
\end{align*}
where we used inequality \eqref{eq:app2} in the third step. We further note that by definition of $V_\alpha(x_k)$, $v_k\in V_\alpha (x_k)$ satisfies $\nabla g(x_k)\T v_k \geq - \alpha g(x_k)$ and therefore $-{\lambda^*}\T \nabla g(x_k)\T v_k \leq - \alpha {\lambda^*}\T g(x_k)$.  In addition, complementary slackness implies ${\lambda^*}\T g(x^*)=0$ and therefore 
\begin{equation*}
l(x_{k+1})-l(x_k)\leq \alpha T (l(x^*)-l(x_k)) +T \frac{T L_l-1}{2} |v_k|^2. 
\end{equation*}
Subtracting and adding $l(x^*)$ on the left-hand side and using $T\leq 1/L_l$ yields
\begin{equation*}
\frac{\mu}{2} |x_k-x^*|^2 ~~ \leq ~~ l(x_{k}) - l(x^*) ~~ \leq ~~ (1-\alpha T)^k ~~ (l(x_0)-l(x^*)),
\end{equation*}
which concludes the proof.
\end{proof}

\subsection{Smoothness-agnostic Rates}\label{Sec:GF2}
Thm.~\ref{Thm:GDV} requires knowledge of $L_l$, the smoothness constant of the Lagrangian, which requires a bound on $\lambda^*$. This is further discussed in App.~\ref{App:Sec1}, where we also add a corollary that shows that a $1/k$ rate is obtained if the smoothness constant $L_l$ is unknown. 
The application of Thm.~\ref{Thm:GDV} requires knowledge of the smoothness constant $L_l$ of $l$, which may require a bound on the optimal dual multiplier $\lambda^*$ if $g$ is nonlinear. This can be viewed as a quantitative constraint qualification, since the multiplier $\lambda^*$ can be bounded by
\begin{equation*}
|\lambda^*|\leq \sup_{x\in C} \frac{|W(x)\T \nabla f(x)|}{\underline{\sigma}(W(x))}, \quad W(x):=(\nabla g_i(x))_{i\in I_x},
\end{equation*}
where $\underline{\sigma}$ denotes the minimum singular value and $W(x)$ the matrix with columns $\nabla g_i(x)$, $i\in I_x$. 

If such a quantitative constraint qualification is not available, the following corollary of Thm.~\ref{Thm:GDV} still applies and characterizes the convergence of Alg.~\ref{alg:CGDV}:
\begin{corollary}
Let the function $f$ be $\mu$-strongly convex and $L$-smooth, let $g$ be concave and $L_\text{g}$-smooth, and let 
\begin{equation*}
B:=\max_{x\in C} |\nabla f(x)|^2/(2\mu).
\end{equation*}
Then, the iterate $x_N$ of Alg.~\ref{alg:CGDV} with $T=1/(L+B L_g/\varepsilon)$, $\alpha=\mu$ satisfies $|x_N-x^*|\leq \varepsilon$, where
\begin{equation*}
N\geq\frac{L+B L_\text{g}/\varepsilon}{\mu} \left( 2\log(1/\varepsilon) + \log(\frac{L+B L_\text{g}/\varepsilon}{\mu}) + 2 \log(|x^*-x_0|)\right) = \mathcal{O}\left(\frac{\log(1/\varepsilon)}{\varepsilon}\right)
\end{equation*}
and $x^*$ denotes the minimizer of
\begin{equation*}
\min_{x\in \mathbb{R}^n} f(x) \quad \text{s.t.}\quad g(x)\geq - \varepsilon.
\end{equation*}
\end{corollary}
\begin{proof}
We start by proving the following claim:

\textit{Claim:} Let $\lambda^*$ be an optimal multiplier of $\min_{x\in \mathbb{R}^n} f(x)~\text{s.t.}~g(x)\geq -\varepsilon$. Then $\lambda^*$ satisfies $|\lambda^*|_1 \leq B/\varepsilon$.

\textit{Proof of the claim:} We consider the slightly modified problem
\begin{equation}
\min_{x\in \mathbb{R}^n, \xi\geq 0} f(x) + \bar{B} \xi/\varepsilon \quad \text{s.t.} \quad g(x)\geq - \xi,\label{eq:modprob}
\end{equation}
and denote its unique minimizer by $x^*(\bar{B}), \xi^*(\bar{B})$ (strong convexity concludes that the minimizer is unique). In the above minimization the variable $\xi$ can be interpreted as slackness, whereby the factor $\bar{B}/\varepsilon$ penalizes the magnitude of $\xi$. The larger the factor $\bar{B}/\varepsilon$, the smaller the amount of constraint violation by the resulting $x^*(\bar{B})$. The modified problem \eqref{eq:modprob} is motivated by the fact that the corresponding Lagrange dual problem is given by
\begin{equation*}
\max_{\lambda \geq 0, |\lambda|_1\leq \bar{B}/\varepsilon} ~ \min_{x\in \mathbb{R}^n} f(x) - \lambda\T g(x),
\end{equation*}
which implies that the minimizer $x^*(\bar{B})$ satisfies
\begin{equation}
x^*(\bar{B})=\argmin_{x\in \mathbb{R}^n} f(x) - \lambda^*(\bar{B})\T g(x), \quad \text{with}\quad |\lambda^*(\bar{B})|_1\leq \bar{B}/\varepsilon, \label{eq:dualApp}
\end{equation}
and where $\lambda^*(\bar{B})$ is an optimal dual multiplier. The maximum theorem implies that $x^*(\bar{B})$ and $\xi^*(\bar{B})$ are continuous functions of $\bar{B}$. We will show next that for $\bar{B}=B$ the amount of constraint violation $g(x^*(B))$ is guaranteed to be below $\varepsilon$. Combined with the continuity of $\xi^*(\bar{B})$, this implies that there exists some $B'\leq B$ for which the corresponding $x^*(B')$ satisfies $x^*(B')=\argmin_{x\in \mathbb{R}^n, g(x)\geq -\varepsilon} f(x)$. This in turn proves the claim in view of \eqref{eq:dualApp}. 

In order to bound the amount of constraint violation for $\bar{B}=B$, we note that the following holds for any $x\in C$
\begin{align*}
f(x^*(\bar{B})) + \bar{B} \xi^*(\bar{B})/\varepsilon \leq f(x) &\leq f(\hat{x}) + |\nabla f(x)|^2/(2\mu)\\
&\leq f(x^*(\bar{B})) +|\nabla f(x)|^2/(2\mu),
\end{align*}
where strong convexity of $f$ has been used in the second step and $\hat{x}$ denotes the unconstrained minimizer of $f$. Maximizing the right-hand side over $x\in C$ and rearranging terms yields
\begin{equation*}
\xi^*(\bar{B}) \leq \varepsilon B/\bar{B},
\end{equation*}
which establishes that $\xi^*(B)\leq \varepsilon$ and proves the claim.

We now turn to the proof of the corollary. As a result of the claim we conclude that the minimizer $x^*$ satisfies $x^*=\argmin_{x\in \mathbb{R}^n} f(x) - {\lambda^*}\T g(x)$, where the $\ell^1$ norm of $\lambda^*$ is bounded by $B/\varepsilon$. As a result of the choice of $T$ and the parameter $\alpha$, we conclude from Thm.~\ref{Thm:GDV} that
\begin{equation*}
\frac{\mu}{2} |x_N-x^*|^2 \leq \left(1-\frac{L+B L_\text{g}/\varepsilon}{\mu} \right)^N ~~\frac{L+B L_\text{g}/\varepsilon}{2} ~~ |x_0-x^*|^2,
\end{equation*}
where the smoothness of $f(x)-{\lambda^*}\T g(x)$ has been used to upper bound the right-hand side. The corollary follows from using the identity $1-(L+B L_\text{g}/\varepsilon)/\mu \leq \exp(- (L+B L_\text{g}/\varepsilon)/\mu)$, which yields
\begin{multline*}
|x_N-x^*|^2 \leq \frac{L+B L_\text{g}/\varepsilon}{\mu} |x_0-x^*|^2 \\
\exp\left( -2\log(1/\varepsilon) - \log(\frac{L+B L_\text{g}/\varepsilon}{\mu}) - 2 \log(|x^*-x_0|)\right) = \varepsilon^2
\end{multline*}
and proves the desired result.
\end{proof}

%% file: AppendixProof.tex
\section{Pseudo-code of Algorithm \eqref{eq:dis1}}\label{App:Alg}
This section lists the pseudo-code of the resulting discrete-time algorithm \eqref{eq:dis1}.

\begin{algorithmic}
\Require parameters $\beta\geq 0$, $\delta >0$ \hfill\Comment{typically below 1; can be time varying}
\Require step size $T_k$ \hfill\Comment{$1/T_k \sim$ smoothness of objective}
\Require constants $\alpha>0$, $0\leq \epsilon\leq 1$ \hfill\Comment{$\alpha$ such that $\alpha T_k\leq 1$}
\Require initial condition $x_0\in \mathbb{R}^n$, $u_0\in \mathbb{R}^n$ \hfill\Comment{usually $u_0=0$}
\Require $\epsilon_{\text{const}}\geq 0$ \hfill\Comment{constraint violation tolerance}
\State $x_k \gets x_0$, $u_k \gets u_0$
\For{$k=0,1,\dots$}
\State $r_k\gets u_k-2\delta T_k u_k - \nabla f(x_k+\beta u_k) T_k$  \hfill\Comment{unconstrained update}
\State $w\gets \{ \}$, $W \gets \{ \}$
\For{$i=1,\dots,n_\text{g}$} \Comment{check constraint violations}
\If{$g_i(x_k)\leq \epsilon_{\text{const}}$}
\State $w\gets (w, -\alpha g_i(x_k) - \epsilon \min\{ \nabla g_i(x_k)\T u_k + \alpha g_i(x_k), 0\})$ 
\State $W\gets (W, \nabla g_i(x_k)\T)$ \hfill\Comment{add violated constraint}
\EndIf
\EndFor
\State $u_{k+1} \gets \argmin_{v \in \mathbb{R}^n} |v-r_k|^2$~\text{s.t.}~$W v \geq w$ \hfill\Comment{solve \eqref{eq:dis1}}
\State $x_{k}\gets x_k + T_k u_{k+1}$ \hfill\Comment{update positions}
\State $u_k \gets u_{k+1}$ \hfill\Comment{update velocities}
\EndFor
\end{algorithmic}

\section{Additional Details on the Numerical Experiments}\label{App:Example2}
We apply \eqref{eq:dis1} and \eqref{eq:disMod} to \eqref{eq:imp} and start by defining the constraint function $g$ as $g(x,\bar{x}):=(g_1(x,\bar{x}),g_2(x,\bar{x}))$, where $g_1(x,\bar{x}):=\bar{x}+x$ and $g_2(x,\bar{x}):=\bar{x}-x$. The decision variables $(x,\bar{x})$ will be represented with the iterates $x_k\in \mathbb{R}^n$, $\bar{x}_k\in \mathbb{R}^n$, and the corresponding velocities will be denoted by $u_k\in \mathbb{R}^n$, $\bar{u}_k\in \mathbb{R}^n$, where $k$ refers to the iteration number. As a result, the optimization in \eqref{eq:dis1}, which is used to determine the velocities $(u_{k+1},\bar{u}_{k+1})$ can be expressed as
\begin{align}
    (u_{k+1},\bar{u}_{k+1})=&\argmin_{(v,\bar{v})\in\mathbb{R}^{2n}} \frac{1}{2} |v-r_{k}|^2 + \frac{1}{2} |\bar{v}-\bar{r}_{k}|^2 \quad \text{s.t.} \label{eq:optitmp}\\
    &\quad\bar{v}_i+v_i+\tilde{g}_{1i} \geq 0, \quad \forall i\in I_{1k},\quad \bar{v}_j-v_j +\tilde{g}_{2j}\geq 0, \quad \forall j \in I_{2k},  \nonumber\\
    &\quad-\sum_{i=1}^n w_l \bar{v}_l + \tilde{g}_{3}\geq 0, \nonumber
\end{align}
where the last constraint only enters if $\sum_{i=1}^n (\bar{x}_{ki})_\Delta^p \geq \nu$ holds, and where $I_{ik}$ denotes the set of constraints in $g_i$ that are active at iteration $k$, that is,
\begin{equation*}
    I_{1k}:=\{i \in [n]~|~g_{1i}(x_k,\bar{x}_k)\leq 0\}, \quad I_{2k}:=\{i\in [n]~|~g_{2i}(x_k,\bar{x}_k)\leq 0\}. \end{equation*}
The variables $\tilde{g}_1$, $\tilde{g}_2$, $\tilde{g}_3$, and $w$ are defined as follows: 
$\tilde{g}_{1} := \alpha g_1(x_k,\bar{x}_k)$, $\tilde{g}_2:=\alpha g_2(x_k,\bar{x}_k)$, $\tilde{g}_3:=\alpha (\nu - \sum_{i=1}^n (\bar{x}_{ki})_\Delta^p)$, and $w_i:=\text{d} (x)^p_\Delta /\text{d}x |_{x=\bar{x}_{ki}}$. We have further set the constant $\epsilon=0$ when deriving \eqref{eq:optitmp} in order to simplify the exposition. Moreover, $r_{k}\in \mathbb{R}^n$ and $\bar{r}_{k}\in \mathbb{R}^n$ are given by
\begin{equation*}
    r_{k}:=u_k-2\delta T u_k - A\T (A (x_k+\beta u_k)-b) T,  \qquad 
    \bar{r}_{k}:=\bar{u}_k-2\delta T \bar{u}_k,
\end{equation*}
where $\delta>0$ and $\beta\geq 0$ are damping parameters and $T>0$ denotes the step size. In fact, the optimization in \eqref{eq:disMod} has the same form as \eqref{eq:optitmp} with slightly different definitions of $\tilde{g}_1, \tilde{g}_2,$ and $\tilde{g}_3$ and all constraints being permanently active, which means, in particular, $I_{1k}=[n]$ and $I_{2k}=[n]$. In both cases, we apply the following change of variables $(v,\bar{v})\rightarrow (\xi,\bar{\xi})$, 
\begin{equation*}
v=\xi-\bar{\xi} -\frac{\tilde{g}_1-\tilde{g}_2}{2}, \qquad \bar{v}=\xi+\bar{\xi} - \frac{\tilde{g}_1+\tilde{g}_2}{2},
\end{equation*}
which transforms \eqref{eq:optitmp} to a projection onto a weighted simplex. The projection onto a weighted simplex has a closed-form solution and takes the form
\begin{align}
\argmin_{(\xi,\bar{\xi})\in \mathbb{R}^{2n}}& \frac{1}{2} |\xi-q_1|^2+\frac{1}{2}|\bar{\xi}-q_2|^2 \quad \text{s.t.} \label{eq:optitmp2}\\
&\xi_i \geq 0,~\forall i\in I_{1k}, \quad \bar{\xi}_j \geq 0,~\forall j\in I_{2k},\quad \sum_{i=1}^n w_i \xi_i + \sum_{i=1}^n w_i \bar{\xi}_i \leq \bar{\nu},\nonumber
\end{align}
where the last constraint only enters if $\sum_{i=1}^n (\bar{x}_{ki})_\Delta^p \geq \nu$ and
where $q_1$, $q_2$, and $\bar{\nu}$ arise from the change of variables. The weights $w_i$ are given by $\text{d}(x)_\Delta^p/\text{d}x$ and are therefore nonnegative by construction of $(\cdot)_\Delta^p$. The closed-form solution of \eqref{eq:optitmp2} is summarized in Alg.~\ref{Alg:weightedl1} below. The solution hinges on the fact that the constraint $\sum_{i=1}^n w_i \xi_i + \sum_{i=1}^n w_i \bar{\xi}_i \leq \bar{\nu}$ can be included in the objective function by means of a dual multiplier, at which point the resulting minimization corresponds to a simple projection on $\mathbb{R}_{\geq 0}$ for each $i\in I_{1k}$ and $i\in I_{2k}$. The optimal multiplier is computed by an iterative procedure that is based on a sorting operation. The worst-case complexity of Alg.~\ref{Alg:weightedl1} is therefore $\mathcal{O}(n \text{log}(n))$.

The resulting algorithms for solving the minimization in \eqref{eq:dis1} and \eqref{eq:disMod} are summarized in Alg.~\ref{Alg:ImageDenoising} and Alg.~\ref{Alg:ImageDenoising2}, respectively.

\begin{algorithm}
\caption{Projection onto the weighted simplex: \\
$\argmin_{\xi\in \mathbb{R}^n} |\xi-q|^2 ~~\text{s.t.}~~\xi_i\geq 0, \forall i\in I,~~w\T \xi \leq \bar{\nu}$}
\label{Alg:weightedl1}
\begin{algorithmic}
\Require $q\in \mathbb{R}^n$, index set $I$, weight $w\in \mathbb{R}^n$, and bound $\bar{\nu}\in \mathbb{R}$
\Require $w\geq 0$ and $\bar{\nu}\geq 0$
\State $\xi_i \gets \begin{cases} 0,  &\text{if} ~ i\in I~\text{and}~q_i< 0,\\
q_i, &\text{else}. \end{cases}$ \hfill\Comment{suppose constraint $w\T \xi\leq \bar{\nu}$ is not active}
\If{$w\T\xi > \bar{\nu}$} \hfill\Comment{constraint $w\T \xi\leq \bar{\nu}$ is active}
\State $q_\text{s}/w_\text{s} \gets \text{sort}(\{q_i/w_i\}_{i\in I}, \text{descending})$ \hfill\Comment{sort s.t. $q_{\text{s}1}/w_{\text{s}1} \geq q_{\text{s}2}/w_{\text{s}2} \geq \dots$}
\State $m\gets 1$\hfill\Comment{compute multiplier $\lambda$ for $w\T \xi \leq \bar{\nu}$}
\While{$m \leq |I|$}
\If{$\sum_{i\not\in I} w_i q_i +\sum_{j<m} w_{\text{s}j} q_{\text{s}j} -(\sum_{i\not\in I} w_i^2 + \sum_{j<m} w_{\text{s}j}^2) q_{\text{s}m}/w_{\text{s}m}>\bar{\nu}$}
\State \textbf{break}
\EndIf
\State $m\gets m+1$
\EndWhile
\State $\lambda \gets (\sum_{i\not\in I} w_i q_i +\sum_{j<m} w_{\text{s}j} q_{\text{s}j}-\bar{\nu})/(\sum_{i\not\in I} w_i^2 + \sum_{j<m} w_{\text{s}j}^2)$
\State $\xi \gets q-w\lambda$ \hfill \Comment{given $\lambda$, perform proj. on $\mathbb{R}_{\geq 0}$ for each $i\in I$}
\State $\xi_i \gets 0$~if~$i\in I$~and~$\xi_i<0$
\EndIf
\State \Return $\xi$
\end{algorithmic}
\end{algorithm}

\begin{algorithm}[H]
\caption{Algorithm \eqref{eq:dis1} applied to \eqref{eq:inverse}.}
\label{Alg:ImageDenoising}
\begin{algorithmic}
\Require damping parameters $\beta_k\geq 0$, $\delta_k >0$; step size $T_k$; $L=|A|^2$; $\alpha_k >0$ 
\Require initial condition $x_0\in \mathbb{R}^n$; approximation parameter $\Delta>0$
\State $x_k \gets x_0, \bar{x}_k \gets \text{abs}(x_0), u_k \gets 0, \bar{u}_k \gets 0$ \hfill\Comment{different initialization is also possible}
\For{$k=0,1,\dots$}
\State $r_k\gets u_k-2\delta_k T_k u_k - T_k A\T(A (x_k+\beta_k u_k)-b)/L$ \hfill\Comment{unconstrained update}
\State $\bar{r}_k\gets \bar{u}_k-2\delta T_k \bar{u}_k$
\State
\State $\tilde{g}_1 \gets \alpha_k(x_k+\bar{x}_k)$ \hfill\Comment{evaluate constraints}
\State $\tilde{g}_2 \gets \alpha_k (-x_k+\bar{x}_k)$
\State $\tilde{g}_3 \gets \alpha_k (\nu - \sum_{i=1}^n (\bar{x}_{ki})_\Delta^p)$
\State  
\If{$\tilde{g}_3>0$} \hfill \Comment{evaluate gradient of $\ell^p$ constraint}
\State $w \gets 0$
\Else
\State $w_i \gets \text{d} (x)^p_\Delta/\text{d}x|_{x=\bar{x}_{ki}}, \quad i=1,2,\dots,n,$
\EndIf
\State
\State $q \gets (\tilde{g}_1/2+(r_k+\bar{r}_k)/2,~~\tilde{g}_2/2+(-r_k+\bar{r}_k)/2)$ \hfill\Comment{change of variables}
\State $\bar{\nu}\gets \tilde{g}_3+w\T (\tilde{g}_1+\tilde{g}_2)/2$
\State $I \gets \{i\in [n]:(\tilde{g}_1,\tilde{g}_2)_i \leq 0\}$ \hfill\Comment{active constraints}
\State
\State $(\xi,\bar{\xi}) \gets \text{ProjectionOnWeightedSimplex}(q, I, (w,w), \bar{\nu})$ \hfill\Comment{apply Alg.~\ref{Alg:weightedl1}}
\State
\State $u_{k}\gets \xi-\bar{\xi}-(\tilde{g}_1-\tilde{g}_2)/2$\hfill\Comment{update velocity}

\State $\bar{u}_{k} \gets \xi+\bar{\xi}- (\tilde{g}_1+\tilde{g}_2)/2$
\State $x_{k}\gets x_k + T_k u_{k}$ \hfill\Comment{update position}
\State $\bar{x}_k \gets \bar{x}_k + T_k \bar{u}_{k}$
\EndFor
\end{algorithmic}
\end{algorithm}

\begin{algorithm}[H]
\caption{Algorithm \eqref{eq:disMod} applied to \eqref{eq:inverse}.}
\label{Alg:ImageDenoising2}
\begin{algorithmic}
\Require damping parameters $\beta_k\geq 0$, $\delta_k >0$; step size $T_k$; $L=|A|^2$; $\alpha_k >0$ 
\Require initial condition $x_0\in \mathbb{R}^n$; approximation parameter $\Delta>0$
\State $x_k \gets x_0, \bar{x}_k \gets \text{abs}(x_0), u_k \gets 0, \bar{u}_k \gets 0$ \hfill\Comment{different initialization is also possible}
\For{$k=0,1,\dots$}
\State $r_k\gets u_k-2\delta_k T_k u_k - T_k A\T(A (x_k+\beta_k u_k)-b)/L$ \hfill\Comment{unconstrained update}
\State $\bar{r}_k\gets \bar{u}_k-2\delta T_k \bar{u}_k$
\State 
\State $w_i \gets \text{d} (x)^p_\Delta/\text{d}x|_{x=\bar{x}_{ki}+\beta_k \bar{u}_{ki}}, \quad i=1,2,\dots,n,$ \Comment{gradient of $\ell^p$ constraint}
\State
\State $g_1 \gets x_k+\bar{x}_k$ \hfill\Comment{evaluate constraints at $x_k$}
\State $g_2 \gets -x_k+\bar{x}_k$
\State $g_3 \gets \nu - \sum_{i=1}^n (\bar{x}_{ki})_\Delta^p$
\State 
\State $g_{y1} \gets x_k+\beta_k u_k + \bar{x}_k+\beta_k \bar{u}_k$\Comment{evaluate constraints at $x_k+\beta_k u_k$}
\State $g_{y2} \gets -(x_k+\beta_k u_k)+\bar{x}_k+\beta_k \bar{u}_k$
\State $g_{y3} \gets \nu - \sum_{i=1}^n (\bar{x}_{ki}+\beta_k \bar{u}_{ki})_\Delta^p$
\State
\State $\tilde{g}_1 \gets \alpha_k g_1+ (g_{y1}-g_1-\beta_k (u_k+\bar{u}_k))/T_k$ \hfill\Comment{r.h.s. of constraints in \eqref{eq:disMod}}
\State $\tilde{g}_2 \gets \alpha_k g_2+ (g_{y2}-g_2-\beta_k (-u_k+\bar{u}_k))/T_k$
\State $\tilde{g}_3 \gets \alpha_k g_3+ (g_{y3}-g_3+\beta_k w\T \bar{u}_k)/T_k$
\State
\State $q \gets ( \tilde{g}_1/2+(r_k+\bar{r}_k)/2,~~\tilde{g}_2/2+(-r_k+\bar{r}_k)/2)$ \hfill\Comment{change of variables}
\State $\bar{\nu}\gets \tilde{g}_3+ w\T (\tilde{g}_1+\tilde{g}_2)/2$
\State 
\State $(\xi,\bar{\xi}) \gets \text{ProjectionOnWeightedSimplex}(q, \{1,\dots,2n\}, (w,w), \bar{\nu})$ \hfill\Comment{apply Alg.~\ref{Alg:weightedl1}}
\State
\State $u_{k}\gets \xi-\bar{\xi}- (\tilde{g}_1-\tilde{g}_2)/2$\hfill\Comment{update velocity}

\State $\bar{u}_{k} \gets \xi+\bar{\xi}- (\tilde{g}_1+\tilde{g}_2)/2$
\State
\State $x_{k}\gets x_k + T_k u_{k}$ \hfill\Comment{update position}
\State $\bar{x}_k \gets \bar{x}_k + T_k \bar{u}_{k}$
\EndFor
\end{algorithmic}
\end{algorithm}

\newpage


\section{Supporting Lemmas in the Proof of Thm.~\ref{Thm:ConvDT}}\label{App:ConvDT}
%
\begin{lemma}\label{Lemma:constr}
Let the assumptions of Thm.~\ref{Thm:ConvDT} be satisfied and let $\text{dist}(x,C)$ denote the distance of $x\in \mathbb{R}^n$ to the set $C$, that is, $\text{dist}(x,C)=\min_{y\in C} |y-x|$. Then, the iterates $x_k$ are bounded, $\text{dist}(x_k,C) \rightarrow 0$, and there exists a constant $c_\text{g}>0$ and $k_0>0$ such that
\begin{equation*}
g_i(x_k) \geq - c_\text{g} T_k,
\end{equation*}
for all $k\geq k_0$ and all $i\in [n_\text{g}]$. 
\end{lemma}
\begin{proof}
\hchange{We start by considering the first constraint, $g_1$, since $\min\{g_1(x),0\}$ has compact level sets and make the following claim:\\
\textit{Claim:} There exists a constant $c_\text{g1}>0$ such that for all $k\geq 1$, 
\begin{equation*}
    g_1(x_k)\geq -\frac{c_\text{g1}}{(k+c_\text{g2})^{s}},
\end{equation*}
where $c_\text{g2} := \max\{0,2s/(\alpha T_0)-1\}$.\\
\textit{Proof of the claim:} We prove the claim by induction and choose $c_\text{g1}$ sufficiently large, such that $g_1(x_1)\geq - c_\text{g1}/(1+c_\text{g2})^{s}$ (base case). We proceed to the induction step and suppose that the claim is satisfied up to iteration $k$. We now show that the claim is also satisfied for iteration $k+1$. We distinguish the following two cases.}

\hchange{i) $g_1(x_k)< 0$: As a consequence of smoothness we infer that $g_1(x_{k+1}) \geq g_1(x_k) + T_k \nabla g_1(x_k)\T u_{k+1} - T_k^2 L_{\text{g}1} |u_{k+1}|^2/2$, where $L_{\text{g}1}$ is the smoothness constant of $g_1$. Due to the fact that $g_1(x_k)< 0$ we have we have $1\in I_{x_k}$ and $\nabla g_1(x_k)\T u_{k+1} T_k \geq -\alpha T_k g_1(x_k)$. This means that
\begin{equation*}
g_1(x_{k+1}) \geq g_1(x_k) (1-\alpha T_k) - T_k^2 L_{\text{g}1} c_\text{u}^2/2, 
\end{equation*}
where we used the upper bound $c_\text{u}$ on $|u_{k+1}|$. As a consequence of the induction hypothesis and the choice of the stepsize $T_k=T_0/k^s$, we conclude
\begin{align*}
    g_1(x_{k+1}) &\geq -c_\text{g1}(k+c_\text{g2})^{-s} (1-\alpha T_0 k^{-s}) - T_0^2 L_{\text{g}1} c_\text{u}^2 k^{-2s}/2.
\end{align*}
The convexity of the function $\xi^{-s}$ for $\xi>0$ implies
\begin{equation*}
(k+c_\text{g2})^{-s} \leq (k+1+c_\text{g2})^{-s} + s (k+c_\text{g2})^{-s-1}.
\end{equation*}
In addition, due to the choice of $c_\text{g2}$, $-s(k+c_\text{g2})^{-1} + \alpha T_0 k^{-s}/2 \geq 0$ holds, which concludes that
\begin{align*}
 g_1(x_{k+1}) &\geq -c_\text{g1}(k+1+c_\text{g2})^{-s} + c_\text{g1} \alpha T_0 (k+c_\text{g2})^{-s} k^{-s}/2 - T_0^2 L_{\text{g}1} c_\text{u}^2 k^{-2s}/2\\
 &\geq -c_\text{g1}(k+1+c_\text{g2})^{-s} + c_\text{g1} \alpha T_0 (1+c_\text{g2})^{-s} k^{-2s}/2 - T_0^2 L_{\text{g}1} c_\text{u}^2 k^{-2s}/2,
\end{align*}
and where we have used $k^s/(k+c_\text{g2})^s \geq 1/(1+c_\text{g2})^s$ for all $k\geq 1$ in the second step. As a result, we conclude that for large enough $c_\text{g1}$ (i.e., $c_\text{g1}\geq T_0 L_{\text{g}1} c_\text{u}^2 (1+c_\text{g2})^s/\alpha$) the constraint violation $g_1(x_{k+1})$ is bounded below by
\begin{equation*}
    g_1(x_{k+1})\geq - \frac{c_\text{g1}}{(k+1+c_\text{g2})^s},
\end{equation*}
which proves the induction step in case i).}\\
\hchange{ii) $g_1(x_k) \geq 0$ and $g_1(x_{k+1}) < 0$ (the constraint $g_1$ becomes active at $k+1$): We have again $g_1(x_{k+1}) \geq \nabla g_1(x_k)\T T_k u_{k+1} - T_k^2 L_{\text{g}1} c_\text{u}^2/2$ as a result of the smoothness of $g_1$ and the boundedness of $u_{k+1}$. Due to the fact that $g_1(x_k) \geq 0$ we conclude that $x_k$ is bounded ($\min\{g_1(x),0\}$ has compact level sets), which means that $|\nabla g_1(x_k)| \leq \bar{c}_{\text{g1}}$ for some constant $\bar{c}_{\text{g1}}>0$. Thus, we conclude
\begin{equation*}
g_1(x_{k+1}) \geq -T_k (\bar{c}_{\text{g1}} c_{\text{u}} + T_k L_{\text{g}1} c_{\text{u}}^2/2) \geq -\frac{T_0 (2+c_\text{g2})^s}{(k+1+c_\text{g2})^s} (\bar{c}_{\text{g1}} c_{\text{u}} + T_k L_{\text{g}1} c_{\text{u}}^2/2),
\end{equation*}
where $1/k^s\leq (2+c_\text{g2})^s/(k+1+c_\text{g2})^s$ for all $k\geq 1$ has been used in the last step.}

\hchange{Combining the two cases (i) and ii)) concludes the proof of the induction step and proves the claim. The claim implies
\begin{equation*}
    g_1(x_k)\geq -\frac{c_\text{g1}}{(k+c_\text{g2})^s}\geq - \frac{c_\text{g1}}{k^s}=-c_\text{g1}T_k /T_0,
\end{equation*}
for all $k\geq 1$, which means that the statement of the lemma is satisfied for the first constraint. This further implies that $x_k$ is bounded, since $\min\{g_1(x),0\}$ has compact level sets. We can now apply a similar reasoning to all the remaining constraints, which establishes that $\text{dist}(x_k,C) \rightarrow 0$ and $g_i(x_k)\geq -c_\text{g} T_k$ for some constant $c_\text{g}>0$, all $i=1,2,\dots,n_\text{g}$, and all $k$ large enough.}
\end{proof}

\begin{lemma}\label{lemma:d2}
Let $\bar{C}$ be a compact set and let $g$ satisfy the Mangasarian-Fromovitz constraint qualification on $\bar{C}$. Then, there exists a constant $c_\lambda$ such that for all $x$ with $x\in \bar{C}$ and $\lambda \in \mathbb{R}^{n_\text{g}}$ with $\lambda_i\geq 0$ $\forall i\in I_x$ and $\lambda_i = 0$ $\forall i\not\in I_x$,
\begin{equation*}
|\sum_{i=1}^{n_\text{g}} \nabla g_i(x) \lambda_i| \geq c_\lambda | \lambda|.
\end{equation*}
\end{lemma}
\begin{proof}
The inequality is satisfied for $\lambda=0$. We therefore consider the case $\lambda\neq 0$ and without loss of generality assume that $|\lambda|=1$. We argue by contradiction and therefore assume that there are two convergent sequences, $\lambda^j$ and $x_j$, with $|\lambda^j|=1$, $\lambda^j_i\geq 0$ for all $i\in I_{x_j}$, $\lambda^j_i=0$ for all $i\not\in I_{x_j}$, and $x_j\in \bar{C}$, such that $\sum_{i=1}^{n_\text{g}} \nabla g_i(x_j)\lambda^j_i \rightarrow 0$. Let the limit point of $x_j$ be denoted by $x$ and the limit point of $\lambda^j$ by $\lambda$. As a consequence of constraint qualification, there exists a vector $w\in \mathbb{R}^n$ such that $w\T \nabla g_i(x) \lambda_i > 0$ for all $i\in I_x$. We note that $g_i(x)>0$ implies $g_i(x_j)>0$ for large enough $j$. This shows that $I_{x_j} \subset I_x$ and therefore
\begin{equation*}
\sum_{i\in I_{x_j}} w\T \nabla g_i(x) \lambda_i > c_1
\end{equation*}
for all large enough $j$ and a small enough $c_1>0$. However, by continuity of $\nabla g$ we have
\begin{multline*}
\sum_{i=1}^{n_\text{g}} w\T \nabla g_i(x_j) \lambda_i^j = \sum_{i \in I_{x_j}} w\T \nabla g_i(x) \lambda_i 
+ \sum_{i\in I_{x_j}} w\T \nabla g_i(x) (\lambda_i^j-\lambda_i) \\
+ \sum_{i\in I_{x_j}} w\T (\nabla g_i(x_j)-\nabla g_i(x))\lambda_i^j > c_1/2,
\end{multline*}
for all large enough $j$, which is a contradiction.
\end{proof}

\begin{lemma}\label{Lemma:d3}
Let the assumptions of Thm.~\ref{Thm:ConvDT} be satisfied. If $x_{k(j)} \rightarrow \bar{x} \in C$, $u_{k(j)}\rightarrow 0$, and $R_{k(j)}/T_{k(j)} \rightarrow \bar{R}$ for a subsequence $k(j), j=1,2,\dots$, then $\bar{x}$ and $\bar{R}$ satisfy
\begin{equation*}
-\nabla f(\bar{x})+\bar{R}=0, \quad -\bar{R}\in N_{V_\alpha(\bar{x})}(0).
\end{equation*}
\end{lemma}
\begin{proof}
We conclude from the continuity of $g_i$ that $g_i(\bar{x})>0$ implies that $g_i(x_{k(j)})>0$ for all $j$ large enough and all $i\in \{1,2,\dots,n_\text{g}\}$. This means that $I_{x_{k(j)}}\subset I_{\bar{x}}$ for all $j$ large enough. We define a slightly modified version of $V_\alpha(x)$ as follows
\begin{equation*}
\tilde{V}_\alpha(x):=\{ u\in \mathbb{R}^n ~|~\nabla g_i(x)\T u \geq -\alpha g_i(x), \quad \forall i\in I_{\bar{x}} \},
\end{equation*}
which ensures that $V_\alpha(x_{k(j)}) \supset \tilde{V}_\alpha(x_{k(j)})$ for large $j$. Hence, the corresponding normal cones satisfy
\begin{equation*}
N_{V_\alpha(x_{k(j)})}(u) \subset N_{\tilde{V}_\alpha(x_{k(j)})}(u),
\end{equation*}
for all $u\in \tilde{V}_\alpha(x_{k(j)})$, which implies, by the update rule of algorithm \eqref{eq:dis1}, 
\begin{equation*}
-\frac{R_{k(j)}}{T_{k(j)}} \in N_{V_\alpha(x_{k(j)})}(u_{k(j)+1}) \subset N_{\tilde{V}_\alpha(x_{k(j)})}(u_{k(j)+1}),
\end{equation*}
where we have used the fact that the normal cone is a cone. We now show that this implies $-\bar{R} \in N_{\tilde{V}_\alpha(\bar{x})}(0)$ and argue by contradiction. This means that there exists a $\hat{u} \in \tilde{V}_\alpha(\bar{x})$ such that $-\bar{R}\T \hat{u} > c_{1}$ for a small $c_{1}>0$. From constraint qualification, we infer that there exists a $w\in \mathbb{R}^n$ and an $\varepsilon>0$ such that $\hat{u} +\varepsilon w \in \tilde{V}_\alpha (x_{k(j)})$ for all $j$ sufficiently large with $-\bar{R}\T (\hat{u} + \varepsilon w) > c_1/2$. However, this leads to a contradiction, since
\begin{multline*}
0\geq -\frac{R_{k(j)}}{T_{k(j)}}\T (\hat{u} + \varepsilon w - u_{k(j)+1})
= \\
\underbrace{-\bar{R} \T (\hat{u}+\varepsilon w)}_{> c_1/2 > 0} + \underbrace{\bar{R}\T u_{k(j)+1} - \left(\frac{R_{k(j)}}{T_{k(j)}} - \bar{R} \right)\T (\hat{u} + \varepsilon w - u_{k(j)+1})}_{\rightarrow 0}.
\end{multline*}
This shows that $-\bar{R} \in N_{\tilde{V}_\alpha(\bar{x})}(0)$ and the desired result follows from the fact that $\tilde{V}_\alpha(\bar{x})=V_\alpha(\bar{x})$.
\end{proof}

\begin{lemma}\label{Lem:Lyap}
Let the assumptions of Thm.~\ref{Thm:ConvDT} be satisfied. Then, there exists a function \hchange{$V_k: \mathbb{R}^n \times \mathbb{R}^n \rightarrow \mathbb{R}$}, which is bounded below (uniformly in $k$), such that 
\begin{multline*}
V_{k+1}(x_{k+1},u_{k+1}) - V_k(x_k,u_k) \leq -c_{\text{V}1} T_k |u_k|^2 - c_{\text{V}2} |R_k|^2   + c_{\text{V}3} T_k^2  - \alpha \sum_{i\in I_{x_k}}\lambda_k^i g_i(x_k),
\end{multline*}
for all $k$ large enough, where $c_{\text{V}1}, c_{\text{V}2}, c_{\text{V}3}>0$ are constant.
\end{lemma}
\begin{proof}
From Lemma~\ref{Lemma:constr} we infer that $x_k$ is bounded and therefore contained in a compact set, which we denote by $\bar{C}$. Without loss of generality we assume that $f$ is analytic on $\bar{C}$ (note that by the Stone-Weierstrass theorem we can find a polynomial that approximates $f$ and $\nabla f$ arbitrarily closely, see \cite{ourWork2} for details). We can now invoke \citet[Proposition~9 of ][]{ourWork2}, which constructs the function $F_k(x,u)$ (continuous in $x$ and $u$) such that
\begin{equation}
|\tilde{H}_{k}(x_{k+1},u_{k+1})-\tilde{H}_k(\bar{x}_k,\bar{u}_k)| \leq c_\text{E} T_k^4 \label{eq:prooftmp}
\end{equation}
for $k$ sufficiently large, where 
\begin{equation*}
\tilde{H}_k(x,u):=\frac{1}{2} |u|^2 + f(x) - \frac{T_k}{2} \nabla f(x)\T u + T_k^2 F_k(x,u),
\end{equation*}
$c_\text{E}$ is constant, and $\bar{x}_k, \bar{u}_k$ are defined in \eqref{eq:decomp}. The function $F_k(x,u)$ is guaranteed to be bounded for $x\in \bar{C}$, $|u|\leq c_\text{u}$ (uniformly in $k$). We note that \eqref{eq:prooftmp} implies that $\tilde{H}$ is almost conserved when applying the symplectic step in \eqref{eq:decomp} (in fact, the right hand side can be replaced with $(T_k)^q$ for an arbitrarily large $q>0$). The function $\tilde{H}$ is therefore often referred to as a modified Hamiltonian and arises from the fact that the second step in \eqref{eq:decomp} is a symplectic map.

\hchange{We claim that the function
\begin{equation*}
V_k(x,u):= \tilde{H}_k(x,u) + \frac{T_k d_k}{2} \nabla f(x)\T u
\end{equation*}
satisfies the desired properties, where $d_k:=1-T_k>0$ (for sufficiently large $k$). The choice $d_k=1-T_k$ is motivated by the fact that, as a result, $V_{k+1}(x,u)-V_k(x,u)$ is of order $T_k^2$ for $x_k \in \bar{C}$ and $|u_k|\leq c_\text{u}$. We start by considering step 1 in \eqref{eq:decomp}, where we note that
\begin{align*}
\frac{1}{2} |\bar{u}_k|^2 &- \frac{1}{2} |u_k|^2 = \frac{1}{2} (\bar{u}_k-u_k)\T (\bar{u}_k+u_k)\\
&=- T_k f_\text{d}\T u_k + u_k\T R_k + \frac{1}{2} T_k^2 |f_\text{d}|^2 - T_k f_\text{d}\T R_k + \frac{1}{2} |R_k|^2\\
&=- T_k f_\text{d}\T u_k +  u_{k+1}\T R_{k} + \frac{1}{2} T_k^2 |f_\text{d}|^2 - \frac{1}{2} |R_k|^2 + T_k R_k\T \nabla f(x_k),
\end{align*}
where we replaced $u_{k}$ by $u_{k+1}$ in the last step and omitted the arguments of $f_\text{d}(x_k,u_k)$. In addition we can bound $f_\text{d}\T u_k$ and $|f_\text{d}|^2$ as follows:
\begin{align*}
-f_\text{d}\T u_k \leq - \bar{\delta} |u_k|^2, \quad |f_\text{d}|^2 \leq (2\delta + \beta)^2 |u_k|^2,
\end{align*}
where $\bar{\delta}=2\delta$ in the convex case and $\bar{\delta}=2\delta-\beta$ in the nonconvex case. As a result, we conclude that for large enough $k$, 
\begin{equation*}
    \frac{1}{2} |\bar{u}_k|^2 - \frac{1}{2} |u_k|^2 \leq  -  \frac{T_k \bar{\delta}}{2} |u_k|^2+  u_{k+1}\T R_{k} - \frac{1}{2} |R_k|^2 + T_k R_k\T \nabla f(x_k).
\end{equation*}}

\hchange{We further have
\begin{align*}
-\frac{T_k}{2} (1-d_k) \nabla f(x_k)\T (\bar{u}_k-u_k) &= -\frac{T_k}{2} (1-d_k) \nabla f(x_k)\T (-T_k f_\text{d} +R_k),
\end{align*}
which means that over step 1 in \eqref{eq:decomp} the function $V_k$ changes by
\begin{align*}
V_{k}(\bar{x}_k,\bar{u}_k)-V_k(x_k,u_k) &\leq -  \frac{T_k \bar{\delta}}{2} |u_k|^2 + u_{k+1}\T R_{k} + \frac{1}{2} T_k R_k\T \nabla f(x_k) (1+d_k) \\
&- \frac{1}{2} |R_k|^2 + \frac{1}{2} T_k^2 (1-d_k) \nabla f(x_k)\T f_\text{d} \\
&+ T_k^2 \nabla_u F_k(x_k,u_k)\T (-T_k f_\text{d} + R_k) 
+ T_k^2 c_\text{F} |-T_k f_\text{d} + R_k |^2,
\end{align*}
where $c_\text{F}>0$ is a bound on the smoothness constant of $F_k$. We now proceed to the second step of \eqref{eq:decomp}, where we exploit the fact that $\tilde{H}_k$ is invariant up to terms of order $T_k^4$. We are therefore left with analyzing the term $T_k d_k \nabla f(x)\T u$, which gives
\begin{align*}
V_k(x_{k+1},u_{k+1})\!-\!V_k(\bar{x}_k,\bar{u}_k) &\!\leq\!\frac{T_k d_k}{2} (\nabla f(x_{k+1})\!-\!\nabla f(x_k))\T \!u_{k+1} \!-\! \frac{T_k^2 d_k}{2} |\nabla f(x_k)|^2\!+\! c_\text{E} T_k^4\\
&\leq \frac{T_k^2 d_k}{2} |u_{k+1}|^2 - \frac{T_k^2 d_k}{2} |\nabla f(x_k)|^2+ c_\text{E} T_k^4,
\end{align*}
where we used the fact that $\nabla f$ is 1-smooth in the second step. We therefore obtain 
\begin{multline*}
V_{k}(x_{k+1},u_{k+1})-V_k(x_k,u_k) \leq -  \frac{T_k \bar{\delta}}{2} |u_k|^2 + u_{k+1}\T R_k - \frac{d_k}{2} |R_k-T_k\nabla f(x_k)|^2 \\
+ \frac{T_k^2 }{2} \nabla f(x_k)\T R_k + \frac{T_{k}^2 d_k}{2} |u_{k+1}|^2 + \frac{T_k^3}{2} \nabla f(x_k)\T f_\text{d}\\ + T_k^2 \nabla_u F_k(x_k,u_k)\T (-T_k f_\text{d} + R_k) 
+ T_k^2 c_\text{F} |-T_k f_\text{d} + R_k |^2 + c_{\text{E}} T_k^4,
\end{multline*}
where we have exploited the fact that $d_k=1- T_k$.
We further note that $|\nabla_u F_k(x_k,u_k)|\leq c_{\text{F}2} (|u_k|+|\nabla f(x_k)|)$ and conclude from $x_k\in \bar{C}$ and $|u_k|\leq c_\text{u}$ that all variables (i.e., $\nabla f(x_k)$, $R_k$, $f_\text{d}$, $u_k$, $u_{k+1}$, and $\nabla_u F_k(x_k,u_k)$) are bounded.  Hence, there exists a constant $c_V$ such that
\begin{equation*}
V_{k}(x_{k+1},u_{k+1})-V_k(x_k,u_k) \leq -  \frac{T_k \bar{\delta}}{2} |u_k|^2 + u_{k+1}\T R_k - \frac{d_k}{2} |R_k-T_k\nabla f(x_k)|^2 + c_V T_k^2,
\end{equation*}
for all $k$ large enough. In addition, the change of $V_{k+1}(x_{k+1},u_{k+1})-V_k(x_{k+1},u_{k+1})$ is of the order $\mathcal{O}(T_k^2)$. We further infer from the update equation for $u_{k+1}$ that 
\begin{equation*}
    u_{k+1}\T R_k = - \alpha \sum_{i\in I_{x_k}} \lambda_k^i g_i(x_k),
\end{equation*} which yields the desired result.}
\end{proof}


%% file: appendixProofADT.tex

%% file: sn-article.bbl

\begin{thebibliography}{56}
\ifx \bisbn   \undefined \def \bisbn  #1{ISBN #1}\fi
\ifx \binits  \undefined \def \binits#1{#1}\fi
\ifx \bauthor  \undefined \def \bauthor#1{#1}\fi
\ifx \batitle  \undefined \def \batitle#1{#1}\fi
\ifx \bjtitle  \undefined \def \bjtitle#1{#1}\fi
\ifx \bvolume  \undefined \def \bvolume#1{\textbf{#1}}\fi
\ifx \byear  \undefined \def \byear#1{#1}\fi
\ifx \bissue  \undefined \def \bissue#1{#1}\fi
\ifx \bfpage  \undefined \def \bfpage#1{#1}\fi
\ifx \blpage  \undefined \def \blpage #1{#1}\fi
\ifx \burl  \undefined \def \burl#1{\textsf{#1}}\fi
\ifx \doiurl  \undefined \def \doiurl#1{\url{https://doi.org/#1}}\fi
\ifx \betal  \undefined \def \betal{\textit{et al.}}\fi
\ifx \binstitute  \undefined \def \binstitute#1{#1}\fi
\ifx \binstitutionaled  \undefined \def \binstitutionaled#1{#1}\fi
\ifx \bctitle  \undefined \def \bctitle#1{#1}\fi
\ifx \beditor  \undefined \def \beditor#1{#1}\fi
\ifx \bpublisher  \undefined \def \bpublisher#1{#1}\fi
\ifx \bbtitle  \undefined \def \bbtitle#1{#1}\fi
\ifx \bedition  \undefined \def \bedition#1{#1}\fi
\ifx \bseriesno  \undefined \def \bseriesno#1{#1}\fi
\ifx \blocation  \undefined \def \blocation#1{#1}\fi
\ifx \bsertitle  \undefined \def \bsertitle#1{#1}\fi
\ifx \bsnm \undefined \def \bsnm#1{#1}\fi
\ifx \bsuffix \undefined \def \bsuffix#1{#1}\fi
\ifx \bparticle \undefined \def \bparticle#1{#1}\fi
\ifx \barticle \undefined \def \barticle#1{#1}\fi
\bibcommenthead
\ifx \bconfdate \undefined \def \bconfdate #1{#1}\fi
\ifx \botherref \undefined \def \botherref #1{#1}\fi
\ifx \url \undefined \def \url#1{\textsf{#1}}\fi
\ifx \bchapter \undefined \def \bchapter#1{#1}\fi
\ifx \bbook \undefined \def \bbook#1{#1}\fi
\ifx \bcomment \undefined \def \bcomment#1{#1}\fi
\ifx \oauthor \undefined \def \oauthor#1{#1}\fi
\ifx \citeauthoryear \undefined \def \citeauthoryear#1{#1}\fi
\ifx \endbibitem  \undefined \def \endbibitem {}\fi
\ifx \bconflocation  \undefined \def \bconflocation#1{#1}\fi
\ifx \arxivurl  \undefined \def \arxivurl#1{\textsf{#1}}\fi
\csname PreBibitemsHook\endcsname

\bibitem[\protect\citeauthoryear{Muehlebach and Jordan}{2022}]{ownWorkC}
\begin{barticle}
\bauthor{\bsnm{Muehlebach}, \binits{M.}},
\bauthor{\bsnm{Jordan}, \binits{M.I.}}:
\batitle{On constraints in first-order optimization: A view from non-smooth
  dynamical systems}.
\bjtitle{Journal of Machine Learning Research}
\bvolume{23}(\bissue{256}),
\bfpage{1}--\blpage{47}
(\byear{2022})
\end{barticle}
\endbibitem

\bibitem[\protect\citeauthoryear{Su et~al.}{2016}]{SuAcc}
\begin{barticle}
\bauthor{\bsnm{Su}, \binits{W.}},
\bauthor{\bsnm{Boyd}, \binits{S.}},
\bauthor{\bsnm{Cand\`{e}s}, \binits{E.J.}}:
\batitle{A differential equation for modeling {N}esterov's accelerated gradient
  method: Theory and insights}.
\bjtitle{Journal of Machine Learning Research}
\bvolume{17}(\bissue{153}),
\bfpage{1}--\blpage{43}
(\byear{2016})
\end{barticle}
\endbibitem

\bibitem[\protect\citeauthoryear{Wibisono et~al.}{2016}]{WibisonoVariational}
\begin{barticle}
\bauthor{\bsnm{Wibisono}, \binits{A.}},
\bauthor{\bsnm{Wilson}, \binits{A.C.}},
\bauthor{\bsnm{Jordan}, \binits{M.I.}}:
\batitle{A variational perspective on accelerated methods in optimization}.
\bjtitle{Proceedings of the National Academy of Sciences}
\bvolume{113}(\bissue{47}),
\bfpage{7351}--\blpage{7358}
(\byear{2016})
\end{barticle}
\endbibitem

\bibitem[\protect\citeauthoryear{Diakonikolas and Jordan}{2021}]{Diakonikolas2}
\begin{barticle}
\bauthor{\bsnm{Diakonikolas}, \binits{J.}},
\bauthor{\bsnm{Jordan}, \binits{M.I.}}:
\batitle{Generalized momentum-based methods: {A} {H}amiltonian perspective}.
\bjtitle{SIAM Journal on Optimization}
\bvolume{31}(\bissue{1}),
\bfpage{915}--\blpage{944}
(\byear{2021})
\end{barticle}
\endbibitem

\bibitem[\protect\citeauthoryear{Krichene et~al.}{2015}]{KricheneAcc}
\begin{botherref}
\oauthor{\bsnm{Krichene}, \binits{W.}},
\oauthor{\bsnm{Bayen}, \binits{A.M.}},
\oauthor{\bsnm{Bartlett}, \binits{P.L.}}:
Accelerated mirror descent in continuous and discrete time.
Advances in Neural Information Processing Systems 28,
2845--2853
(2015)
\end{botherref}
\endbibitem

\bibitem[\protect\citeauthoryear{Fran{\c{c}}a et~al.}{2020}]{Gui}
\begin{barticle}
\bauthor{\bsnm{Fran{\c{c}}a}, \binits{G.}},
\bauthor{\bsnm{Sulam}, \binits{J.}},
\bauthor{\bsnm{Robinson}, \binits{D.P.}},
\bauthor{\bsnm{Vidal}, \binits{R.}}:
\batitle{Conformal symplectic and relativistic optimization}.
\bjtitle{Journal of Statistical Mechanics: Theory and Experiment}
\bvolume{2020}(\bissue{12}),
\bfpage{1}--\blpage{30}
(\byear{2020})
\end{barticle}
\endbibitem

\bibitem[\protect\citeauthoryear{Betancourt et~al.}{2018}]{Betancourt}
\begin{botherref}
\oauthor{\bsnm{Betancourt}, \binits{M.}},
\oauthor{\bsnm{Jordan}, \binits{M.I.}},
\oauthor{\bsnm{Wilson}, \binits{A.C.}}:
On symplectic optimization.
arXiv:1802.03653v2,
1--20
(2018)
\end{botherref}
\endbibitem

\bibitem[\protect\citeauthoryear{Muehlebach and Jordan}{2019}]{ourWork}
\begin{barticle}
\bauthor{\bsnm{Muehlebach}, \binits{M.}},
\bauthor{\bsnm{Jordan}, \binits{M.I.}}:
\batitle{A dynamical systems perspective on {N}esterov acceleration}.
\bjtitle{Proceedings of Machine Learning Research}
\bvolume{97},
\bfpage{4656}--\blpage{4662}
(\byear{2019})
\end{barticle}
\endbibitem

\bibitem[\protect\citeauthoryear{Muehlebach and Jordan}{2020}]{ourWork3}
\begin{barticle}
\bauthor{\bsnm{Muehlebach}, \binits{M.}},
\bauthor{\bsnm{Jordan}, \binits{M.I.}}:
\batitle{Continuous-time lower bounds for gradient-based algorithms}.
\bjtitle{Proceedings of Machine Learning Research}
\bvolume{119},
\bfpage{7088}--\blpage{7096}
(\byear{2020})
\end{barticle}
\endbibitem

\bibitem[\protect\citeauthoryear{Muehlebach and Jordan}{2021}]{ourWork2}
\begin{barticle}
\bauthor{\bsnm{Muehlebach}, \binits{M.}},
\bauthor{\bsnm{Jordan}, \binits{M.I.}}:
\batitle{Optimization with momentum: Dynamical, control-theoretic, and
  symplectic perspectives}.
\bjtitle{Journal of Machine Learning Research}
\bvolume{22}(\bissue{73}),
\bfpage{1}--\blpage{50}
(\byear{2021})
\end{barticle}
\endbibitem

\bibitem[\protect\citeauthoryear{Attouch et~al.}{2018}]{Attouch}
\begin{barticle}
\bauthor{\bsnm{Attouch}, \binits{H.}},
\bauthor{\bsnm{Chbani}, \binits{Z.}},
\bauthor{\bsnm{Peypouquet}, \binits{J.}},
\bauthor{\bsnm{Redont}, \binits{P.}}:
\batitle{Fast convergence of inertial dynamics and algorithms with asymptotic
  vanishing viscosity}.
\bjtitle{Mathematical Programming, Series B}
\bvolume{168}(\bissue{1-2}),
\bfpage{123}--\blpage{175}
(\byear{2018})
\end{barticle}
\endbibitem

\bibitem[\protect\citeauthoryear{Alvarez and Attouch}{2001}]{Attouch1}
\begin{barticle}
\bauthor{\bsnm{Alvarez}, \binits{F.}},
\bauthor{\bsnm{Attouch}, \binits{H.}}:
\batitle{An inertial proximal method for maximal monotone operators via
  discretization of a nonlinear oscillator with damping}.
\bjtitle{Set-Valued and Variational Analysis}
\bvolume{9},
\bfpage{3}--\blpage{11}
(\byear{2001})
\end{barticle}
\endbibitem

\bibitem[\protect\citeauthoryear{Attouch and \'{E}mile
  Maing\'{e}}{2011}]{Attouch2}
\begin{barticle}
\bauthor{\bsnm{Attouch}, \binits{H.}},
\bauthor{\bsnm{Maing\'{e}}, \binits{P.-E.}}:
\batitle{Asymptotic behavior of second-order dissipative evolution equations
  combining potential with non-potential effects}.
\bjtitle{ESAIM: Control, Optimisation and Calculus of Variations}
\bvolume{17}(\bissue{3}),
\bfpage{836}--\blpage{857}
(\byear{2011})
\end{barticle}
\endbibitem

\bibitem[\protect\citeauthoryear{Er et~al.}{2024}]{Dilsad}
\begin{botherref}
\oauthor{\bsnm{Er}, \binits{G.D.}},
\oauthor{\bsnm{Trimpe}, \binits{S.}},
\oauthor{\bsnm{Muehlebach}, \binits{M.}}:
Distributed event-based learning via {ADMM}.
arXiv:2405.10618,
1--29
(2024)
\end{botherref}
\endbibitem

\bibitem[\protect\citeauthoryear{Guanchun and Muehlebach}{2023}]{Guanchun}
\begin{barticle}
\bauthor{\bsnm{Guanchun}, \binits{T.}},
\bauthor{\bsnm{Muehlebach}, \binits{M.}}:
\batitle{A dynamical systems perspective on discrete optimization}.
\bjtitle{Proceedings of Machine Learning Research}
\bvolume{211},
\bfpage{1373}--\blpage{1386}
(\byear{2023})
\end{barticle}
\endbibitem

\bibitem[\protect\citeauthoryear{Moreau}{1988}]{moreau}
\begin{bchapter}
\bauthor{\bsnm{Moreau}, \binits{J.J.}}:
\bctitle{Unilateral contact and dry friction in finite freedom dynamics}.
In: \bbtitle{Nonsmooth Mechanics and Applications},
pp. \bfpage{1}--\blpage{82}.
\bpublisher{Springer},
\blocation{Wien}
(\byear{1988})
\end{bchapter}
\endbibitem

\bibitem[\protect\citeauthoryear{Glocker}{2001}]{Glocker}
\begin{bbook}
\bauthor{\bsnm{Glocker}, \binits{C.}}:
\bbtitle{Set-Valued Force Laws}.
\bpublisher{Springer},
\blocation{Berlin}
(\byear{2001})
\end{bbook}
\endbibitem

\bibitem[\protect\citeauthoryear{Studer}{2009}]{Studer}
\begin{bbook}
\bauthor{\bsnm{Studer}, \binits{C.}}:
\bbtitle{Numerics of Unilateral Contacts and Friction}.
\bpublisher{Springer},
\blocation{Berlin}
(\byear{2009})
\end{bbook}
\endbibitem

\bibitem[\protect\citeauthoryear{Wang and Liu}{2006}]{InexactPG}
\begin{barticle}
\bauthor{\bsnm{Wang}, \binits{C.}},
\bauthor{\bsnm{Liu}, \binits{Q.}}:
\batitle{Convergence properties of inexact projected gradient methods}.
\bjtitle{Optimization}
\bvolume{55}(\bissue{3}),
\bfpage{301}--\blpage{310}
(\byear{2006})
\end{barticle}
\endbibitem

\bibitem[\protect\citeauthoryear{Birgin et~al.}{2003}]{InexactPG2}
\begin{barticle}
\bauthor{\bsnm{Birgin}, \binits{E.G.}},
\bauthor{\bsnm{Mart\'{i}nez}, \binits{J.M.}},
\bauthor{\bsnm{Raydan}, \binits{M.}}:
\batitle{Inexact spectral projected gradient methods on convex sets}.
\bjtitle{IMA Journal of Numerical Analysis}
\bvolume{23},
\bfpage{539}--\blpage{559}
(\byear{2003})
\end{barticle}
\endbibitem

\bibitem[\protect\citeauthoryear{Beck and Teboulle}{2011}]{SignalProcessing}
\begin{bchapter}
\bauthor{\bsnm{Beck}, \binits{A.}},
\bauthor{\bsnm{Teboulle}, \binits{M.}}:
\bctitle{Gradient-based algorithms with applications to signal-recovery
  problems}.
In: \bbtitle{Convex Optimization in Signal Processing and Communications},
pp. \bfpage{42}--\blpage{88}.
\bpublisher{Cambridge University Press},
\blocation{Cambridge}
(\byear{2011})
\end{bchapter}
\endbibitem

\bibitem[\protect\citeauthoryear{Bloom et~al.}{2016}]{SVMwPG}
\begin{barticle}
\bauthor{\bsnm{Bloom}, \binits{V.}},
\bauthor{\bsnm{Griva}, \binits{I.}},
\bauthor{\bsnm{Quijada}, \binits{F.}}:
\batitle{Fast projected gradient method for support vector machines}.
\bjtitle{Optimization and Engineering}
\bvolume{17}(\bissue{4}),
\bfpage{651}--\blpage{662}
(\byear{2016})
\end{barticle}
\endbibitem

\bibitem[\protect\citeauthoryear{Jaggi}{2013}]{Jaeggi}
\begin{barticle}
\bauthor{\bsnm{Jaggi}, \binits{M.}}:
\batitle{Revisiting {F}rank-{W}olfe: {P}rojection-free sparse convex
  optimization}.
\bjtitle{Proceedings of Machine Learning Research}
\bvolume{28}(\bissue{1}),
\bfpage{427}--\blpage{435}
(\byear{2013})
\end{barticle}
\endbibitem

\bibitem[\protect\citeauthoryear{Clarkson}{2010}]{Clarkson}
\begin{barticle}
\bauthor{\bsnm{Clarkson}, \binits{K.L.}}:
\batitle{Coresets, sparse greedy approximation, and the {F}rank-{W}olfe
  algorithm}.
\bjtitle{ACM Transactions on Algorithms}
\bvolume{6}(\bissue{4}),
\bfpage{1}--\blpage{30}
(\byear{2010})
\end{barticle}
\endbibitem

\bibitem[\protect\citeauthoryear{Hazan and Kale}{2012}]{PfOnlineLearning}
\begin{botherref}
\oauthor{\bsnm{Hazan}, \binits{E.}},
\oauthor{\bsnm{Kale}, \binits{S.}}:
Projection-free online learning.
Proceeding of the International Conference on Machine Learning,
1--8
(2012)
\end{botherref}
\endbibitem

\bibitem[\protect\citeauthoryear{Zhang et~al.}{2020}]{OneSampleStochasticFW}
\begin{barticle}
\bauthor{\bsnm{Zhang}, \binits{M.}},
\bauthor{\bsnm{Shen}, \binits{Z.}},
\bauthor{\bsnm{Mokhtari}, \binits{A.}},
\bauthor{\bsnm{Hassani}, \binits{H.}},
\bauthor{\bsnm{Karbasi}, \binits{A.}}:
\batitle{One sample stochastic {F}rank-{W}olfe}.
\bjtitle{Proceedings of Machine Learning Research}
\bvolume{108},
\bfpage{4012}--\blpage{4023}
(\byear{2020})
\end{barticle}
\endbibitem

\bibitem[\protect\citeauthoryear{Garber and Hazan}{2015}]{FasterRates}
\begin{barticle}
\bauthor{\bsnm{Garber}, \binits{D.}},
\bauthor{\bsnm{Hazan}, \binits{E.}}:
\batitle{Faster rates for the {F}rank-{W}olfe method over strongly-convex
  sets}.
\bjtitle{Proceedings of Machine Learning Research}
\bvolume{37},
\bfpage{541}--\blpage{549}
(\byear{2015})
\end{barticle}
\endbibitem

\bibitem[\protect\citeauthoryear{Combettes and
  Pokutta}{2020}]{BoostingFrankWolfe}
\begin{barticle}
\bauthor{\bsnm{Combettes}, \binits{C.W.}},
\bauthor{\bsnm{Pokutta}, \binits{S.}}:
\batitle{Boosting {F}rank-{W}olfe by chasing gradients}.
\bjtitle{Proceedings of Machine Learning Research}
\bvolume{119},
\bfpage{2111}--\blpage{2121}
(\byear{2020})
\end{barticle}
\endbibitem

\bibitem[\protect\citeauthoryear{Nemirovski and
  Yudin}{1983}]{ProblemComplexity}
\begin{bbook}
\bauthor{\bsnm{Nemirovski}, \binits{A.S.}},
\bauthor{\bsnm{Yudin}, \binits{D.B.}}:
\bbtitle{Problem Complexity and Method Efficiency in Optimization}.
\bpublisher{John Wiley \& Sons},
\blocation{Chichester}
(\byear{1983})
\end{bbook}
\endbibitem

\bibitem[\protect\citeauthoryear{Beck and Teboulle}{2003}]{BeckMirror}
\begin{barticle}
\bauthor{\bsnm{Beck}, \binits{A.}},
\bauthor{\bsnm{Teboulle}, \binits{M.}}:
\batitle{Mirror descent and nonlinear projected subgradient methods for convex
  optimization}.
\bjtitle{Operations Research Letters}
\bvolume{31}(\bissue{3}),
\bfpage{167}--\blpage{175}
(\byear{2003})
\end{barticle}
\endbibitem

\bibitem[\protect\citeauthoryear{Bubeck and Cesa-Bianchi}{2012}]{BubeckBandits}
\begin{barticle}
\bauthor{\bsnm{Bubeck}, \binits{S.}},
\bauthor{\bsnm{Cesa-Bianchi}, \binits{N.}}:
\batitle{Regret analysis of stochastic and nonstochastic multi-armed bandit
  problems}.
\bjtitle{Foundation and Trends in Machine Learning}
\bvolume{5}(\bissue{1}),
\bfpage{1}--\blpage{122}
(\byear{2012})
\end{barticle}
\endbibitem

\bibitem[\protect\citeauthoryear{Attouch}{2021}]{Attouch3}
\begin{barticle}
\bauthor{\bsnm{Attouch}, \binits{H.}}:
\batitle{Fast inertial proximal {ADMM} algorithms for convex structured
  optimization with linear constraint}.
\bjtitle{Minimax Theory and its Applications}
\bvolume{6}(\bissue{1}),
\bfpage{1}--\blpage{24}
(\byear{2021})
\end{barticle}
\endbibitem

\bibitem[\protect\citeauthoryear{Attouch et~al.}{2022}]{Attouch4}
\begin{barticle}
\bauthor{\bsnm{Attouch}, \binits{H.}},
\bauthor{\bsnm{Chbani}, \binits{Z.}},
\bauthor{\bsnm{Fadili}, \binits{J.}},
\bauthor{\bsnm{Riahi}, \binits{H.}}:
\batitle{Fast convergence of dynamical {ADMM} algorithms for convex structured
  optimization with linear constraint}.
\bjtitle{Journal of Optimization Theory and Applications}
\bvolume{193}(\bissue{1-3}),
\bfpage{704}--\blpage{736}
(\byear{2022})
\end{barticle}
\endbibitem

\bibitem[\protect\citeauthoryear{Kolev et~al.}{2023}]{Pavel}
\begin{barticle}
\bauthor{\bsnm{Kolev}, \binits{P.}},
\bauthor{\bsnm{Martius}, \binits{G.}},
\bauthor{\bsnm{Muehlebach}, \binits{M.}}:
\batitle{Online learning under adversarial nonlinear constraints}.
\bjtitle{Advances in Neural Information Processing Systems}
\bvolume{36},
\bfpage{1}--\blpage{26}
(\byear{2023})
\end{barticle}
\endbibitem

\bibitem[\protect\citeauthoryear{Schechtman et~al.}{2023}]{COLT}
\begin{barticle}
\bauthor{\bsnm{Schechtman}, \binits{S.}},
\bauthor{\bsnm{Tiapkin}, \binits{D.}},
\bauthor{\bsnm{Muehlebach}, \binits{M.}},
\bauthor{\bsnm{Moulines}, \binits{{\'{E}}.}}:
\batitle{Orthogonal directions constrained gradient method: from non-linear
  equality constraints to {S}tiefel manifold}.
\bjtitle{Proceedings of Machine Learning Research}
\bvolume{195},
\bfpage{1228}--\blpage{1258}
(\byear{2023})
\end{barticle}
\endbibitem

\bibitem[\protect\citeauthoryear{Gon{\c{c}}alves et~al.}{2017}]{Liberti}
\begin{barticle}
\bauthor{\bsnm{Gon{\c{c}}alves}, \binits{D.S.}},
\bauthor{\bsnm{Mucherino}, \binits{A.}},
\bauthor{\bsnm{Lavor}, \binits{C.}},
\bauthor{\bsnm{Liberti}, \binits{L.}}:
\batitle{Recent advances on the interval distance geometry problem}.
\bjtitle{Journal of Global Optimization}
\bvolume{69}(\bissue{3}),
\bfpage{525}--\blpage{545}
(\byear{2017})
\end{barticle}
\endbibitem

\bibitem[\protect\citeauthoryear{Ibrahim et~al.}{2023}]{Abdul}
\begin{botherref}
\oauthor{\bsnm{Ibrahim}, \binits{A.A.}},
\oauthor{\bsnm{Muehlebach}, \binits{M.}},
\oauthor{\bsnm{Bacco}, \binits{C.D.}}:
Optimal transport with constraints: from mirror descent to classical mechanics.
arXiv:2309.04727,
1--14
(2023)
\end{botherref}
\endbibitem

\bibitem[\protect\citeauthoryear{Zhao et~al.}{2024}]{Wenshuai}
\begin{botherref}
\oauthor{\bsnm{Zhao}, \binits{W.}},
\oauthor{\bsnm{Zhao}, \binits{Y.}},
\oauthor{\bsnm{Pajarinen}, \binits{J.}},
\oauthor{\bsnm{Muehlebach}, \binits{M.}}:
Bi-level motion imitation for humanoid robots.
Proceedings of the Conference on Robot Learning,
1--19
(2024)
\end{botherref}
\endbibitem

\bibitem[\protect\citeauthoryear{Bolte et~al.}{2018}]{Bolte}
\begin{barticle}
\bauthor{\bsnm{Bolte}, \binits{J.}},
\bauthor{\bsnm{Hochart}, \binits{A.}},
\bauthor{\bsnm{Pauwels}, \binits{E.}}:
\batitle{Qualification conditions in semialgebraic programming}.
\bjtitle{SIAM Journal on Optimization}
\bvolume{28}(\bissue{2}),
\bfpage{1867}--\blpage{1891}
(\byear{2018})
\end{barticle}
\endbibitem

\bibitem[\protect\citeauthoryear{Bertsekas}{1977}]{BertsekasComp}
\begin{barticle}
\bauthor{\bsnm{Bertsekas}, \binits{D.P.}}:
\batitle{Approximation procedures based on the method of multipliers}.
\bjtitle{Journal of Optimization Theory and Applications}
\bvolume{23}(\bissue{4}),
\bfpage{487}--\blpage{510}
(\byear{1977})
\end{barticle}
\endbibitem

\bibitem[\protect\citeauthoryear{Drusvyatskiy and
  Paquette}{2019}]{PaquetteComp}
\begin{barticle}
\bauthor{\bsnm{Drusvyatskiy}, \binits{D.}},
\bauthor{\bsnm{Paquette}, \binits{C.}}:
\batitle{Efficiency of minimizing compositions of convex functions and smooth
  maps}.
\bjtitle{Mathematical Programming}
\bvolume{178},
\bfpage{503}--\blpage{558}
(\byear{2019})
\end{barticle}
\endbibitem

\bibitem[\protect\citeauthoryear{Fletcher}{1982}]{FletcherComp}
\begin{barticle}
\bauthor{\bsnm{Fletcher}, \binits{R.}}:
\batitle{A model algorithm for composite nondifferentiable optimization
  problems}.
\bjtitle{Mathematical Programming Studies}
\bvolume{17},
\bfpage{67}--\blpage{76}
(\byear{1982})
\end{barticle}
\endbibitem

\bibitem[\protect\citeauthoryear{Bolte et~al.}{2020}]{BolteComp}
\begin{botherref}
\oauthor{\bsnm{Bolte}, \binits{J.}},
\oauthor{\bsnm{Chen}, \binits{Z.}},
\oauthor{\bsnm{Pauwels}, \binits{E.}}:
The multiproximal linearization method for convex composite problems.
Mathematical Programming,
1--36
(2020)
\end{botherref}
\endbibitem

\bibitem[\protect\citeauthoryear{Doikov and Nesterov}{2022}]{NesterovComp}
\begin{barticle}
\bauthor{\bsnm{Doikov}, \binits{N.}},
\bauthor{\bsnm{Nesterov}, \binits{Y.}}:
\batitle{High-order optimization methods for fully composite problems}.
\bjtitle{SIAM Journal on Optimization}
\bvolume{32}(\bissue{3}),
\bfpage{402}--\blpage{2427}
(\byear{2022})
\end{barticle}
\endbibitem

\bibitem[\protect\citeauthoryear{Vladarean et~al.}{2023}]{NicolasComp}
\begin{barticle}
\bauthor{\bsnm{Vladarean}, \binits{M.-L.}},
\bauthor{\bsnm{Doikov}, \binits{N.}},
\bauthor{\bsnm{Jaggi}, \binits{M.}},
\bauthor{\bsnm{Flammarion}, \binits{N.}}:
\batitle{Linearization algorithms for fully composite optimization}.
\bjtitle{Proceedings of Machine Learning Research}
\bvolume{195},
\bfpage{3669}--\blpage{3695}
(\byear{2023})
\end{barticle}
\endbibitem

\bibitem[\protect\citeauthoryear{Polyak}{1964}]{PolyakHeavyBall}
\begin{barticle}
\bauthor{\bsnm{Polyak}, \binits{B.T.}}:
\batitle{Some methods of speeding up the convergence of iteration methods}.
\bjtitle{USSR Computational Mathematics and Mathematical Physics}
\bvolume{4}(\bissue{5}),
\bfpage{1}--\blpage{17}
(\byear{1964})
\end{barticle}
\endbibitem

\bibitem[\protect\citeauthoryear{Leine and van~de Wouw}{2008}]{Leine}
\begin{bbook}
\bauthor{\bsnm{Leine}, \binits{R.I.}},
\bauthor{\bsnm{Wouw}, \binits{N.}}:
\bbtitle{Stability and Convergence of Mechanical Systems with Unilateral
  Constraints}.
\bpublisher{Springer},
\blocation{Berlin}
(\byear{2008})
\end{bbook}
\endbibitem

\bibitem[\protect\citeauthoryear{Piazza et~al.}{2021}]{DiPiazza}
\begin{barticle}
\bauthor{\bsnm{Piazza}, \binits{L.D.}},
\bauthor{\bsnm{Marraffa}, \binits{V.}},
\bauthor{\bsnm{Satco}, \binits{B.}}:
\batitle{Measure differential inclusions: Existence results and minimum
  problems}.
\bjtitle{Set-Valued and Variational Analysis}
\bvolume{29}(\bissue{2}),
\bfpage{361}--\blpage{382}
(\byear{2021})
\end{barticle}
\endbibitem

\bibitem[\protect\citeauthoryear{Leine and van~de Wouw}{2008}]{Leine2}
\begin{barticle}
\bauthor{\bsnm{Leine}, \binits{R.I.}},
\bauthor{\bsnm{Wouw}, \binits{N.}}:
\batitle{Uniform convergence of monotone measure differential inclusions: With
  application to the control of mechanical systems with unilateral
  constraints}.
\bjtitle{International Journal of Bifurcation and Chaos}
\bvolume{18}(\bissue{5}),
\bfpage{1435}--\blpage{1457}
(\byear{2008})
\end{barticle}
\endbibitem

\bibitem[\protect\citeauthoryear{Teel}{1999}]{Teel}
\begin{barticle}
\bauthor{\bsnm{Teel}, \binits{A.R.}}:
\batitle{Asymptotic convergence from {$\mathcal{L}_p$} stability}.
\bjtitle{IEEE Transactions on Automatic Control}
\bvolume{44}(\bissue{11}),
\bfpage{2169}--\blpage{2170}
(\byear{1999})
\end{barticle}
\endbibitem

\bibitem[\protect\citeauthoryear{Rockafellar}{1970}]{RockafellarConvex}
\begin{bbook}
\bauthor{\bsnm{Rockafellar}, \binits{R.T.}}:
\bbtitle{Convex Analysis}.
\bpublisher{Princeton University Press},
\blocation{Princeton, New Jersey}
(\byear{1970})
\end{bbook}
\endbibitem

\bibitem[\protect\citeauthoryear{Polyak}{1987}]{Polyak}
\begin{bbook}
\bauthor{\bsnm{Polyak}, \binits{B.T.}}:
\bbtitle{Introduction to Optimization}.
\bpublisher{Optimization Software, Inc.},
\blocation{New York}
(\byear{1987})
\end{bbook}
\endbibitem

\bibitem[\protect\citeauthoryear{Dontchev and Rockafellar}{2014}]{Donchev}
\begin{bbook}
\bauthor{\bsnm{Dontchev}, \binits{A.L.}},
\bauthor{\bsnm{Rockafellar}, \binits{R.T.}}:
\bbtitle{Implicit Functions and Solution Mappings},
\bedition{2}nd edn.
\bpublisher{Springer},
\blocation{New York}
(\byear{2014})
\end{bbook}
\endbibitem

\bibitem[\protect\citeauthoryear{Nesterov}{2004}]{NesterovIntro}
\begin{bbook}
\bauthor{\bsnm{Nesterov}, \binits{Y.}}:
\bbtitle{Introductory Lectures on Convex Optimization - A Basic Course}.
\bpublisher{Springer},
\blocation{New York}
(\byear{2004})
\end{bbook}
\endbibitem

\bibitem[\protect\citeauthoryear{Hastie et~al.}{2009}]{ElementsOfSL}
\begin{bbook}
\bauthor{\bsnm{Hastie}, \binits{T.}},
\bauthor{\bsnm{Tibshirani}, \binits{R.}},
\bauthor{\bsnm{Friedman}, \binits{J.}}:
\bbtitle{The Elements of Statistical Learning},
\bedition{2}nd edn.
\bpublisher{Springer},
\blocation{New York}
(\byear{2009})
\end{bbook}
\endbibitem

\bibitem[\protect\citeauthoryear{Beck and Teboulle}{2009}]{FISTA}
\begin{barticle}
\bauthor{\bsnm{Beck}, \binits{A.}},
\bauthor{\bsnm{Teboulle}, \binits{M.}}:
\batitle{A fast iterative shrinkage-thresholding algorithm for linear inverse
  problems}.
\bjtitle{SIAM Journal on Imaging Sciences}
\bvolume{2}(\bissue{1}),
\bfpage{183}--\blpage{202}
(\byear{2009})
\end{barticle}
\endbibitem

\end{thebibliography}
